\documentclass[11pt,a4paper]{amsart}
\usepackage{ucs}
\usepackage[utf8x]{inputenc}
\usepackage{verbatim}
\usepackage{paralist}
\usepackage{leftidx}
\usepackage{dsfont}
\usepackage{amsthm}
\usepackage{amsmath}
\usepackage{amssymb}
\usepackage{amscd}
\usepackage{graphicx}
\usepackage{setspace}
\usepackage[all]{xy}
\usepackage{mathrsfs} 
\usepackage{hyperref}
\begin{document}
	\numberwithin{equation}{section}
	\newtheorem{Def}{Definition}[section]
	\newtheorem{Rem}[Def]{Remark}
	\newtheorem{Lem}[Def]{Lemma}
	\newtheorem{Que}[Def]{Question}
	\newtheorem{Cor}[Def]{Corollary}
	\newtheorem{Exam}[Def]{Example}
	\newtheorem{Thm}[Def]{Theorem}
	\newtheorem*{clm}{Claim}
	\newtheorem{Pro}[Def]{Proposition}
	\newcommand{\ch}{\mathrm{char}~}
	\newcommand{\spec}{\mathrm{Spec}~}
	\renewcommand{\mod}{\mathrm{mod}~}
	\renewcommand{\O}{\mathcal{O}}
	\renewcommand{\Im}{\mathrm{Im~}}
	\newcommand{\Ar}{\mathrm{Ar}}
	\newcommand{\Z}{\mathds{Z}}
	\newcommand{\Q}{\mathds{Q}}
	\newcommand{\C}{\mathds{C}}
	\newcommand{\R}{\mathds{R}}
	\newcommand{\N}{\mathds{N}}
	\newcommand{\Der}{\mathrm{Der}}
	\newcommand{\res}{\mathrm{res}}
	\newcommand{\Hom}{\mathrm{Hom}}
	\newcommand{\Pic}{\mathrm{Pic}}
	\newcommand{\Ext}{\mathrm{Ext}}
	\newcommand{\Jac}{\mathrm{Jac}}
	\newcommand{\id}{\mathrm{id}}
	\newcommand{\Gr}{\mathrm{Gr}}
	\newcommand{\gr}{\mathrm{gr}}
	\newcommand{\tr}{\mathrm{tr}}
	\newcommand{\HdR}{\mathscr{H}_{dR}}
	\newcommand{\Bild}{\mathrm{Bild}}
	\newcommand{\rk}{\mathrm{rk}~}
	\newcommand\gf[2]{\genfrac{}{}{0pt}{}{#1}{#2}}
\title{New Explicit Formulas for Faltings' Delta-Invariant}
\author{Robert Wilms}
\begin{abstract}
In this paper we give new explicit formulas for Faltings' $\delta$-invariant in terms of integrals of theta functions, and we deduce an explicit lower bound for $\delta$ only in terms of the genus and an explicit upper bound for the Arakelov--Green function in terms of $\delta$. Furthermore, we give a canonical extension of $\delta$ and the Zhang--Kawazumi invariant $\varphi$ to the moduli space of indecomposable principally polarised complex abelian varieties.
\end{abstract}
\maketitle
\section{Introduction}
\label{intro}
Faltings' $\delta$-invariant of compact and connected Riemann surfaces plays a crucial role in Arakelov theory of arithmetic surfaces. For example, $\delta$ appears in the arithmetic Noether formula proven by Faltings \cite{Fal84}.
In this paper we prove new explicit formulas for this invariant and we deduce several applications.
\subsection{Results}
Let $X$ be any compact and connected Riemann surface of genus $g\ge 1$.
We denote by $\|\theta\|$ the real-valued function on $\Pic_{g-1}(X)$ introduced by Faltings
in \cite[p. 401]{Fal84}, we recall the definition in Section \ref{secinvariantsabelian}. Further, we use the canonical identification $\Pic_{0}(X)\cong\Pic_{g-1}(X)$ by the vector of Riemann constants. Let $\nu$ be the Kähler form of the Hodge metric on $\Pic_{0}(X)$.
Our main result, stated in the following theorem, gives a relation
between Faltings' $\delta$-invariant and the Zhang--Kawazumi invariant $\varphi$ defined in \cite[Theorem 1.3.1]{Zha10} and \cite[Section 1]{Kaw08}, see also Section \ref{invariant} for the definitions of these invariants.
\begin{Thm}\label{mainresult}
Any compact and connected Riemann surface $X$ of genus $g\ge 1$ satisfies
$$\delta(X)=-24\int_{\Pic_{g-1}(X)}\log\|\theta\|\tfrac{\nu^{g}}{g!}+2\varphi(X)-8g\log2\pi.$$
\end{Thm}
Note, that $\varphi(X)=0$ if $g=1$.
Next, we state some applications of the theorem. The first application is an explicit lower bound for $\delta(X)$ depending only on $g$.
Until now, such bounds were only known for $g\le2$.
In the case $g=1$ the explicit lower bound $\delta>-9$ follows from Faltings' calculation in \cite[Section 7]{Fal84}, see also \cite[p. 92]{vKa14b}.
For $g=2$ von K\"anel proved in \cite[Proposition 5.1]{vKa14} that $\delta(X)\ge-186$.
We shall deduce from Theorem \ref{mainresult} the following corollary.
\begin{Cor}\label{deltalowerbound}
Any compact and connected Riemann surface $X$ of genus $g\ge 1$ satisfies $\delta(X)> -2g\log 2\pi^{4}$.
\end{Cor}
In the proof we apply the inequalities $\int_{\Pic_{g-1}(X)}\log\|\theta\|\tfrac{\nu^{g}}{g!}< -\tfrac{g}{4}\log2$ and $\varphi(X)\ge0$, where the former follows from Jensen's formula.

As another application, we obtain a canonical extension of the invariants $\delta$ and $\varphi$ to indecomposable principally polarised complex abelian varieties. De Jong introduced in \cite{dJo10} the function $\eta=\ltrans{(\theta_j)}(\theta_{jk})^c(\theta_j)$ on $\C^g$, where $(\theta_j)$ is the vector of the first partial derivatives of a theta function $\theta$ associated to a principally polarised complex abelian variety and $(\theta_{jk})$ is the matrix of its second partial derivatives.
In \cite{dJo08} he also defined a real valued version $\|\eta\|$ on $\Theta$, the zero divisor of $\theta$, see also Section \ref{secinvariantsabelian}. We will deduce the following theorem from Theorem \ref{mainresult} and from a formula for $\delta(X)$ by de Jong \cite[Theorem 4.4]{dJo08}.

\begin{Thm}\label{thmdeltavarphiabelian}
For any compact and connected Riemann surface $X$ of genus $g\ge 1$, the invariant $\delta(X)$ satisfies
$$\delta(X)=2(g-7)\int_{\Pic_{g-1}(X)}\log\|\theta\|\tfrac{\nu^{g}}{g!}-2\int_{\Theta}\log\|\eta\|\tfrac{\nu^{g-1}}{g!}-4g\log 2\pi.$$
Moreover, the invariant $\varphi(X)$ satisfies
$$\varphi(X)=(g+5)\int_{\Pic_{g-1}(X)}\log\|\theta\|\tfrac{\nu^{g}}{g!}-\int_{\Theta}\log\|\eta\|\tfrac{\nu^{g-1}}{g!}+2g\log 2\pi.$$
\end{Thm}
Here, $\Theta\subseteq\Pic_{g-1}(X)$ is the canonical divisor consisting of line bundles of degree $g-1$ having global sections.
Also, we deduce an explicit formula for the Arakelov--Green function $G(\cdot,\cdot)$. For this purpose, we calculate the invariant $A(X)$ in Bost's formula for $G$, see \cite{Bos87}.
\begin{Thm}\label{thmgreenexplicit}
	For any compact and connected Riemann surface $X$ of genus $g\ge2$ it holds
	$$\log G(P,Q)=\int_{\Theta+P-Q}\log\|\theta\|\tfrac{\nu^{g-1}}{g!}+\tfrac{1}{2g} \varphi(X)-\int_{\Pic_{g-1}(X)}\log\|\theta\|\tfrac{\nu^{g}}{g!}.$$
\end{Thm}
As a consequence of the theorem, we obtain the following upper bound for the Arakelov--Green function.
\begin{Cor}\label{corgreenbound}
	Let $X$ be any compact and connected Riemann surface of genus $g\ge 1$.
For any real number $r$ satisfying $r\ge \tfrac{6}{g}-1$ and $r>0$ we have
$$\sup_{P,Q\in X}g(P,Q)\le\tfrac{1+r}{24}\delta(X)+\tfrac{g(1+r)}{3}\log 2\pi+\log c_g+\tfrac{g(1+r)}{4}\log\tfrac{1+r}{2r},$$
where $c_g=\tfrac{g+2}{2}$ for $g\le 3$ and $c_g=\tfrac{g+2}{2}\left(\tfrac{g+2}{\pi\sqrt{3}}\right)^{g/2}$ for $g\ge 4$.
In particular, the Arakelov--Green function
	is bounded by
$$\sup_{P,Q\in X}g(P,Q)<\tfrac{1}{24g}\max(6,g+1)\delta(X)+\tfrac{3}{4}g\cdot\log g+4.$$ 
\end{Cor}
Upper bounds for the Arakelov--Green function were already obtained by Merkl \cite[Theorem 10.1]{EC11} and Jorgenson--Kramer \cite{JK06} by very different methods. But our bound seems to be more explicit and more natural in the sense of Arakelov theory.
If $X$ is the modular curve $X_1(N)$, one can apply this bound to compute the complexity of an algorithm by Edixhoven for the computation of Galois representations associated to modular forms, see \cite{EC11}. Indeed, $\delta(X_1(N))$ can be bounded polynomially in $N$, see for example \cite[Remark 5.8]{JK09} if $X_1(N)$ has genus $g\ge 2$, or \cite[Corollary 1.5.1]{Jav14} in general.

The method of proof of Theorem \ref{mainresult} allows us moreover to establish the following generalization of Rosenhain's formula on $\theta$-derivatives. We denote by $\|J\|$ the derivative version of $\|\theta\|$ introduced by Gu\`ardia \cite[Definition 2.1]{Gua99}.
\begin{Thm}\label{conjguar}
Let $X$ be any hyperelliptic Riemann surface of genus $g\ge 2$ and denote by $W_1,\dots,W_{2g+2}$ the Weierstraß points of $X$. For any permutation $\tau\in\mathrm{Sym}(2g+2)$ it holds
$$\|J\|(W_{\tau(1)},\dots,W_{\tau(g)})=\pi^g\prod_{j=g+1}^{2g+2}\|\theta\|(W_{\tau(1)}+\dots+W_{\tau(g)}-W_{\tau(j)}).$$
\end{Thm}
This gives an absolute value answer to a conjecture by Gu\`ardia \cite[Conjecture 14.1]{Gua02}.
\subsection{Main ideas of the proof}
We  describe the principal ideas of the proof of Theorem \ref{mainresult}. The case $g=1$ follows from Faltings' computations for elliptic curves in \cite[Section 7]{Fal84}. Hence, we can assume $g\ge 2$.
\paragraph{Reduction to hyperelliptic Riemann surfaces.}
Consider
\begin{align}\label{equf}
f(X)=\delta(X)+24\int_{\Pic_{g-1}(X)}\log\|\theta\|\tfrac{v^g}{g!}-2\varphi(X)
\end{align}
as a real-valued function on $\mathcal{M}_g$, the moduli space of compact and connected Riemann surfaces of genus $g$. Theorem \ref{mainresult} asserts that we have $f(X)=-8g\log 2\pi$.
We will reduce to prove Theorem \ref{mainresult} for hyperelliptic Riemann surfaces by showing that $f(X)$ is pluriharmonic. Since there are no non-constant pluriharmonic functions on $\mathcal{M}_g$ for $g\ge 3$, it follows that $f(X)$ is constant. Since for any $g\ge 3$ there exists at least one hyperelliptic Riemann surface of genus $g$ and for $g=2$ all compact and connected Riemann surface are hyperelliptic, it is then enough to compute $f(X)$ if $X$ is hyperelliptic.

To prove that $f(X)$ is pluriharmonic on $\mathcal{M}_g$, we apply the Laplace operator $\partial\overline{\partial}$ on $\mathcal{M}_g$ to the terms in (\ref{equf}). For $\varphi(X)$ and $\delta(X)$ we have an expression for the resulting forms in terms of the canonical forms $e_1^A$, $\int_{\pi_2}h^3$ and $\omega_{\mathrm{Hdg}}$ (see Section \ref{secforms}) on $\mathcal{M}_g$ by de Jong \cite{dJo16}. To calculate the application of $\partial\overline{\partial}$ to the integral in (\ref{equf}), 
we apply the Laplace operator on the universal abelian variety with level $2$ structure to $\log\|\theta\|$, and we pull back the integral to the $(g+1)$-th power of the universal Riemann surface with level $2$ structure $\mathcal{X}_g\to\mathcal{M}_g[2]$.
The pullback can be expressed in terms of Deligne pairings by a result due to de Jong \cite[Proposition 6.3]{dJo16}. The main difficulty is to express the first Chern form of the $(g+1)$-th power in the sense of Deligne pairings of the line bundle
$$\bigotimes_{j=1}^g pr_j^*T_{\mathcal{X}_g/\mathcal{M}_g[2]}\otimes \bigotimes_{j=1}^g pr_{j,g+1}^*\O(\Delta)^{\vee}\otimes\bigotimes_{j<k}^{g} pr_{j,k}^*\O(\Delta)$$
on $\mathcal{X}_g^{g+1}$ in terms of the forms $e_1^A$, $\int_{\pi_2}h^3$ and $e^A$. For the definition of $e^A$ see also Section \ref{secforms}. Here, $T_{\mathcal{X}_g/\mathcal{M}_g[2]}$ denotes the relative tangent bundle, $\Delta\subseteq\mathcal{X}_g^2$ is the diagonal and $pr_j$ and $pr_{j,k}$ denote the projections to the respective factors of $\mathcal{X}_g^{g+1}$.
We solve this problem by associating graphs to the terms in the expansion of the power, which we can classify and count.

\paragraph{The hyperelliptic case.}
To prove Theorem \ref{mainresult} for any hyperelliptic Riemann surface $X$ of genus $g\ge 2$, we generalize a formula by Bost for $\delta$ for $g=2$, see \cite[Proposition 4]{Bos87}, to hyperelliptic Riemann surfaces of genus $g\ge 2$. The generalized formula states:
\begin{align}\label{bostgeneral}
\delta(X)=-\tfrac{8(g-1)}{g}\int_{\Pic_{g-1}(X)}\log\|\theta\|\tfrac{\nu^{g}}{g!}-\tbinom{2g}{g-1}^{-1}\log\|\Delta_{g}\|(X)-8g\log2\pi.
\end{align}
We will see later in Section \ref{secgeneral} that we can canonically define the invariant $\|\Delta_{g}\|(X)$ also for non-hyperelliptic Riemann surfaces, but formula (\ref{bostgeneral}) will not be true for general Riemann surfaces.
The main ingredient of the proof of formula (\ref{bostgeneral}) is the following formula
\begin{align}\label{thetaintegrals}
\int_{\Pic_{g-1}(X)}\log\|\theta\|^{g-1}\tfrac{\nu^{g}}{g!}=\int_{X^{g}}\log\tfrac{\|\theta\|(P_{1}+\dots+P_{g}-Q)^{g}}{\|\theta\|(gP_{1}-Q)}\mu(P_{1})\dots \mu(P_{g}),
\end{align}
where $\mu$ denotes the canonical Arakelov $(1,1)$-form and $Q\in X$ is an arbitrary point.
The proof of (\ref{thetaintegrals}) consists essentially of two steps.
The first step is to decompose the function $\log\|\theta\|$ into a sum of Arakelov--Green functions and an additional invariant.
This step generalizes the decomposition in \cite[A.1.]{BMM90} for $g=2$ to arbitrary hyperelliptic Riemann surfaces.
In the second step, we use the decomposition of the first step to compare the pullbacks of the
integral of $\log\|\theta\|$ on $\Pic_{g-1}(X)$ under the maps
\begin{align*}
\Phi&\colon X^{g}\to \Pic_{g-1}(X),\quad(P_{1},\dots,P_{g})\mapsto (P_{1}+\dots+P_{g}-Q),\\
\Psi&\colon X^{g}\to \Pic_{g-1}(X),\quad(P_{1},\dots,P_{g})\mapsto (2P_{1}+P_{2}+\dots+P_{g-1}-P_{g}).
\end{align*}
In this step we also obtain a connection of the integrals of $\log\|\theta\|$ to the invariant $\varphi(X)$.

To connect formula (\ref{thetaintegrals}) with Faltings' $\delta$-invariant, we compare
two invariants obtained by the function $\|J\|$: The iterated integral over $X^{g}$ of its logarithm and the product of its values in Weierstraß points. 
Then we apply Gu\`ardia's expression for $\delta(X)$ in \cite[Proposition 1.1]{Gua99} to connect the first invariant with (\ref{thetaintegrals}) and $\delta(X)$.
De Jong proved in \cite[Theorem 9.1]{dJo07} that the second invariant is essentially $\|\Delta_{g}\|(X)$.
This together with de Jong's expression for $\delta(X)$ of hyperelliptic Riemann surfaces,
given in \cite[Corollary 1.7]{dJo05b}, leads to formula (\ref{bostgeneral}) and also to the formula in Theorem \ref{mainresult} by the connection of the integrals of $\log\|\theta\|$ to $\varphi(X)$.

We remark, that there is an alternative way to compute the constant $f(X)$:
Consider a family of Riemann surfaces $\mathscr{X}_t$ of genus $g\ge 2$ degenerating to a singular complex curve $\mathscr{X}_0$ consisting of two Riemann surfaces $X_1$ and $X_2$ of genus $g-1$ respectively $1$ meeting in one point. For this family the integral in (\ref{equf}) degenerates nicely and the asymptotic behaviour of $\delta(\mathscr{X}_t)$ and $\varphi(\mathscr{X}_t)$ were studied by Wentworth \cite{Wen91} and de Jong \cite{dJo14a}. Using their results, one can deduce that $f(\mathscr{X}_t)$ degenerates to $f(X_1)+f(X_2)$ in this family. Hence, one obtains $f(X)=-8g\log 2 \pi$ by induction.
However, the methods of our proof for hyperelliptic Riemann surfaces are of independent interest. For example, they also prove formulas (\ref{bostgeneral}) and Theorem \ref{conjguar}.

\subsection{Overview}
In the following, we explain the structure of this paper.
We define all required invariants of abelian varieties and Riemann surfaces in Section \ref{sectioninvariants}.
In Section \ref{Sectionintegrals} we study certain integrals of $\log\|\theta\|$ and of the Arakelov--Green function. The subsequent section deals with the proof of Theorem \ref{mainresult} for hyperelliptic Riemann surfaces. First, we introduce our decomposition of $\log\|\theta\|$ in Section \ref{secdecomposition}. We prove equation (\ref{thetaintegrals}) in Section \ref{secintegralhyper}. In Section \ref{sectionmainhyper} we obtain Theorem \ref{mainresult} for hyperelliptic Riemann surfaces, and we give some consequences and examples. In Section \ref{secrosenhain} we apply our decomposition of $\log\|\theta\|$ to prove Theorem \ref{conjguar}. 

In Section \ref{Sectiongeneral} we prove Theorem \ref{mainresult} in general. For this purpose, we discuss the forms obtained by applying the Laplace operator $\partial\overline{\partial}$ on $\mathcal{M}_g$ to $\varphi(X)$, $\delta(X)$ and $H(X)$ in Section \ref{secforms}, where $H(X)$ denotes the integral in Theorem \ref{mainresult}. To compare the latter one with the former ones, we introduce the Deligne pairing in Section \ref{secdeligne}, and we calculate the expansion of the power in the sense of Deligne pairings of a certain line bundle using graphs in Section \ref{secgraphs}. In Section \ref{secgeneral} we conclude our main result Theorem \ref{mainresult}, and we deduce Corollary \ref{deltalowerbound}. After bounding the function $\|\theta\|$ in Section \ref{secbound}, we study the Arakelov--Green function in Section \ref{green}, where we obtain Theorem \ref{thmgreenexplicit} and Corollary \ref{corgreenbound}.

In the last section we consider indecomposable principally polarised complex abelian varieties.
We prove Theorem \ref{thmdeltavarphiabelian} in Section \ref{secabelian}. This yields a canonical extension of $\delta$ and $\varphi$ to indecomposable principally polarised complex abelian varieties.
We discuss some of their asymptotic behaviours in Section \ref{secasmypt}.

\subsection*{Acknowledgement}
The contents of this paper largely coincide with my Ph.D. thesis. I would like to thank my advisor Gerd Faltings for introducing me into Arakelov theory and for his suggestion to study the $\delta$-invariant.
I also would like to thank Rafael von K\"anel and Robin de Jong for useful discussions and Michael Rapoport for a remark on an early version of this paper.

\section{Invariants}\label{sectioninvariants}
We give the definitions of the invariants appearing in this paper and we state some of their properties and relations.
\subsection{Invariants of abelian varieties}
\label{secinvariantsabelian}
In this subsection we define some invariants of abelian varieties.
Let $(A,\Theta)$ be any principally polarised complex abelian variety of dimension $g\ge 1$, where $\Theta\subseteq A$ denotes a divisor, such that $\O(\Theta)$ is an ample and symmetric line bundle satisfying $\dim H^0(A,\O(\Theta))=1$. The principal polarisation of $(A,\Theta)$ determines the divisor class of $\Theta$ only up to a translation by a $2$-torsion point.
There exists a complex and symmetric $g\times g$ matrix $\Omega$ with positive definite imaginary part $Y=\Im\Omega$, such that $A=\C^g/(\Z^g+\Omega\Z^g)$ and $\Theta$ is the zero divisor of the function
$$\theta\colon \C^g \to \C,\quad z\mapsto \theta(z)=\theta(\Omega;z)=\sum_{n\in\Z^g}\exp(\pi i\ltrans{n}\Omega n+2\pi i\ltrans{n}z),$$
see for example \cite[Section 8]{BL04}.

Since we have $\theta(z+m+n\Omega)=\exp(-\pi i\ltrans{n}\Omega n-2\pi i\ltrans{n}z)\theta(z)$
for $m,n\in\Z^{g}$, we obtain a well-defined, real-valued function $\|\theta\|\colon A\to \R_{\ge0}$ by
$$\|\theta\|(z)=\|\theta\|(\Omega;z)=\det(Y)^{1/4}\exp(-\pi\ltrans{(\Im z)}Y^{-1}(\Im z))\cdot |\theta|(z).$$

We associate to $(A,\Theta)$ the canonical $(1,1)$-form
$$\nu=\nu_{(A,\Theta)}=\tfrac{i}{2}\sum_{j,k=1}^g (Y^{-1})_{j k} dZ_j\wedge d\overline{Z}_k,$$
where $Z_1,\dots,Z_g$ are coordinates in $\C^g$. This form is translation-invariant. The function $\|\theta\|$ could also
be defined as the unique function $\|\theta\|\colon A\to \R_{\ge 0}$ satisfying:
\begin{enumerate}[($\theta$1)]
	\item The function $\|\theta\|^2$ is $C^{\infty}$ on $A$.
	\item The zero divisor of $\|\theta\|$ is $\Theta$.
	\item For $z\notin \Theta$ its curvature is given by $\partial \bar{\partial}\log\|\theta\|(z)^2=2\pi i\nu$.\label{theta2}
	\item The function is normed by $\tfrac{1}{g!}\int_{A}\|\theta\|^{2}\nu^{g}=2^{-g/2}$.
\end{enumerate}
For a calculation of the third property see \cite[Proposition 8.5.6]{BL04}. In particular, the function $\|\theta\|$ depends on the choice of $\Theta$ for a principally polarised complex abelian variety.
Further, we define the following invariant
$$H(A,\Theta)=\tfrac{1}{g!}\int_A \log\|\theta\|\nu^g.$$
If $(A,\Theta)$ is the Jacobian variety of a Riemann surface $X$ of genus $g=2$, this definition coincides with the definition of $\log\|H\|(X)$ in \cite{Bos87}. The invariant $H(A,\Theta)$ is bounded from above.
\begin{Pro}\label{Hbound}
Let $(A,\Theta)$ be any principally polarised complex abelian variety of dimension $g\ge 1$. Then it holds $H(A,\Theta)< -\tfrac{g}{4}\log 2$.
\end{Pro}
\begin{proof}
Since $\int_{A}\nu^{g}=g!$, Jensen's inequality and ($\theta$4) give
	$$2H(A,\Theta)=\tfrac{1}{g!}\int_{A}\log\|\theta\|^{2}\nu^{g}<\log\left(\tfrac{1}{g!}\int_{A}\|\theta\|^{2}\nu^{g}\right)=-\tfrac{g}{2}\log 2.$$
 \end{proof}

We obtain another function $\eta\colon \C^g\to \C$ by considering certain partial derivatives of $\theta$:
$$\eta(z)=\det\begin{pmatrix}
\tfrac{\partial^2\theta}{\partial Z_1\partial Z_1}(z) &\dots&\tfrac{\partial^2\theta}{\partial Z_1\partial Z_g}(z)& \tfrac{\partial\theta}{\partial Z_1}(z)\\
\vdots&\ddots&\vdots&\vdots\\
\tfrac{\partial^2\theta}{\partial Z_g\partial Z_1}(z) &\dots&\tfrac{\partial^2\theta}{\partial Z_g\partial Z_g}(z)& \tfrac{\partial\theta}{\partial Z_g}(z)\\
\tfrac{\partial\theta}{\partial Z_1}(z)&\dots&\tfrac{\partial\theta}{\partial Z_g}(z)&0  \end{pmatrix},$$
see also \cite{dJo10}. Further, we get a real-valued variant $\|\eta\|\colon \Theta\to \R_{\ge0}$ by
$$\|\eta\|(z)=\det(Y)^{(g+5)/4}\exp(-\pi(g+1)\ltrans{(\Im z)}Y^{-1}(\Im z))\cdot|\eta|(z),$$
see also \cite{dJo08}.
The function $\eta$ is identically zero on $\Theta$ if and only if $(A,\Theta)$ is decomposable, see \cite[Corollary 3.2]{dJo10}. We set 
$$\Lambda(A,\Theta)=\tfrac{1}{g!}\int_{\Theta}\log\|\eta\|\nu^{g-1}$$
if $(A,\Theta)$ is indecomposable.

By definition, the invariants $H(A,\Theta)$ and $\Lambda(A,\Theta)$ are invariant under translation by $2$-torsion points of the divisor $\Theta$. Hence, we can indeed consider them as invariants of (indecomposable) principally polarised complex abelian varieties.
\subsection{Invariants of Riemann surfaces}
\label{invariant}
In this subsection we define some invariants of Riemann surfaces.
Let $X$ be a compact and connected Riemann surface of genus $g\ge 1$.
Let $A_{i},B_{i}\in H_{1}(X,\Z)$ be a symplectic basis of homology,
that means for all $i,j$ the intersection pairings $(A_{i}.A_{j})$ and $(B_{i}.B_{j})$ are zero and $(A_{i}.B_{j})=\delta_{ij}$, where $\delta_{ij}$ denotes the Kronecker symbol. Further, we choose a basis of one forms $\omega_{1},\dots,\omega_{g}\in H^{0}(X,\Omega_{X}^{1})$,
such that $\int_{A_{j}}\omega_{i}=\delta_{ij}$. We associate to $X$ its period matrix $\Omega=\Omega_X$, which is given by $\Omega_{ij}=\int_{B_{i}}\omega_{j}$. It is symmetric and has positive definite imaginary part denoted by $Y=\Im \Omega$. The Jacobian of $X$, denoted by $\Jac(X)$, is the principally polarised abelian variety associated to $\Omega$. In the following, we shortly write $\nu=\nu_{\Jac(X)}$. For a fixed base point $Q\in X$ the Abel--Jacobi map is given by the embedding
$I\colon X\to\Jac(X),~P\mapsto(\int_{Q}^{P}\omega_{1},\dotsc,\int_{Q}^{P}\omega_{g}).$
We define the canonical $(1,1)$ form $\mu$ on $X$ by
$\mu=\tfrac{1}{g}I^{*}\nu,$
which has volume $\int_{X}\mu=1$.
Further, we denote the canonical bundle on $X$ by $K_X$.

There is a unique theta characteristic $\alpha_X$, that is $2\alpha_X=K_X$, which gives an isomorphism
\begin{align}\label{PicJac}
\Pic_{g-1}(X)\xrightarrow{\sim} \Jac(X),\quad\mathcal{L}\mapsto (\mathcal{L}-\alpha_X),
\end{align}
such that $\|\theta\|(\Omega;\mathcal{L}-\alpha_X)=0$ if and only if $H^{0}(X,\mathcal{L})\neq 0$, see for example \cite[Corollary II.3.6]{Mum83}.
We simply write $\|\theta\|(D)=\|\theta\|(\O(D)-\alpha_X)$
for a divisor $D$ of degree $g-1$. It follows, that $\|\theta\|(D)=0$ if and only
if $D$ is linearly equivalent to an effective divisor. We denote $\Theta\subseteq\Pic_{g-1}(X)$
for the divisor in $\Pic_{g-1}(X)$ defined by the zeros of $\|\theta\|$. Equivalently, $\Theta$ is given
by the line bundles of degree $g-1$ having global sections. For any effective divisor $D$ of degree $g-1$ we also simply write $\|\eta\|(D)=\|\eta\|(\O(D)-\alpha_X)$.
Since the divisor $\Theta\subseteq \Pic_{g-1}(X)$ is canonical, the functions $\|\theta\|$ and $\|\eta\|$ on $\Pic_{g-1}(X)$ do not depend on the choice of the period matrix $\Omega$.

We set $H(X)=H(\Jac(X))$ and $\Lambda(X)=\Lambda(\Jac(X))$.
Another invariant $S(X)$ of $X$ was defined by de Jong in \cite[Section 2]{dJo05a}. It satisfies
$$\log S(X)=-\int_{X}\log\|\theta\|(gP-Q)\mu(P).$$
We generalize this to a family of invariants given by
$$S_{k}(X)=\int_{X^{k}}\log\|\theta\|((g-k+1)P_{1}+P_{2}+\dots+P_{k}-Q)\mu(P_{1})\dots\mu(P_{k})$$
for every $1\le k\le g$, where $k$ stands for the dimension of the integral. In particular, we have $S_{1}(X)=-\log S(X)$. In Remark \ref{independofQ} we will see, that $S_k(X)$ does not depend on the choice of $Q$.
We will prove a relation between $H(X)$, $S_{1}(X)$ and $S_{g}(X)$ for hyperelliptic Riemann surfaces in Section \ref{secintegralhyper}.

Another way to build an invariant from the function $\theta$ is to consider
the integral of certain derivatives. For this purpose, we define as in \cite[Definition 2.1]{Gua99}
\begin{align*}
&\|J\|(P_{1},\dots,P_{g})=\det(Y)^{(g+2)/4}\exp\left(-\pi\sum_{k=1}^{g}\ltrans{y_{k}}Y^{-1}y_{k}\right)\left|\det\left(\tfrac{\partial\theta}{\partial Z_{l}}(w_k)\right)\right|,
\end{align*}
where $P_1,\dots,P_g$ denote arbitrary points of $X$, $w_{k}\in \C^g$ is a lift of the divisor class $(P_{1}+\dots+P_{g}-P_{k}-\alpha_X)\in\Jac(X)$ and $y_k$ is the imaginary part of $w_k$. To get an invariant we set
$$B(X)=\int_{X^{g}}\log\|J\|(P_{1},\dots,P_{g})\mu(P_{1})\dots\mu(P_{g}).$$
For hyperelliptic Riemann surfaces we will give a relation between $B(X)$ and $S_{g}(X)$
in Section \ref{secdecomposition}.

We define the Arakelov--Green function $G\colon X^{2}\to\R_{\ge 0}$ as the unique function
satisfying the following conditions:
\begin{enumerate}[(G1)]
\item The function $G(P,Q)^2$ is $C^{\infty}$ on $X^2$.
	\item We have $G(P,Q)>0$ for $P\neq Q$. For a fixed $Q\in X$, $G(P,Q)$ has a simple zero in $P=Q$.
	\item For $P\neq Q$ the curvature is given by $\partial_{P}\bar{\partial}_{P}\log G(P,Q)^2=2\pi i \mu(P).$\label{green2}
	\item It is normalized by $\int_{X}\log G(P,Q)\mu(P)=0$.
\end{enumerate}
One can check that $G(P,Q)=G(Q,P)$. We shortly write
$g(P,Q)=\log G(P,Q).$
Bost has shown in \cite[Proposition 1]{Bos87}, that there is an invariant $A(X)$, such that
\begin{align}\label{Bostinvariant}
g(P,Q)=\tfrac{1}{g!}\int_{\Theta+P-Q}\log\|\theta\|\nu^{g-1}+A(X).
\end{align}
We will give an explicit expression for the invariant $A(X)$ in Section \ref{green}.
Another natural invariant defined by the Arakelov--Green function is its supremum. We will bound
it by more explicit invariants also in Section \ref{green}.

Next, we recall the definition of Faltings' $\delta$ invariant in \cite[p. 402]{Fal84}.
Denote by $\psi_{1},\dots,\psi_{g}$ another basis of $H^{0}(X,\Omega_{X}^{1})$, which is
orthonormal with respect to the inner product
\begin{align}\label{hodgemetric}
\langle\psi,\psi'\rangle=\tfrac{i}{2}\int_{X}\psi\wedge\overline{\psi'}.
\end{align}
Then the defining equation for $\delta(X)$ is
\begin{align*}
\|\theta\|(P_{1}+\dots+P_{g}-Q)=\exp\left(-\tfrac{1}{8}\delta(X)\right)\cdot\frac{\|\det(\psi_{j}(P_{k}))\|_{\Ar}}{\prod_{j<k}G(P_{j},P_{k})}\cdot \prod_{j=1}^gG(P_{j},Q),
\end{align*}
where $P_{1},\dots,P_{g},Q$ are pairwise different points, such that the class of the divisor $(P_{1}+\dots+P_{g}-Q)$ lies not in $\Theta$
and the Arakelov norm $\|\cdot\|_{\Ar}$ of holomorphic one forms is induced by
\begin{align}\label{arakelovnorm}
\|dz\|_{\Ar}(P)=\lim_{Q\to P}\frac{|z(Q)-z(P)|}{G(P,Q)},
\end{align}
where $z\colon U\to\C$ is a local coordinate of a neighbourhood $P\in U\subseteq X$. 
This invariant plays an important role in Arakelov theory. For example, it is up to a constant 
the $\delta$ in the arithmetic Noether formula for the archimedean places.
In Section \ref{secgeneral} we will obtain a new expression and
a lower bound only in terms of $g$ for $\delta(X)$.

In \cite[Proposition 1.1]{Gua99} Guàrdia gave the following expression
\begin{align}\label{deltagua}
\|\theta\|(P_{1}+\dots+P_{g}-Q)^{g-1}=\exp\left(\tfrac{1}{8}\delta(X)\right)\frac{\|J\|(P_{1},\dots,P_{g})}{\prod_{j<k}G(P_{j},P_{k})}\prod_{j=1}^gG(P_{j},Q)^{g-1}.
\end{align}
Taking logarithm and integrating with $\mu(P_{1})\dots\mu(P_{g})$ gives
\begin{align}\label{deltaab}
	\delta(X)=8(g-1)S_{g}(X)-8B(X).
\end{align}

We state another formula for $\delta(X)$ by de Jong. For this purpose, let $\Theta^{sm}$ be the smooth part of $\Theta\subseteq\Pic_{g-1}(X)$. Every divisor $D\in\Theta^{sm}$ has a unique representation $D=P_1+\dots+P_{g-1}$ for some points $P_1,\dots,P_{g-1}$ on $X$. By Riemann--Roch, the involution 
$$\sigma\colon \Pic_{g-1}(X)\to \Pic_{g-1}(X),\quad D\to K_X-D$$
induces an involution on $\Theta^{sm}$.
Let $D=P_1+\dots+P_{r}$ and $E=Q_1+\dots+Q_s$ be two effective divisors. We define the Arakelov--Green function for $D$ and $E$ by
$$G(D,E)=\prod_{j=1}^r\prod_{k=1}^s G(P_j,Q_k).$$

For any $D\in\Theta^{sm}$ and any different points $Q,R\in X$, such that $Q$ (respectively $R$) is not contained in the unique expression of $D$ (respectively $\sigma(D)$) as sum of $g-1$ points, we have by \cite[Theorem 4.4]{dJo08}
\begin{align*}
\|\eta\|(D)=\exp\left(-\tfrac{1}{4}\delta(X)\right) G(D,\sigma(D)) \left(\frac{\|\theta\|(D+R-Q)}{G(R,Q)G(D,Q)G(\sigma(D),R)}\right)^{g-1}.
\end{align*}
Write $D=P_1+\dots +P_{g-1}$. If we take the product of the $g-1$ equations obtained by putting $R=P_j$ for each $j\le g-1$ in the above equation, we get
\begin{align}\label{deltaeta}
\|\eta\|(P_1+\dots +P_{g-1})=\exp\left(-\tfrac{1}{4}\delta(X)\right)\prod_{j=1}^{g-1}\frac{\|\theta\|(P_1+\dots +P_{g-1}+P_j-Q)}{G(P_j,Q)^g}.
\end{align}

\begin{Rem}\label{independofQ}
For a divisor $D\in\Theta^{sm}$ and points $P,Q\in X$ the term
\begin{align}\label{Lambda}
\log\|\Lambda\|(D)=\log\|\theta\|(D+P-Q)-g(P,Q)-g(D,Q)-g(\sigma(D),P)
\end{align}
does not depend on $P$ or $Q$, see \cite[Proposition 4.3]{dJo08}. Let $1\le k\le g$ be an integer. Setting $P=P_k$ and $D=(g-k+1)P_1+P_2+\dots+P_{k-1}$ and integrating equation (\ref{Lambda}) multiplied with $\mu(P_1)\dots\mu(P_k)$ over $X^k$, one notices that the invariant $S_k$ does indeed not depend on the choice of $Q\in X$.
\end{Rem}

Next, we define the Zhang--Kawazumi invariant $\varphi(X)$, which was introduced and studied
independently in \cite{Kaw08} and \cite{Zha10}. For this purpose, we consider the diagonal divisor $\Delta\subseteq X^{2}$.
We have a hermitian metric on $\O_{X^{2}}(\Delta)$ by $\|1\|(P_{1},P_{2})=G(P_{1},P_{2})$, where $1$ is the 
canonical section of $\O_{X^2}(\Delta)$. We denote by $h_{\Delta}$ the curvature form of $\O_{X^2}(\Delta)$.
It can be given explicitly by
\begin{align}\label{hdelta}
h_{\Delta}(P_{1},P_{2})=\mu(P_{1})+\mu(P_{2})-\tfrac{i}{2}\sum_{k=1}^{g}\left(\psi_{k}(P_{1})\wedge\bar{\psi}_{k}(P_{2})+\psi_{k}(P_{2})\wedge\bar{\psi}_{k}(P_{1})\right),
\end{align}
see \cite[Proposition 3.1]{Ara74}). We define $\varphi(X)$ by
\begin{align}\label{kzdef}
\varphi(X)=\int_{X^{2}}g(P_{1},P_{2})h^{2}_{\Delta}(P_{1},P_{2}),
\end{align}
see \cite[Proposition 5.1]{dJo16} or the proof of \cite[Proposition 2.5.3]{Zha10}.
It is not difficult to prove $\varphi(X)=0$ for $g=1$ and that we have
the lower bound
\begin{align}\label{kzbound}
	\varphi(X)>0
\end{align}
for $g\ge2$, see \cite[Corollary 1.2]{Kaw08} or \cite[Proposition 4.2]{dJo16}.

\subsection{Invariants of hyperelliptic Riemann surfaces}
We consider the more special case of hyperelliptic Riemann surfaces.
Let $X$
be any hyperelliptic Riemann surface of genus $g\ge2$. 
That means, there are
pairwise different complex numbers $a_{1},\dots,a_{2g+1}\in\C$, such that $X$ is given by the equation
\begin{align}\label{hypdef}
y^{2}=(x-a_{1})\cdot(x-a_{2})\cdot\ldots\cdot(x-a_{2g+1})(=:f(x))
\end{align}
and the unique point at infinity, denoted by $\infty\in X$.
There is a canonical involution induced by $y\mapsto -y$,
which we denote by $\sigma\colon X\to X$. The fixed points of $\sigma$ are
the Weierstraß points of $X$. They correspond to the points $x=a_{1},\dots,x=a_{2g+1}$
and the point $\infty$. We denote them by $W_{1},\dots,W_{2g+2}$, where $W_{2g+2}=\infty$.
For the symplectic basis of homology $A_1,\dots,A_g,B_1,\dots,B_g$ we choose the canonical one, see \cite[Section IIIa, §5]{Mum84}.

For hyperelliptic Riemann surfaces we define the Petersson norm of the modular discriminant $\|\varphi_{g}\|(X)$.
For every $\eta=\left[\gf{\eta'}{\eta''}\right]$ with $\eta',\eta''\in \frac{1}{2}\Z^{g}$ we set
$$\theta[\eta](z)=\exp\left(\pi i\ltrans{\eta'}\Omega\eta'+2\pi i\ltrans{\eta'}(\eta''+z)\right)\theta\left(\Omega\eta'+\eta''+z\right).$$
Further, as in \cite[Section IIIa, Definition 5.7]{Mum84} we define
\begin{align*}
	\eta_{2k-1}&=\left[\gf{\ltrans{(0,\dots,0,\frac{1}{2},0,\dots,0)}}{\ltrans{(\frac{1}{2},\dots,\frac{1}{2},0,0,\dots,0)}}\right]\text{ for } 1\le k\le g+1,\\
	\eta_{2k}&=\left[\gf{\ltrans{(0,\dots,0,\frac{1}{2},0,\dots,0)}}{\ltrans{(\frac{1}{2},\dots,\frac{1}{2},\frac{1}{2},0,\dots,0)}}\right]\text{ for } 1\le k\le g,
\end{align*}
where the non-zero entry in the top row occurs in the $k$-th position.
For a subset $S\subseteq\lbrace 1,\dots,2g+1\rbrace$ we set $\eta_{S}=\sum_{k\in S}\eta_{k}(\mod 1)$.
We denote by $\mathcal{T}$ the collection of all subsets of $\lbrace 1,\dots, 2g+1\rbrace$ of cardinality $g+1$.
Further, we set $U=\lbrace 1,3,\dots,2g+1\rbrace$ and we write $\circ$ for the symmetric difference.

We define the Petersson norm of the modular discriminant of $X$ by
$$\|\varphi_{g}\|(X)=(\det Y)^{2\binom{2g+1}{g+1}}\prod_{T\in\mathcal{T}}\left|\theta[\eta_{T\circ U}](0)\right|^{8}.$$
Further, we denote a modified version by
$$\|\Delta_{g}\|(X)=2^{-4(g+1)\binom{2g}{g-1}}\|\varphi_{g}\|(X).$$
For a discussion on the relation between $\|\varphi_{g}\|(X)$ and the discriminant of the polynomial $f$ in (\ref{hypdef})
we refer to Lockhart \cite{Loc94}.

As a direct consequence of the correspondence in \cite[IIIa Proposition 6.2]{Mum84} we obtain the following identity
\begin{align}\label{equdiscinfty}
\|\varphi_{g}\|(X)=\prod_{\lbrace j_{1},\dots,j_{g+1}\rbrace\in\mathcal{T}}\|\theta\|(W_{j_{1}}+\dots+W_{j_{g}}-W_{j_{g+1}})^{8}.
\end{align}
Since we can choose every Weierstraß point to be the point at infinity and the invariant $\|\varphi_{g}\|(X)$ does not depend on
this choice, we get by taking the product over all these choices
\begin{align}\label{equdiscweier}
\|\varphi_{g}\|(X)=\prod_{\lbrace j_{1},\dots,j_{g+1}\rbrace\in\mathcal{U}_{g+1}}\|\theta\|(W_{j_{1}}+\dots+W_{j_{g}}-W_{j_{g+1}})^{4},
\end{align}
where we denote by $\mathcal{U}_{k}$ the collection of all subsets of $\lbrace 1,\dots,2g+2\rbrace$ of cardinality $k$.
Due to de Jong we have the relations
\begin{align}\label{deltadejong}
\delta(X)=\tfrac{4(g-1)}{g^{2}}S_{1}(X)-\tfrac{3g-1}{2g}\tbinom{2g}{g-1}^{-1}\log\|\Delta_{g}\|(X)-8g\log2\pi,
\end{align}
see \cite[Corollary 1.7]{dJo05b}, and
\begin{align}\label{djr}
	\prod_{\lbrace j_{1},\dots,j_{g}\rbrace\in\mathcal{U}_{g}}\|J\|(W_{j_{1}},\dots,W_{j_{g}})=\pi^{\binom{2g+2}{g}g}\cdot\|\varphi_{g}\|(X)^{(g+1)/4},
\end{align}
see \cite[Theorem 9.1]{dJo07}.

\section{Integrals}\label{Sectionintegrals}
We give some relations between different integrals of the function $\log\|\theta\|$ and the Arakelov-Green function. Let $X$ be any compact and connected Riemann surface of genus $g\ge 2$ throughout this section.
\subsection{Integrals of theta functions}
\label{integraltheta}
In this subsection we establish two ways to write the invariant $H(X)$ as an integral over $X^g$.
For a fixed base point $Q\in X$ we define the maps
\begin{align*}
\Phi&\colon X^{g}\to \Pic_{g-1}(X),\quad(P_{1},\dots,P_{g})\mapsto (P_{1}+\dots+P_{g}-Q),\\
\Psi&\colon X^{g}\to \Pic_{g-1}(X),\quad(P_{1},\dots,P_{g})\mapsto (2P_{1}+P_{2}+\dots+P_{g-1}-P_{g}).
\end{align*}
Here and in the following, we call a map smooth if it is a smooth map of manifolds in the sense of \cite[Definition 1.2.1]{DFN85}, i.e. in local coordinates the first partial derivatives exist and are continuous. In particular, it has not to be flat as an algebraic morphism.
\begin{Pro}\label{propphipsi}
	The maps $\Phi$ and $\Psi$ are smooth and surjective. Moreover, $\Phi$ is generically of
	degree $g!$, $\Psi$ is generically of degree $4g!$,
	and the pullbacks of the volume form $\nu^{g}$ satisfy $\Psi^{*}\nu^{g}=4\Phi^{*}\nu^{g}$.
\end{Pro}
\begin{proof}
	The maps are defined as linear combinations of the Abel--Jacobi map. Hence, they are smooth.
	Jacobi's inversion theorem gives the surjectivity of $\Phi$, see for example \cite[III.6]{FK80}. If we
	have divisors $P_{1}+\dots+P_{g}-Q$ and $R_{1}+\dots+R_{g}-Q$ representing the same class in $\Pic_{g-1}(X)\setminus\Theta$,
	then
	$$\|\theta\|(P_{1}+\dots+P_{g}-Q)=\|\theta\|(R_{1}+\dots+R_{g}-Q)$$
	has zeros in $P_{1},\dots,P_{g},R_{1},\dots,R_{g}$ as a function in $Q$. But it has exactly $g$ zeros
	counted	with multiplicities, see for example \cite[Theorem II.3.1]{Mum83}. Hence, $(P_{1},\dots,P_{g})$ and $(R_{1},\dots,R_{g})$
	coincide up to order. Thus, $\Phi$ is generically of degree $g!$.
	
	For the pullbacks of $dZ_{k}$ we get
	$$\Phi^{*}(dZ_{k})=\sum_{j=1}^{g}\omega_{k}(P_{j})\quad\text{and}\quad\Psi^{*}(dZ_{k})=2\omega_{k}(P_{1})-\omega_{k}(P_{g})+\sum_{j=2}^{g-1}\omega_{k}(P_{j}).$$
	Therefore, $\Phi^{*}\nu^{g}$ is a linear combination of terms of the form
	$$\omega_{\rho(1)}(P_{1})\wedge\bar{\omega}_{\tau(1)}(P_{1})\wedge\dots\wedge\omega_{\rho(g)}(P_{g})\wedge\bar{\omega}_{\tau(g)}(P_{g}),$$
	for two permutations $\rho,\tau\in \mathrm{Sym(g)}$. But $\Psi^{*}\nu^{g}$ is
	the same linear combination in the terms
	\begin{align*}
	&2\omega_{\rho(1)}(P_{1})\wedge2\bar{\omega}_{\tau(1)}(P_{1})\wedge\omega_{\rho(2)}(P_{2})\wedge\dots\wedge\bar{\omega}_{\tau(g-1)}(P_{g-1})\\
	&\wedge-\omega_{\rho(g)}(P_{g})\wedge-\bar{\omega}_{\tau(g)}(P_{g}).
	\end{align*}
	Thus, we have $\Psi^{*}\nu^{g}=4\Phi^{*}\nu^{g}$.
	
	Since $\Psi^*\nu^g$ is non-zero, the image of $\Psi$ has to be of dimension $g$. Hence, we have $\Psi(X^g)=\Pic_{g-1}(X)$, since the image is compact and $\Pic_{g-1}(X)$ is an abelian variety. The generic degree of $\Psi$ is $4g!$ by $\Psi^{*}\nu^{g}=4\Phi^{*}\nu^{g}$.
 \end{proof}

We can now compute $H(X)$ by pulling back the integral by $\Phi$ 
\begin{align}\label{intpb}
H(X)&=\tfrac{1}{(g!)^{2}}\int_{X^{g}}\log\|\theta\|(P_{1}+\dots+P_{g}-Q)\Phi^{*}\nu^{g}
\end{align}
and by pulling back the integral by $\Psi$
\begin{align}\label{intpb2}
H(X)&=\tfrac{1}{4(g!)^{2}}\int_{X^{g}}\log\|\theta\|(2P_{1}+P_{2}+\dots+P_{g-1}-P_g)\Psi^{*}\nu^{g}\nonumber \\
&=\tfrac{1}{(g!)^{2}}\int_{X^{g}}\log\|\theta\|(2P_{1}+P_{2}+\dots+P_{g-1}-P_{g})\Phi^{*}\nu^{g},
\end{align}
see for example \cite[Theorem 14.1.1]{DFN85}.

\subsection{Integrals of the Arakelov--Green function}
We compare integrals of the Arakelov--Green function with respect to the form $\Phi^*\nu^g$ with the Zhang--Kawazumi invariant $\varphi$. First, we need a general lemma.
For $k\le g$ and points $R_{1},\dots,R_{g-k},Q\in X$ we define the map
\begin{align*}
	\Phi_{k}\colon X^{k}&\to \Pic_{g-1}(X),\\
	(P_{1},\dots,P_{k})&\mapsto P_{1}+\dots+P_{k}+R_{1}+\dots+R_{g-k}-Q.
\end{align*}
The pullbacks of $dZ_{l}$ are $\Phi_{k}^{*}(dZ_{l})=\sum_{j=1}^{k}\omega_{l}(P_{j})$.
In particular, we have $g\mu=\Phi_{1}^{*}\nu$.
For the relation to $\Phi^{*}\nu^{g}$ we have the following lemma.
\begin{Lem}\label{lemformen}
The integral of $\Phi^{*}\nu^{g}$ over the variables $P_{k+1},\dots,P_{g}$
gives the following multiple of the form $\Phi_{k}^{*}\nu^{k}$ in the remaining variables $P_{1},\dots,P_{k}$:
$$\int_{(P_{k+1},\dots,P_{g})\in X^{g-k}}\Phi^{*}\nu^{g}(P_{1},\dots,P_{g})=\tfrac{g!(g-k)!}{k!}\Phi_{k}^{*}\nu^{k}(P_{1},\dots,P_{k}).$$
\end{Lem}
\begin{proof}
	By changing coordinates in $\C^{g}$ by a matrix $B$ with $B^{2}=Y^{-1}$, we can restrict to
	the case, where $\nu$ is of the form
		$\nu=\tfrac{i}{2}\sum_{j=1}^{g}dZ_{j}\wedge d\bar{Z}_{j}$ 
		and the pullbacks $\psi_{j}=I^{*}(dZ_{j})$ form an orthonormal basis of $H^0(X,\Omega_X^1)$.
	Taking the $g$-th power of $\nu$ yields
	$$\nu^{g}=\left(\tfrac{i}{2}\right)^{g} g! \cdot dZ_{1}\wedge d\bar{Z}_{1}\wedge\dots\wedge dZ_{g}\wedge d\bar{Z}_{g}.$$
	Since $\Phi^{*}(dZ_{j})=\sum_{k=1}^{g}\psi_{j}(P_{k})$, we get by pulling back $\nu^g$ with $\Phi$
	$$\Phi^{*}\nu^{g}=\left(\tfrac{i}{2}\right)^{g}g!\sum_{\rho,\tau\in\mathrm{Sym}(g)}\bigwedge_{m=1}^{g}\psi_{m}(P_{\rho(m)})\wedge\bar{\psi}_{m}(P_{\tau(m)}).$$
	
	Since the $\psi_{j}$'s are orthonormal, only the summands with $\rho(j)=\tau(j)$ will remain after integrating over $P_{j}$. Hence, 
	we can reduce to sum over permutations in $\mathrm{Sym}(k)$:
	\begin{align*}
		&\int_{(P_{k+1},\dots,P_{g})\in X^{g-k}}\Phi^{*}\nu^{g}(P_{1},\dots,P_{g})\\
		=&(g-k)!\left(\tfrac{i}{2}\right)^{k}g!\sum_{1\le j_{1}<\dots<j_{k}\le g}\sum_{\rho,\tau\in\mathrm{Sym}(k)}\bigwedge_{m=1}^{k}\psi_{j_{m}}(P_{\rho(m)})\wedge\bar{\psi}_{j_{m}}(P_{\tau(m)}),
	\end{align*}
	where the factor $(g-k)!$ comes from the permutations of the forms in $P_{j}$ for all $k<j\le g$.
	On the other hand, the pullback of
	$$\nu^{k}=\left(\tfrac{i}{2}\right)^{k} k! \sum_{1\le j_{1}<\dots<j_{k}\le g} dZ_{j_{1}}\wedge d\bar{Z}_{j_{1}}\wedge\dots\wedge dZ_{j_{k}}\wedge d\bar{Z}_{j_{k}}$$
	yields
	$$\Phi^{*}\nu^{k}=\left(\tfrac{i}{2}\right)^{k}k!\sum_{1\le j_{1}<\dots<j_{k}\le g}\sum_{\rho,\tau\in\mathrm{Sym}(k)}\bigwedge_{m=1}^{k}\psi_{j_{m}}(P_{\rho(m)})\wedge\bar{\psi}_{j_{m}}(P_{\tau(m)}).$$
Now the lemma follows by comparing these formulas.
 \end{proof}

Next, we calculate integrals of the Arakelov--Green function.
As a consequence of Lemma \ref{lemformen}, we get for all $k$
\begin{align*}
	\tfrac{1}{(g!)^{2}}\int_{X^{g}}g(P_{k},Q)\Phi^{*}\nu^{g}(P_{1},\dots,P_{g})=\int_{X}g(P_{k},Q)\mu(P_{k})=0.
\end{align*}
For the terms $g(P_{k},P_{l})$ we get the following lemma relating their integrals
to the Zhang--Kawazumi invariant $\varphi(X)$.
\begin{Lem}\label{lemgsigmaint}
	For $k\neq l$ we have
	$$\tfrac{1}{(g!)^{2}}\int_{X^{g}}g(P_{k},P_{l})\Phi^{*}\nu^{g}(P_1,\dots,P_g)=\tfrac{1}{2g(g-1)}\cdot\varphi(X).$$
\end{Lem}
\begin{proof}	
	As in the proof of Lemma \ref{lemformen}, we change coordinates in $\C^{g}$, such that 
		$\nu=\tfrac{i}{2}\sum_{j=1}^{g}dZ_{j}\wedge d\bar{Z}_{j}$
	and the $\psi_{j}=I^{*}(dZ_{j})$ form an orthonormal basis
	of $H^{0}(X,\Omega_{X}^{1})$.
	We have
	$\nu^{2}=-\tfrac{1}{4}\sum_{p\neq q}dZ_{p}\wedge d\bar{Z}_{p}\wedge dZ_{q}\wedge d\bar{Z}_{q}$ and for the pullback by $\Phi_2$ we obtain
	$\Phi^{*}_{2}(dZ_{j})=\psi_{j}(P_{1})+\psi_{j}(P_{2})$.
	Hence, we get for the pullback of $\nu^{2}$ after some calculations
	\begin{align*}
		\Phi^{*}_{2}\nu^{2}=\tfrac{1}{2}\sum_{p\neq q}(&\psi_{p}(P_{1})\wedge\bar{\psi}_{q}(P_{1})\wedge\psi_{q}(P_{2})\wedge\bar{\psi}_{p}(P_{2})\\
		-&\psi_{p}(P_{1})\wedge\bar{\psi}_{p}(P_{1})\wedge\psi_{q}(P_{2})\wedge\bar{\psi}_{q}(P_{2})).
	\end{align*}
	Since $\mu=\tfrac{i}{2g}\sum_{j=1}^{g}\psi_{j}\wedge\bar{\psi}_{j}$, we get on the other hand
	$$\mu(P_{1})\mu(P_{2})=-\tfrac{1}{4g^{2}}\sum_{p,q=1}^{g}\psi_{p}(P_{1})\wedge\bar{\psi}_{p}(P_{1})\wedge\psi_{q}(P_{2})\wedge\bar{\psi}_{q}(P_{2})$$
	and by (\ref{hdelta})
	$$h^{2}_{\Delta}=2\mu(P_{1})\mu(P_{2})+\tfrac{1}{2}\sum_{p,q=1}^{g}\left(\psi_{p}(P_{1})\wedge\bar{\psi}_{q}(P_{1})\wedge\psi_{q}(P_{2})\wedge\bar{\psi}_{p}(P_{2})\right).$$
	Putting this together, we obtain $h^{2}_{\Delta}=\Phi^{*}_{2}\nu^{2}-2(g^{2}-1)\mu(P_{1})\mu(P_{2}).$
	By (G4) in Section \ref{invariant} the integral $\int_{X^{2}}g(P_{1},P_{2})\mu(P_{1})\mu(P_{2})$ vanishes.
	Using the defining equation (\ref{kzdef}) for $\varphi(X)$, we obtain
	$$\tfrac{1}{2g(g-1)}\varphi(X)=\tfrac{1}{2g(g-1)}\int_{X^{2}}g(P_{1},P_{2})\Phi^{*}_{2}\nu^{2}=\tfrac{1}{(g!)^{2}}\int_{X^{g}}g(P_{1},P_{2})\Phi^{*}\nu^{g},$$
	where the latter equality is due to Lemma \ref{lemformen}. Now the lemma follows by symmetry.
 \end{proof}

The function $g(\sigma(P_1+\dots+P_{g-1}),P_g)$ is defined on a dense subset of $X^g$. Hence, we can compute the integral over $X^g$ and we obtain the following relation.
\begin{Lem}\label{lemgsigmavarphi}
It holds
$$\frac{1}{(g!)^2}\int_{X^g}g(\sigma(P_1+\dots+P_{g-1}),P_g)\Phi^*\nu^g=\tfrac{1}{2g}\varphi(X).$$
\end{Lem}
\begin{proof}
Denote by $X^{(g-1)}$ the $(g-1)$-th symmetric power of $X$. We have the canonical map
$$\widetilde{\Phi}_{\Theta}\colon X^{(g-1)}\to \Theta,\quad (P_1,\dots,P_{g-1})\to P_1+\dots+P_{g-1}.$$
We denote $\widetilde{\Phi}_{\Theta}^{-1}(\Theta^{sm})=\widetilde{X^{(g-1)}}$. The map $\widetilde{\Phi}_{\Theta}$ induces an isomorphism $\widetilde{X^{(g-1)}}\cong\Theta^{sm}$. In particular, we obtain the involution $\sigma$ also on $\widetilde{X^{(g-1)}}$.
We define the map
\begin{align*}
\widetilde{\Phi}\colon \widetilde{X^{(g-1)}}\times X&\to \Pic_{g-1}(X),\\
((P_1,\dots,P_{g-1}),P_g)&\mapsto P_1+\dots+P_g-Q
\end{align*}
and the map
\begin{align*}
\widetilde{\Phi}_{\sigma}\colon \widetilde{X^{(g-1)}}\times X&\to \Pic_{g-1}(X),\\
((P_1,\dots,P_{g-1}),P_g)&\mapsto \sigma(P_1+\dots+P_{g-1})+P_g-Q.
\end{align*}
A direct computation as in the proof of Proposition \ref{propphipsi} gives $\widetilde{\Phi}^*\nu^g=\widetilde{\Phi}_{\sigma}^*\nu^g$. Since it holds $\widetilde{\Phi}=\widetilde{\Phi}_{\sigma}\circ(\sigma\times\id_X)$ and $(\sigma\times\id_X)$ is an automorphism, we can compute
\begin{align*}
&\frac{1}{(g!)^2}\int_{X^g}g(\sigma(P_1+\dots+P_{g-1}),P_g)\Phi^*\nu^g\\
=&\frac{1}{g\cdot g!}\int_{\widetilde{X^{(g-1)}}\times X}g(\sigma(P_1+\dots+P_{g-1}),P_g)\widetilde{\Phi}^*\nu^g\\
=&\frac{1}{g\cdot g!}\int_{\widetilde{X^{(g-1)}}\times X}g(P_1+\dots+P_{g-1},P_g)\widetilde{\Phi}_{\sigma}^*\nu^g\\
=&\frac{1}{g\cdot g!}\int_{\widetilde{X^{(g-1)}}\times X}g(P_1+\dots+P_{g-1},P_g)\widetilde{\Phi}^*\nu^g\\
=&\frac{1}{(g!)^2}\sum_{j=1}^{g-1}\int_{X^g}g(P_j,P_g)\Phi^*\nu^g.
\end{align*} 
By Lemma \ref{lemgsigmaint} this equals $\tfrac{1}{2g}\varphi(X)$.
 \end{proof}

\section{The hyperelliptic case}
In this section we restrict to the case of hyperelliptic Riemann surfaces. In particular, we obtain an explicit description of the invariant $\delta$ in this case. Therefore, let $X$ be any hyperelliptic Riemann surface $X$ of genus $g\ge 2$ throughout this section.
\subsection{Decomposition of theta functions}\label{secdecomposition}
We give a decomposition of $\log\|\theta\|$ into a sum of Arakelov--Green functions and a certain invariant of $X$ and we state some consequences.
\begin{Pro}\label{equhypfactor}
The function $\log\|\theta\|$ decomposes in the following way:
	$$\log\|\theta\|(P_{1}+\dots+P_{g}-Q)=S_{g}(X)+\sum_{j=1}^{g}g(P_{j},Q)+\sum_{k<l}g(\sigma(P_{k}),P_{l}).$$
\end{Pro}
\begin{proof}
	We consider 
	\begin{align}\label{constantfunction}
		\alpha(P_{1})=\log\|\theta\|(P_{1}+\dots+P_{g}-Q)-\sum_{j=1}^{g}g(P_{j},Q)-\sum_{k<l}g(\sigma(P_{k}),P_{l})
	\end{align}
	as a function in the variable $P_{1}$ by fixing the remaining points, such that each summand on
	the right hand side is well defined for at least some choices of $P_{1}$. 
	For any point $P\in X$ the divisors $P+\sigma(P)$ and $2\cdot\infty$ are linearly equivalent, see \cite[Section IIIa.§2.]{Mum84}.
	Hence, $P_{1}+\dots+P_{g}-Q$
	is effective if $P_{1}=\sigma(P_{k})$ for some $k\neq1$ or $P_{1}=Q$.
	But $\|\theta\|(P_{1}+\dots+P_{g}-Q)$
	has exactly $g$ zeros as a function in $P_{1}$, see \cite[Theorem II.3.1]{Mum83}.
	Therefore, $\alpha(P_{1})$ has no poles. Further, we get
	$$\partial\bar{\partial}\alpha(P_{1})=\pi i I^{*}\nu(P_{1})-g\pi i \mu(P_{1})=0$$
	by ($\theta$3) and (G3) in Section \ref{invariant}.
	Hence, $\alpha(P_{1})$ is a harmonic function on a compact space. Thus, it
	is constant. Analogously, we can show, that the expression (\ref{constantfunction})
	is constant as a function in $P_{2},\dots,P_{g}$ or $Q$.
	Integrating with $\mu(P_{1})\dots\mu(P_{g})$ shows that $\alpha=S_{g}(X)$ since the Arakelov--Green functions vanish
	by (G4).
 \end{proof}

As a corollary, we obtain a similar decomposition
for the function $\log\|J\|$.
\begin{Cor}\label{jfactor}
The function $\log\|J\|$ decomposes in the following way:
$$\log\|J\|(P_{1},\dots,P_{g})=B(X)+\sum_{k<l}g(P_{k},P_{l})+(g-1)\sum_{k<l}g(P_{k},\sigma(P_{l})).$$
\end{Cor}
\begin{proof}
	We apply the decomposition of $\log\|\theta\|$ in Proposition \ref{equhypfactor} to
	formula (\ref{deltagua}) and we eliminate $\delta(X)$ by (\ref{deltaab}).
	This gives the corollary.
 \end{proof}

Another application of the decomposition in Proposition \ref{equhypfactor} is the following relation of invariants of $X$.
\begin{Cor}\label{discab}
	We have
		\begin{align*}
			\log\|\varphi_{g}\|(X)=4\tbinom{2g}{g-1}\left(\tfrac{g+1}{g}B(X)-(g-1)S_{g}(X)-(g+1)\log\pi\right).
		\end{align*}
\end{Cor}
\begin{proof}
	Applying the decomposition of Proposition \ref{equhypfactor} to (\ref{equdiscweier}) gives:
	\begin{align}\label{discwkwl}
			\log\|\varphi_{g}\|(X)=4\tbinom{2g+2}{g+1}S_{g}(X)+4\tbinom{2g}{g-1}\sum_{1\le k<l\le 2g+2} g(W_{k},W_{l}).
		\end{align}
	In the same way, the decomposition of Corollary \ref{jfactor} applied to (\ref{djr}) yields:
	\begin{align}\label{discbwkwl}
		&\tbinom{2g+2}{g}g\log\pi+\tfrac{g+1}{4}\log\|\varphi_{g}\|(X)\nonumber
		\\ =&\tbinom{2g+2}{g}B(X)+\tbinom{2g}{g-2}g\sum_{1\le k<l\le 2g+2}g(W_{k},W_{l}).
	\end{align}
	Now the lemma follows by combining (\ref{discwkwl}) and (\ref{discbwkwl}).
 \end{proof}

\subsection{Comparison of integrals}\label{secintegralhyper}
In this subsection we prove the following relation between integrals of $\log\|\theta\|$.
\begin{Thm}\label{mainthm}
	It holds $(g-1)H(X)=gS_{g}(X)-S_{1}(X)$.
\end{Thm}
The idea of the proof is to apply the decomposition in Proposition \ref{equhypfactor} to the two different expressions of $H(X)$ in (\ref{intpb}) and (\ref{intpb2}). First, we prove the following two lemmas.
\begin{Lem}\label{diffS}
	We have $2S_{1}(X)=g(g-1)S_{g-1}(X)-(g+1)(g-2)S_{g}(X).$
\end{Lem}
\begin{proof}
	If we apply Proposition \ref{equhypfactor} to $\log\|\theta\|((g-k+1)P_{1}+P_{2}+\dots+P_{k}-Q)$
	and if we integrate with $\mu(P_{1})\dots\mu(P_{k})$, we get
	\begin{align}\label{equsksg}	
	S_{k}(X)=S_{g}(X)+\tfrac{(g-k)(g-k+1)}{2}\int_{X}g(\sigma(P),P)\mu(P).
	\end{align}
	If we do this for $k=1$ and for $k=g-1$, we can solve the two resulting
	equations for $S_{1}(X)$, $S_{g-1}(X)$ and $S_{g}(X)$. This yields the assertion of the lemma.
 \end{proof}

The proof shows, that we can give similar relations for any three of the $S_{j}(X)$'s,
but we will not need this.
\begin{Lem}\label{lemgsigmaint2}
For $k\neq l$ we have
$$\tfrac{1}{(g!)^{2}}\int_{X^{g}}g(\sigma(P_{k}),P_{l})\Phi^{*}\nu^{g}=\tfrac{1}{2g(g-1)}\cdot\varphi(X).$$
\end{Lem}
\begin{proof}
The involutions on $\Pic_{g-1}(X)$ and on $X$ are compatible in the sense that the divisors $\sigma(P_1+\dots+P_{g-1})$ and $\sigma(P_1)+\dots+\sigma(P_{g-1})$ are linearly equivalent. This follows, since $\sigma(P_j)+P_j$ and $2\cdot\infty$ are linearly equivalent and $(2g-2)\cdot\infty$ represents the canonical divisor class $K_X$, see \cite[Section IIIa §2.]{Mum84}.
Thus, the lemma is a direct consequence of Lemma \ref{lemgsigmavarphi}.
 \end{proof}

\begin{proof}[Proof of Theorem \ref{mainthm}.]
	We can now prove the theorem using Lemma \ref{lemgsigmaint} and Lemma \ref{lemgsigmaint2} to compute
	the terms which we get by applying the decomposition in Proposition \ref{equhypfactor} to the equations (\ref{intpb}) and (\ref{intpb2}). This yields on the one hand
	\begin{align*}
	H(X)&=\tfrac{1}{(g!)^{2}}\int_{X^{g}}\log\|\theta\|(P_{1}+\dots+P_{g}-Q)\Phi^{*}\nu^{g}=S_{g}(X)+\tfrac{1}{4}\varphi(X),
	\end{align*}
	and on the other hand
	\begin{align*}
	H(X)&=\tfrac{1}{(g!)^{2}}\int_{X^{g}}\log\|\theta\|(2P_{1}+P_{2}+\dots+P_{g-1}-P_{g})\Phi^{*}\nu^{g}\\
	&=S_{g}(X)+\int_{X}g(\sigma(P),P)\mu(P)+\left(\tfrac{g(g+1)}{2}-1\right)\tfrac{1}{2g(g-1)}\varphi(X)\\
	&=S_{g-1}(X)+\tfrac{(g+2)}{4g}\varphi(X).
	\end{align*}
	The last equality follows by (\ref{equsksg}).
	A simple computation yields
	$$H(X)=\tfrac{g+2}{2}S_{g}(X)-\tfrac{g}{2}S_{g-1}(X).$$
	Using Lemma \ref{diffS}, we can substitute $S_{g-1}(X)$ to obtain the formula in the theorem.
 \end{proof}

As a corollary of the proof we get the following explicit expression for the Zhang--Kawazumi invariant.
\begin{Cor}\label{KZexplicit}
	It holds $\varphi(X)=\tfrac{4}{g}(H(X)-S_{1}(X))$.
\end{Cor}

\subsection{Explicit formulas for the delta invariant}\label{sectionmainhyper}
Now we can deduce an explicit formula for $\delta(X)$.
As mentioned in the introduction, Bost \cite[Proposition 4]{Bos87} stated the following expression for $\delta(X)$ for $g=2$:
$$\delta(X)=-4 H(X)-\tfrac{1}{4}\log\|\Delta_{2}\|(X)-16\log(2\pi).$$
We generalize this to hyperelliptic Riemann surfaces. Furthermore, we give a relation between $\delta(X)$ and $\varphi(X)$.
\begin{Thm}\label{hypdelta}
	We have
	$$\delta(X)=-\tfrac{8(g-1)}{g}H(X)-\tbinom{2g}{g-1}^{-1}\log\|\Delta_{g}\|(X)-8g\log2\pi$$
	and
	$$\delta(X)=-24H(X)+2\varphi(X)-8g\log2\pi.$$
\end{Thm}
\begin{proof}
	First, we substitute $S_{1}(X)$ in formula (\ref{deltadejong}) by the result of Theorem \ref{mainthm}. This yields
	\begin{align}\label{deltasghdisc}
	\delta(X)=\tfrac{4(g-1)}{g}S_{g}(X)-\tfrac{4(g-1)^{2}}{g^{2}}H(X)-\tfrac{3g-1}{2gn}\log\|\Delta_{g}\|(X)-8g\log2\pi,
	\end{align}
	where we denote shortly $n=\binom{2g}{g-1}$.
	A combination of formula (\ref{deltaab}) and Corollary \ref{discab} yields
	\begin{align}\label{sgdeltadisc}
		S_{g}(X)=\tfrac{g(g+1)}{g-1}\log 2\pi+\tfrac{g}{4n(g-1)}\log\|\Delta_{g}\|(X)+\tfrac{g+1}{8(g-1)}\delta(X).
	\end{align}
	If we now insert (\ref{sgdeltadisc}) for the $S_{g}(X)$-term in (\ref{deltasghdisc}) and solve for $\delta(X)$, we obtain the first formula in the theorem. If we apply again equation (\ref{deltadejong}) to this formula,
	we can eliminate the $\log\|\Delta_g\|(X)$-term to obtain
	$$\delta(X)=-\tfrac{8(3g-1)}{g}H(X)-\tfrac{8}{g}S_1(X)-8g\log2\pi.$$
	Now Corollary \ref{KZexplicit} gives the second formula in the theorem.
 \end{proof}

We deduce the following formula for $\delta(X)$, which was also proved by de Jong in \cite[Corollary 1.8]{dJo13} by different methods.
\begin{Cor}\label{discdeltavarphi}
It holds
$$\delta(X)=-\tfrac{2(g-1)}{2g+1}\varphi(X)-\tfrac{3g}{(2g+1)}\tbinom{2g}{g-1}^{-1}\log\|\Delta_g\|(X)-8g\log 2\pi.$$
\end{Cor}
\begin{proof}
This formula directly follows by combining the two formulas in Theorem \ref{hypdelta}.
 \end{proof}

We can also conclude the following corollary about the Zhang--Kawazumi invariant $\varphi(X)$ and the modified discriminant $\|\Delta_{g}\|(X)$.
\begin{Cor}\label{discbound}
	We obtain the following explicit formula for $\varphi(X)$
	$$\varphi(X)=\tfrac{4(2g+1)}{g}H(X)-\tfrac{1}{2}\tbinom{2g}{g-1}^{-1}\log\|\Delta_{g}\|(X).$$
	In particular, we get the upper bound $\log\|\Delta_g\|(X)<-2(2g+1)\tbinom{2g}{g-1}\log 2$.
\end{Cor}
\begin{proof}
	One gets the formula for $\varphi(X)$ by comparing the two formulas in Theorem \ref{hypdelta} and solving for $\varphi(X)$, $\log\|\Delta_{g}\|(X)$ and $H(X)$.
	The bound follows by (\ref{kzbound}) and Proposition \ref{Hbound}.
 \end{proof}

Von Känel has already given an upper bound for $\|\Delta_{g}\|(X)$ in \cite[Lemma 5.4]{vKa14}.
However, our bound for $\|\Delta_g\|(X)$ is much sharper. In particular, it decreases for growing $g$.
\begin{Exam}
The formulas in Theorem \ref{hypdelta} and Corollary \ref{discbound} allow us to compute the invariants $\delta$ and $\varphi$ effectively for hyperelliptic Riemann surfaces. For any integer $n\ge 5$ consider the hyperelliptic Riemann surface $X_n$ given by the projective closure of the complex, affine curve defined by 
$$y^2=x^n+a,$$ where $a\in\C\setminus\lbrace 0\rbrace$. The isomorphism class of $X_n$ does not depend on $a$, as one sees by a change of coordinates. It is also isomorphic to the hyperelliptic Riemann surface associated to the equation $y^2+y=x^n$. Using the software Mathematica, we obtained explicit values for $\delta(X_n)$ and $\varphi(X_n)$, see Table \ref{tab1}.

The values of $H(X_5)$, $\|\Delta_2\|(X_5)$ and $\delta(X_5)$ were also computed in \cite{BMM90}. More recently, Pioline found in \cite{Pio16} formulas for the invariants $\delta$ and $\varphi$ of Riemann surfaces of genus $2$, which allow a noticeably more efficient computation of $\delta$ and $\varphi$ than our formulas.
In particular, he computed the values of $\delta(X_5)$, $\varphi(X_5)$, $\delta(X_6)$ and $\varphi(X_6)$ in \cite[Section 4.1]{Pio16}. 
\begin{table}[h]
\begin{tabular}{|c|c|c|c|c|c|}
\hline
$n$& Genus of $X_n$ & $\log\|\Delta_g\|(X_n)$ & $H(X_n)$ & $\delta(X_n)\approx$ & $\varphi(X_n)\approx$\\ \hline
$5$&$2$&$-43.14$&$-0.485~(\pm 0.003)$&$-16.68$&$0.54$\\
$6$&$2$&$-44.34$&$-0.495~(\pm 0.001)$&$-16.34$&$0.59$\\ \hline
$7$&$3$&$-239.75$&$-0.706~(\pm 0.019)$&$-24.36$&$1.40$\\
$8$&$3$&$-246.58$&$-0.719~(\pm 0.011)$&$-23.84$&$1.51$\\
\hline
\end{tabular}
\caption{Examples for $\delta$ and $\varphi$.}\label{tab1}
\end{table}
\end{Exam}

The invariant $\|\Delta_g\|(X)$ can be computed much more efficiently than the invariant $H(X)$. Moreover, the Noether formula predicts, that $\delta$ is the archimedean analogue of the logarithm of the discriminant of the finite places. Indeed, $\delta$ is essentially the logarithm of the norm of the modular discriminant for elliptic Riemann surfaces.
Hence, it may be interesting to approximate $\delta(X)$ by $\log\|\Delta_g\|(X)$ for hyperelliptic Riemann surfaces.
\begin{Cor}\label{deltainterval}
	We have the following
	relation between the invariants $\delta(X)$ and $\|\Delta_{g}\|(X)$:
	$$-\tfrac{1}{n}\log\|\Delta_{g}\|(X)+2(g-1)\log 2<\delta(X)+8g\log2\pi< \tfrac{-3g}{(2g+1)n}\log\|\Delta_{g}\|(X),$$
	where we write shortly $n=\binom{2g}{g-1}$.
\end{Cor}
\begin{proof}
	The first bound directly follows from the first formula in Theorem \ref{hypdelta} and the bound in Proposition \ref{Hbound}.
	The second inequality follows by applying the bound $\varphi(X)>0$ to the formula in Corollary \ref{discdeltavarphi}.
 \end{proof}

\subsection{A generalized Rosenhain formula}\label{secrosenhain}
Finally, we apply the decomposition in Proposition \ref{equhypfactor} to give an absolute value answer to a conjecture by Guàrdia in \cite[Conjecture 14.1]{Gua02}.
Rosenhain stated in \cite{Ros51} an identity for the case $g=2$, which can be written in our setting as
$$\|J\|(W,W')=\pi^{2}\prod_{W''\neq W,W'}\|\theta\|(W''+W-W'),$$
where $W,W'$ are two different Weierstraß points and the product runs over all Weierstraß points $W''$ different
from $W$ and $W'$. Looking for a generalization to genus $g\ge 2$, de Jong
has found formula (\ref{djr}).
We deduce the following more general result.
\begin{Thm}\label{rosgen}
	For any permutation $\tau\in\mathrm{Sym}(2g+2)$ it holds
	$$\|J\|(W_{\tau(1)},\dots,W_{\tau(g)})=\pi^{g}\prod_{j=g+1}^{2g+2}\|\theta\|(W_{\tau(1)}+\dots+W_{\tau(g)}-W_{\tau(j)}).$$
\end{Thm}
\begin{proof}
	First, we compare the applications of the decomposition in Proposition \ref{equhypfactor} to (\ref{equdiscinfty}) and (\ref{equdiscweier}).
	This yields
	$$8\tbinom{2g-1}{g-1}\sum_{1\le k<l\le 2g+1}g(W_{k},W_{l})=4\tbinom{2g}{g-1}\sum_{1\le k<l\le 2g+2}g(W_{k},W_{l}).$$
	An elementary calculation gives
	$$\sum_{1\le k<l\le 2g+1}g(W_{k},W_{l})=g\sum_{k=1}^{2g+1}g(W_{k},W_{2g+2}).$$
	The decomposition corresponding to (\ref{equdiscinfty}) is
	$$\log\|\varphi_{g}\|(X)=8\tbinom{2g+1}{g+1}S_{g}(X)+8\tbinom{2g-1}{g-1}\sum_{1\le k<l\le 2g+1} g(W_{k},W_{l}).$$
	Hence, we get
	$$8g\tbinom{2g-1}{g-1}\sum_{k=1}^{2g+1}g(W_{k},W_{2g+2})=\log\|\varphi_{g}\|(X)-8\tbinom{2g+1}{g+1}S_{g}(X).$$
	Since this does not depend on the choice of the Weierstraß point at infinity, we more generally get for a fixed $1\le m\le 2g+2$
	$$8g\tbinom{2g-1}{g-1}\sum_{\gf{k=1}{k\neq m}}^{2g+2}g(W_{k},W_{m})=\log\|\varphi_{g}\|(X)-8\tbinom{2g+1}{g+1}S_{g}(X).$$
	Summing this for $m=\tau(1),\dots,\tau(g)$ and using Corollary \ref{discab} to eliminate the term $\log\|\varphi_{g}\|(X)$, we get
	$$\sum_{j=1}^{g}\sum_{\gf{k=1}{k\neq\tau(j)}}^{2g+2}g(W_{k},W_{\tau(j)})=B(X)-(g+2)S_{g}(X)-g\log \pi.$${}
	Now we can conclude the theorem by the following calculation:
	\begin{align*}
	&\sum_{j=g+1}^{2g+2}\log\|\theta\|(W_{\tau(1)}+\dots+W_{\tau(g)}-W_{\tau(j)})\\
	=&(g+2)S_{g}(X)+(g+2)\sum_{1\le k<l\le g}g(W_{\tau(k)},W_{\tau(l)})+\sum_{j=g+1}^{2g+2}\sum_{k=1}^{g}g(W_{\tau(k)},W_{\tau(j)})\\
	=&(g+2)S_{g}(X)+g\sum_{1\le k<l\le g}g(W_{\tau(k)},W_{\tau(l)})+\sum_{j=1}^{g}\sum_{\gf{k=1}{k\neq\tau(j)}}^{2g+2}g(W_{k},W_{\tau(j)})\\
	=&(g+2)S_{g}(X)+g\sum_{1\le k<l\le g}g(W_{\tau(k)},W_{\tau(l)})+B(X)-(g+2)S_{g}(X)-g\log\pi\\
	=&\log\|J\|(W_{\tau(1)},\dots,W_{\tau(g)})-g\log\pi,
	\end{align*}
	where the last equality follows by Corollary \ref{jfactor}.
	This completes the proof.
 \end{proof}

\section{The general case}\label{Sectiongeneral}
We prove our main result in this section, see Theorem \ref{generalthm}, and we deduce some applications, for example a lower bound for $\delta$ and an explicit expression and an upper bound for the Arakelov--Green function.
\subsection{Forms on universal families}
\label{secforms}
In this subsection we discuss canonical forms on the universal family of compact and connected Riemann surfaces of fixed genus and on the universal family of principally polarised complex abelian varieties of fixed dimension with level $2$ structure. We use these forms to compute the application of $\partial\overline{\partial}$ to invariants of Riemann surfaces considered as functions on the moduli space.

Let $g\ge 3$. Denote by $\mathcal{M}_g$ the moduli space of compact and connected Riemann surfaces of genus $g$ and by $q\colon \mathcal{C}_g\to\mathcal{M}_g$ the universal family of compact and connected Riemann surfaces of genus $g$. 
The Arakelov--Green function defines a function $G\colon\mathcal{C}_g\times_{\mathcal{M}_g}\mathcal{C}_g\to\R_{\ge0}$, which again defines a metric on $\O(\Delta)$, where $\Delta\subseteq \mathcal{C}_g\times_{\mathcal{M}_g}\mathcal{C}_g$ is the diagonal. This induces a metric on the relative tangent bundle $T_{\mathcal{C}_g/\mathcal{M}_g}$, since $T_{\mathcal{C}_g/\mathcal{M}_g}$ is the normal bundle of $\Delta$. Denote by $h=c_1(\O(\Delta))$ the first Chern form of $\O(\Delta)$, that means, we have an equality 
$$\tfrac{1}{\pi i}\partial\overline{\partial}\log G=h-\delta_{\Delta}$$
of currents on $\mathcal{C}_g\times_{\mathcal{M}_g}\mathcal{C}_g$. Further, we set $e^A=h|_\Delta$, which is the first Chern form of $T_{\mathcal{C}_g/\mathcal{M}_g}$. We write $e^A_1=\int_{q} (e^A)^2$.
A direct calculation, see also \cite[Proposition 5.3]{dJo16}, gives the equality
\begin{align}\label{laplacekz}
\tfrac{1}{\pi i}\partial\overline{\partial} \varphi=\int_{q_2} h^3-e^A_1
\end{align}
of forms on $\mathcal{M}_g$, where $q_2\colon \mathcal{C}_g\times_{\mathcal{M}_g}\mathcal{C}_g\to \mathcal{M}_g$ is the canonical morphism.

Denote by $\det q_*\Omega^1_{\mathcal{C}_g/\mathcal{M}_g}$ the determinant of the Hodge bundle of $\mathcal{C}_g$ over $\mathcal{M}_g$ equipped with the metric induced by (\ref{hodgemetric}) and write $\omega_{\mathrm{Hdg}}$ for its first Chern form. The invariant $\delta$ satisfies
\begin{align}\label{laplacedelta}
\tfrac{1}{\pi i}\partial\overline{\partial}\delta=e_1^A-12\omega_{\mathrm{Hdg}},
\end{align}
see for example \cite[Section 10]{dJo16}.

Now we consider $\mathcal{M}_g[2]$, the moduli space of compact and connected Riemann surfaces of genus $g$ with level $2$ structure, see for example \cite[Section 7.4]{HL97} for a precise definition. Denote $\pi\colon\mathcal{X}_g\to\mathcal{M}_g[2]$ for the universal compact and connected Riemann surface over $\mathcal{M}_g[2]$. 
We will fix some notation.
We write $\mathcal{X}_g^n$ for the product $\mathcal{X}_g\times_{\mathcal{M}_g[2]}\dots\times_{\mathcal{M}_g[2]}\mathcal{X}_g$ with $n$ factors over $\mathcal{M}_g[2]$ and $\pi_{n}\colon \mathcal{X}_g^n\to \mathcal{M}_g[2]$ for the canonical morphism. Further, we denote $\mathcal{X}_g^{(n)}$ for the corresponding symmetric product and $\rho_n\colon \mathcal{X}^n\to\mathcal{X}^{(n)}$ for the canonical map. For any $m,n$ with $m\le n$ and pairwise different $j_1,\dots,j_m$ we denote by $pr_{j_{1},\dots,j_{m}} \mathcal{X}_g^n\to \mathcal{X}_g^m$ the projection to the $j_1$-th, $\dots$, $j_{m}$-th factors.
Moreover, we write $pr^{j_1,\dots,j_m}\colon \mathcal{X}_g^n\to \mathcal{X}_g^{n-m}$ for the projection forgetting the $j_1$-th, $\dots$, $j_m$-th factors.
We obtain forms $h$ on $\mathcal{X}_g^2$, $e^A$ on $\mathcal{X}_g$ and $e_1^A$ and $\omega_{\mathrm{Hdg}}$ on $\mathcal{M}_g[2]$ by pulling back the forms $h$, $e^A$, $e_1^A$ and $\omega_{\mathrm{Hdg}}$ defined above by the maps forgetting the level $2$ structure.

Further, we denote by $\mathcal{A}_g[2]$ the moduli space of principally polarised complex abelian varieties with level $2$ structure and we write $p\colon \mathcal{U}_g\to\mathcal{A}_g[2]$ for the universal principally polarised complex abelian variety over $\mathcal{A}_g[2]$.
There exists a $2$-form $\omega_0$ on $\mathcal{U}_g$ such that the restriction of $\omega_0$ to a principally polarised abelian variety $(A,\Theta)$ with arbitrary level $2$ structure considered as a fibre of $p$ is $\nu_{(A,\Theta)}$ and the restriction of $\omega_0$ along the zero section of $p$ is trivial, see for example \cite{HR01}.
Without risk of confusions, we also write $\omega_{\mathrm{Hdg}}$ for the first Chern form of $\det p_* \Omega^1_{\mathcal{U}_g/\mathcal{A}_g[2]}$ endowed with its $L^2$-metric. If we denote the Torelli map by $t\colon \mathcal{M}_g[2]\to \mathcal{A}_g[2]$,
it holds $t^*\omega_{\mathrm{Hdg}}=\omega_{\mathrm{Hdg}}$ as forms on $\mathcal{M}_g[2]$, see \cite[Lemme 3.2.1]{Szp85b}.

Next, we would like to define the function $\|\theta\|$ on $\mathcal{U}_g$. However, there is no canonical theta divisor for an arbitrary principally polarised complex abelian variety. 
But for any compact and connected Riemann surface $X$, there is a canonical theta divisor in $\Pic_{g-1} (X)$ given by the image of the canonical map $X^{(g-1)}\to\Pic_{g-1}(X)$. Every theta characteristic $\alpha$ of $X$ defines a theta divisor $\Theta_{\alpha}\subseteq\Jac(X)$, see (\ref{PicJac}). On $\mathcal{M}_g[2]$ we can consistently choose a theta characteristic on each curve. Hence, we obtain a theta characteristic $\alpha$ of $\mathcal{X}_g$.
For every such theta characteristic $\alpha$ of $\mathcal{X}_g$ we get a theta divisor $\Theta_{\alpha}$ in $\mathcal{U}_g$. Using the properties of uniqueness ($\theta$1)--($\theta$4) in Section \ref{secinvariantsabelian} on each fibre of $p$, we obtain a function $\|\theta_{\alpha}\|$ on $\mathcal{U}_g$. We define a metric on the line bundle $\O(\Theta_{\alpha})$ on $\mathcal{U}_g$ by $\|\theta_{\alpha}\|$. This line bundle has first Chern form $\omega_0+\tfrac{1}{2}\omega_{\mathrm{Hdg}}$, see \cite[Proposition 2]{HR01}. Hence, we obtain
\begin{align*}
\tfrac{1}{\pi i}\partial\overline\partial \log\|\theta_{\alpha}\|=\omega_0+\tfrac{1}{2}\omega_{\mathrm{Hdg}}-\delta_{\Theta_{\alpha}}.
\end{align*}

We would like to express $\tfrac{1}{\pi i}\partial\overline{\partial}H(X)$ by the forms $e_1^A$, $\int_{\pi_2} h^3$ and $\omega_{\mathrm{Hdg}}$. For this purpose, we fix a theta characteristic $\alpha$ of $\mathcal{X}_g$ and we consider the map
\begin{align*}
\gamma'\colon\mathcal{X}_g^{(g-1)}\times_{\mathcal{M}_g[2]} \mathcal{X}_g^2&\to \mathcal{U}_g,\\
[X;(P_1,\dots,P_{g-1}),P_g,P_{g+1}]&\mapsto [\Jac (X);P_1+\dots+P_g-P_{g+1}-\alpha].
\end{align*}
Further, we write $\gamma=\gamma'\circ (\rho_{g-1}\times \id_{\mathcal{X}_g^2})\colon \mathcal{X}_g^{g+1}\to\mathcal{U}_g$. Note that $\gamma'$ and $\gamma$ depend on the choice of $\alpha$. The restriction of $\gamma$ to a fibre of $\pi_{g+1}$, that means to the $(g+1)$-th power of a compact and connected Riemann surface $X$ with a level $2$ structure inducing a theta characteristic $\alpha_X$, is
$$\gamma|_{X^{g+1}}\colon X^{g+1}\to \Jac(X),\quad (P_1,\dots,P_{g+1})\mapsto P_1+\dots+P_g-P_{g+1}-\alpha_X.$$
Fixing a Riemann surface $X\in\mathcal{M}_g$ and a point $Q\in X$ we obtain a map 
$$s_Q\colon X^g\to X^{g+1},\quad(P_1,\dots,P_g)\mapsto (P_1,\dots,P_g,Q),$$
which is a section of $pr^{g+1}|_{X^{g+1}}\colon X^{g+1}\to X^g$.
As shown in Section \ref{integraltheta}, we have
$$H(X)=\tfrac{1}{(g!)^2}\int_{X^g}\log\|\theta_{\alpha_X}\|(P_1+\dots+P_g-Q-\alpha_X)((\gamma|_{X^{g+1}})\circ s_Q)^*\nu^g.$$
A direct computation yields
\begin{align*}
&\tfrac{1}{(g!)^2}\int_{X^g}\log\|\theta_{\alpha_X}\|(P_1+\dots+P_g-Q-\alpha_X)((\gamma|_{X^{g+1}})\circ s_Q)^*\nu^g\\
=&\tfrac{1}{(g!)^2}\int_{pr_{g+1}|_{X^{g+1}}}\log\|\theta_{\alpha_X}\|(P_1+\dots+P_g-P_{g+1}-\alpha_X)(\gamma|_{X^{g+1}})^*\nu^g,
\end{align*}
which shows that the latter equals $H(X)$ and it is independent of the choice of the point $P_{g+1}$.

The restriction of $\omega_{\mathrm{Hdg}}$ to a fibre of $p$ is trivial and the restriction of $\omega_0$ to a fibre of $p$ equals $\nu$. Hence, we obtain
$$H(X)=\tfrac{1}{(g!)^2}\int_{pr_{g+1}}\log\|\theta_{\alpha}\|(P_1+\dots+P_g-P_{g+1}-\alpha)\gamma^*(\omega_0+\tfrac{1}{2}\omega_{\mathrm{Hdg}})^g.$$
Using this expression, we compute $\tfrac{1}{\pi i}\partial\overline{\partial} H(X)$
by applying the Laplace operator $\tfrac{1}{\pi i}\partial\overline{\partial}$ on $\mathcal{X}_g^{g+1}$:
\begin{align*}
&\tfrac{1}{\pi i}\partial\overline{\partial}\int_{pr_{g+1}}\log\|\theta_{\alpha}\|(P_1+\dots+P_g-P_{g+1}-\alpha)\gamma^*(\omega_0+\tfrac{1}{2}\omega_{\mathrm{Hdg}})^g\\
=&\int_{pr_{g+1}}\gamma^*(\omega_0+\tfrac{1}{2}\omega_{\mathrm{Hdg}})^{g+1}-\int_{pr_{g+1}}\gamma^*\left(\delta_{\Theta_{\alpha}} \right)\gamma^*(\omega_0+\tfrac{1}{2}\omega_{\mathrm{Hdg}})^g\\
=&\int_{pr_{g+1}}\gamma^*\omega_0^{g+1}+\tfrac{g+1}{2}\int_{pr_{g+1}}\gamma^*\omega_0^g\wedge\omega_{\mathrm{Hdg}}-\int_{pr_{g+1}} \gamma^*\left(\delta_{\Theta_{\alpha}}\right)\gamma^*\omega_0^g
\\&-\tfrac{g}{2}\int_{pr_{g+1}} \gamma^*(\delta_{\Theta_{\alpha}})\gamma^*\omega_0^{g-1}\wedge \omega_{\mathrm{Hdg}}.
\end{align*}

Since the restriction of $\omega_{\mathrm{Hdg}}$ to a fibre of $pr_{g+1}$ is trivial and for any principally polarised complex abelian variety $(A,\Theta)$ it holds $\int_A \nu_{(A,\Theta)}^g=\int_{\Theta}\nu_{(A,\Theta)}^{g-1}=g!$, we get
$$\tfrac{g+1}{2}\int_{pr_{g+1}}\gamma^*\omega_0^g\wedge \omega_{\mathrm{Hdg}}=\tfrac{(g+1)\cdot(g!)^2}{2}\omega_{\mathrm{Hdg}} \quad\text{and}$$
$$\tfrac{g}{2}\int_{pr_{g+1}}\gamma^*\left(\delta_{\Theta_{\alpha}}\right)\gamma^*\omega_0^{g-1}\wedge\omega_{\mathrm{Hdg}} =\tfrac{g\cdot(g!)^2}{2}\omega_{\mathrm{Hdg}}.$$
Therefore, we obtain
\begin{align}\label{laplaceH}
\tfrac{1}{\pi i}\partial\overline{\partial} H(X)=\tfrac{1}{2}\omega_{\mathrm{Hdg}}+\tfrac{1}{(g!)^2}\left(\int_{pr_{g+1}}\gamma^*\omega_0^{g+1}-\int_{pr_{g+1}}\gamma^* \left(\delta_{\Theta_{\alpha}}\right) \gamma^*\omega_0^{g}\right).
\end{align}
Thus, we have to compute the form $\gamma^*\omega_0$.

\subsection{Deligne pairings}
\label{secdeligne}
In this subsection we introduce the Deligne pairing for hermitian line bundles as it was defined by Deligne in \cite[Section 6]{Del85} and extended to arbitrary relative dimension by Elkik in \cite{Elk89}, see also \cite{Zha96}. We will use it to study the form $\gamma^*\omega_0$.

Let $q\colon \mathfrak{X}\to S$ be a smooth, flat and projective morphism of complex manifolds of pure relative dimension $n$, and let $\mathcal{L}_0,\dots, \mathcal{L}_n$ be hermitian line bundles on $\mathfrak{X}$.
Then the line bundle $\langle \mathcal{L}_0,\dots,\mathcal{L}_n\rangle(\mathfrak{X}/S)$ is the line bundle on $S$, which is locally generated by symbols $\langle l_0,\dots,l_n\rangle$, where the ${l_j}'s$ are sections of the respective ${\mathcal{L}_j}'s$ such that their divisors have no intersection, and if for some $0\le j \le n$ and some function $f$ on $\mathfrak{X}$ the intersection $\prod_{k\neq j}\mathrm{div}(l_k)=\sum_{i}n_i Y_i$ is finite over $S$ and it has empty intersection with $\mathrm{div}(f)$, then it holds the relation
$$\langle l_0,\dots,l_{j-1},f\cdot l_j,l_{j+1},\dots,l_n\rangle=\prod_i \mathrm{Norm}_{Y_i/S}(f)^{n_i}\langle l_0,\dots, l_n\rangle.$$
By induction we define a metric on $\langle \mathcal{L}_0,\dots,\mathcal{L}_n\rangle(\mathfrak{X}/S)$ such that
$$\log\|\langle l_0,\dots,l_n\rangle\|=\log\|\langle l_0,\dots,l_{n-1}\rangle\|(\mathrm{div}(l_n))+\int_{q}\log\|l_n\|\bigwedge_{i=0}^{n-1}c_1(\mathcal{L}_i),$$
where $c_1(\mathcal{L})$ denotes the first Chern form of a hermitian line bundle $\mathcal{L}$.

In the following, we list some properties, which can be found in \cite[Section 1]{Zha96}.
The Deligne pairing is multilinear and symmetric and it satisfies
\begin{align}\label{delignechern}
c_1(\langle \mathcal{L}_0,\dots,\mathcal{L}_n\rangle)=\int_q \bigwedge_{i=0}^n c_1(\mathcal{L}_i).
\end{align}
Further, let $\phi\colon \mathfrak{X}\to \mathcal{Y}$ be a smooth, flat and projective morphism of complex manifolds over $S$ with $m_1=\dim \mathcal{Y}/S$ and $m_2=\dim\mathfrak{X}/\mathcal{Y}$, $\mathcal{K}_0,\dots,\mathcal{K}_{m_2}$ hermitian line bundles on $\mathfrak{X}$ and $\mathcal{L}_1,\dots,\mathcal{L}_{m_1}$ hermitian line bundles on $\mathcal{Y}$. We have an isometry
\begin{align}\label{doublefamily}
&\langle \mathcal{K}_0,\dots, \mathcal{K}_{m_2},\phi^*\mathcal{L}_1,\dots,\phi^*\mathcal{L}_{m_1}\rangle(\mathfrak{X}/S)\\ \nonumber
\cong&\langle\langle \mathcal{K}_0,\dots, \mathcal{K}_{m_2}\rangle(\mathfrak{X}/\mathcal{Y}),\mathcal{L}_1,\dots,\mathcal{L}_{m_1}\rangle(\mathcal{Y}/S).
\end{align}
If $m_2=1$ and $\mathcal{K}_0=\phi^*\mathcal{L}_0$ for some hermitian line bundle $\mathcal{L}_0$ on $\mathcal{Y}$, we obtain
\begin{align}\label{delignedegree}
c_1(\langle \mathcal{K}_1,\phi^*\mathcal{L}_0,\dots,\phi^*\mathcal{L}_{n-1}\rangle(\mathfrak{X}/S))
=\deg(\mathcal{K}_1)\cdot c_1(\langle \mathcal{L}_0,\dots,\mathcal{L}_{n-1}\rangle(\mathcal{Y}/S)).
\end{align}
Moreover, for general $m_2$ and hermitian line bundles $\mathcal{L}_0,\dots,\mathcal{L}_{m_1+1}$ on $\mathcal{Y}$, we have the isometry
\begin{align}\label{delignetrivial}
\langle \mathcal{K}_1,\dots\mathcal{K}_{m_2-1},\phi^*\mathcal{L}_0,\dots,\phi^*\mathcal{L}_{m_1+1}\rangle(\mathfrak{X}/S)
=\O_S.
\end{align}
We will often omit $(\mathfrak{X}/S)$ in the notation and we will also use the shorter notation $\mathcal{L}_0^{\langle n+1\rangle}=\langle\mathcal{L}_0,\dots,\mathcal{L}_0\rangle$, where the $\mathcal{L}_0$ occurs $(n+1)$ times on the right hand side.

We apply this to the family $pr^{g+2}\colon\mathcal{X}_g^{g+2}\to\mathcal{X}_g^{g+1}$.
For any positive integers $j\le k$ we have a canonical section of $pr^{k+1}$:
$$s_{k+1,j}\colon \mathcal{X}_g^{k}\to\mathcal{X}_g^{k+1},\quad [X;P_1,\dots,P_{k}]\mapsto [X;P_1,\dots,P_{k},P_j].$$
We set $\mathcal{L}=\O(2s_{g+2,1}+\dots+2s_{g+2,g}-2s_{g+2,g+1})\otimes pr_{g+2}^* T$ as a line bundle on $\mathcal{X}_g^{g+2}$, where we write $T=T_{\mathcal{X}_g/\mathcal{M}_g[2]}$ for the relative tangent bundle.
The first Chern form of $\mathcal{L}$ vanishes if we restrict to any fibre of $pr^{g+2}\colon \mathcal{X}_g^{g+2}\to\mathcal{X}^{g+1}_g$. In particular, $\mathcal{L}$ is of degree $0$ on each fibre of $pr^{g+2}$, such that we obtain a section of the Jacobian bundle $\mathcal{U}_g\times_{\mathcal{A}_g[2]}\mathcal{X}^{g+1}_g\to \mathcal{X}^{g+1}_g$.
By definition this section equals $([2]\circ\gamma)\times\id_{\mathcal{X}_g^{g+1}}$, where $[2]$ denotes the multiplication with $2$ on $\mathcal{U}_g$.
By a result due to de Jong \cite[Proposition 6.3]{dJo16}, we have $c_1(\mathcal{L}^{\langle 2\rangle})=-2([2]\circ\gamma)^*\omega_0$. Thus, we can compute
\begin{align}\label{omegapullback}
\gamma^*\omega_0=\tfrac{1}{4}([2]\circ\gamma)^*\omega_0=-\tfrac{1}{8}c_1(\mathcal{L}^{\langle 2\rangle}).
\end{align}

Let $s$ be a section of $pr^{g+2}$ and $\mathcal{L}_0$ any hermitian line bundle on $\mathcal{X}_g^{g+2}$, which is fiberwise admissible. That means that the first Chern form of the restriction of $\mathcal{L}_0$ to any fibre of $pr^{g+2}$ is a multiple of $\mu$. Then
we obtain a canonical isometry $\langle\O(s),\mathcal{L}_0\rangle\cong s^*\mathcal{L}_0$. Hence, it follows
\begin{align}\label{delignetangent}
\langle \O(s_{g+2,j}),\O(s_{g+2,j})\rangle\cong s_{g+2,j}^*\O(s_{g+2,j})\cong pr_j^*T,
\end{align}
where the last isometry follows, since $s_{g+2,j}^*\O(s_{g+2,j})$ is the pullback of the line bundle $s_{2,1}^*\O(\Delta)$ by the projection $pr_{j}\colon \mathcal{X}_g^{g+1}\to\mathcal{X}_g$ to the $j$-th factor and $s_{2,1}$ is the diagonal embedding $\mathcal{X}_g\to\mathcal{X}_g^2$, such that $s_{2,1}^*\O(\Delta)\cong T$.
Moreover, we have for $j\neq k$
\begin{align}\label{delignedelta}
\langle \O(s_{g+2,j}),\O(s_{g+2,k})\rangle\cong s_{g+2,j}^*\O(s_{g+2,k})\cong pr_{j,k}^* \O(\Delta)
\end{align}
and for all $1\le j\le g+1$
$$\langle \O(s_{g+2,j}),pr_{g+2}^*T\rangle= s_{g+2,j}^*pr_{g+2}^*T=pr_j^*T.$$
Now we can express the line bundle $\mathcal{L}^{\langle 2\rangle}$ by
\begin{align}\label{equlinebundle}
\mathcal{L}^{\langle 2\rangle}\cong\left(\bigotimes_{j=1}^g pr_j^*T\otimes \bigotimes_{j=1}^g pr_{j,g+1}^*\O(\Delta)^{\vee}\otimes\bigotimes_{j<k}^{g} pr_{j,k}^*\O(\Delta)\right)^{\otimes 8}\otimes \left(pr_{g+2}^* T\right)^{\langle 2\rangle},\end{align}
where we denote $\mathscr{L}^{\vee}$ for the dual of a line bundle $\mathscr{L}$. 
We define $\widetilde{\mathcal{L}}$ by
$$\widetilde{\mathcal{L}}=\bigotimes_{j=1}^g pr_j^*T\otimes \bigotimes_{j=1}^g pr_{j,g+1}^*\O(\Delta)^{\vee}\otimes\bigotimes_{j<k}^{g} pr_{j,k}^*\O(\Delta),$$
such that $\mathcal{L}^{\langle 2\rangle}=\widetilde{\mathcal{L}}^{\otimes 8}\otimes\left(pr_{g+2}^* T\right)^{\langle 2\rangle}$.
It holds $c_1(pr_{g+2}^* T)=pr_{g+2}^*e^A$ and hence, we deduce by (\ref{delignechern}) that $c_1\left(\left(pr_{g+2}^* T\right)^{\langle 2\rangle}\right)=e_1^A$.
Since the restriction of $e_1^A$ to a fibre of $\pi_{g+1}$ is trivial and further, the restriction of $c_1(\widetilde{\mathcal{L}})$ to a fibre $X^{g+1}$ of $\pi_{g+1}$ is equal to $-(\gamma^*\omega_0)|_{X^{g+1}}=-(\gamma|_{X^{g+1}})^*\nu_{\Jac(X)}$, we get by (\ref{omegapullback})
\begin{align*}
\int_{pr_{g+1}}\gamma^*\omega_0^{g+1}&=\left(-\tfrac{1}{8}\right)^{g+1}\int_{pr_{g+1}}c_1(\mathcal{L}^{\langle 2\rangle})^{g+1}\\
&=(-1)^{g+1}\cdot\left(\int_{pr_{g+1}}c_1(\widetilde{\mathcal{L}})^{g+1}+\tfrac{g+1}{8} \int_{pr_{g+1}}c_1(\widetilde{\mathcal{L}})^g \wedge e_1^A\right)\\
&=(-1)^{g+1}\cdot\int_{pr_{g+1}}c_1(\widetilde{\mathcal{L}})^{g+1}-\tfrac{g+1}{8}\cdot(g!)^2 e_1^A.
\end{align*}

Next, we compute the second integral in equation (\ref{laplaceH}). Denote by $\mathcal{H}_g[2]$ the moduli space of hyperelliptic Riemann surfaces of genus $g$ with level $2$ structure. We restrict for the rest of this section to the open subspace $\mathcal{M}_g'=\mathcal{M}_g[2]\setminus\mathcal{H}_g[2]$ of $\mathcal{M}_g[2]$. In particular, we sloppily write $\mathcal{X}_g$ for the restriction $\mathcal{X}_g\times_{\mathcal{M}_g[2]}\mathcal{M}_g'$, $\gamma$ instead of $\gamma|_{\mathcal{X}_g\times_{\mathcal{M}_g[2]}\mathcal{M}_g'}$, etc. The singular locus $\Theta_{\alpha}^{sing}$ of the divisor $\Theta_{\alpha}$ has codimension $4$ in the restriction of $\mathcal{U}_g/\mathcal{A}_g[2]$ to $\mathcal{M}'_g$, and its preimage under the map
$$\gamma_{\Theta_{\alpha}}\colon\mathcal{X}_g^{g-1}\to\Theta_{\alpha}\quad [X;P_1,\dots,P_{g-1}]\mapsto [\Jac(X);P_1+\dots+P_{g-1}-\alpha]$$
has codimension $2$. This follows from the proof of \cite[Proposition 11.2.8]{BL04}. Hence, the 
points $[X;P_1+\dots+P_g-P_{g+1}]\in pr_{g+1}^{-1}([X;P_{g+1}])$ satisfying
$P_1+\dots+P_{g-1}\in \Theta_{\alpha}^{sing}$ form a subspace of the fibre $pr_{g+1}^{-1}([X;P_{g+1}])$ of dimension at most $(g-2)$. Since the current $\gamma^*\left(\delta_{\Theta_{\alpha}}\right)$ restricts the space for the integration to a space of dimension $g-1$, it is enough to integrate over the subspace where $P_1+\dots+P_{g-1}\notin \Theta_{\alpha}^{sing}$. 
Write $\widetilde{\mathcal{X}_g^{(g-1)}}$ for the subspace of $\mathcal{X}_g^{(g-1)}$, where $P_1+\dots+P_{g-1}\notin \Theta_{\alpha}^{sing}$.
The canonical involution on the universal Jacobian induces an involution $\sigma$ on $\widetilde{\mathcal{X}_g^{(g-1)}}$. This is given as follows:
If $(P_1,\dots,P_{g-1})$ denotes a section of $\widetilde{\mathcal{X}_g^{(g-1)}}\to \mathcal{M}'_g$, then $\sigma(P_1,\dots,P_{g-1})$ is the unique section $(R_1,\dots,R_{g-1})$ of $\widetilde{\mathcal{X}_g^{(g-1)}}\to \mathcal{M}'_g$, such that the divisor obtained by the sum $P_1+\dots+P_{g-1}+R_1+\dots+R_{g-1}$ represents the canonical bundle on $\mathcal{X}_g/\mathcal{M}'_g$.
Now the integral can be computed as follows
\begin{align}\label{equintegralrest}
\int_{pr_{g+1}}\gamma^* (\delta_{\Theta_{\alpha}}) \gamma^*\omega_0^{g}
=
&\sum_{j=1}^g\int_{pr_{g+1}}\delta_{\lbrace P_j=P_{g+1}\rbrace}\gamma^*\omega_0^g\nonumber
\\&+\int_{pr_{g+1}}\delta_{\lbrace P_g\in\sigma(P_1,\dots,P_{g-1})\rbrace}\gamma^*\omega_0^{g},
\end{align}
where $P_g\in \sigma(P_1,\dots,P_{g-1})$ means, that $\sigma(P_1,\dots,P_{g-1})=(R_1,\dots,R_{g-1})$ and $P_g=R_j$ for some $j\le g-1$. For the terms in the sum we get by symmetry
$$\int_{pr_{g+1}}\delta_{\lbrace P_j=P_{g+1}\rbrace}\gamma^*\omega_0^g=\int_{pr_{g+1}}\delta_{\lbrace P_g=P_{g+1}\rbrace}\gamma^*\omega_0^g=\int_{pr_{g}}s_{g+1,g}^*\gamma^*\omega_0^g,$$
where the last integral is with respect to the fibres of $pr_g\colon \mathcal{X}_g^g\to\mathcal{X}_g$.
We define the following line bundle on $\mathcal{X}_g^{g}$
$$\widetilde{\mathcal{L}}'=\bigotimes_{j=1}^{g-1} pr_j^*T\otimes\bigotimes_{j<k}^{g-1} pr_{j,k}^*\O(\Delta)$$
and set $\mathcal{L}'=\widetilde{\mathcal{L}}'^{\otimes 8}\otimes T^{\langle 2\rangle}$.
Since it holds $s_{g+1,g}^*(\mathcal{L}^{\langle 2\rangle})=\mathcal{L}'$, we can conclude 
$s_{g+1,g}^*\gamma^*\omega_0=s_{g+1,g}^*(-\tfrac{1}{8}c_1(\mathcal{L}^{\langle 2\rangle}))=-\tfrac{1}{8}c_1(\mathcal{L}')$.

Thus, we compute
$$\int_{pr_g} s_{g+1,g}^*\gamma^*\omega_0^g=(-1)^g\cdot \left(\int_{pr_g}c_1(\widetilde{\mathcal{L}}')^g+\tfrac{g}{8}\int_{pr_g}c_1(\widetilde{\mathcal{L}}')^{g-1}\wedge e_1^A\right).$$
Since the restriction of $e_1^A$ to a fibre of $\pi_g\colon \mathcal{X}_g^g\to \mathcal{M}'_g$ is trivial and the restriction of $c_1(\widetilde{\mathcal{L}}')$ to a fibre $X^{g-1}$ of $pr_g$ is equal to $-\Phi_{g-1}^*\nu_{\Jac(X)}$, where $\Phi_{g-1}$ is the map
$$\Phi_{g-1}\colon X^{g-1}\to \Jac(X),\quad(P_1,\dots,P_{g-1})\mapsto P_1+\dots+P_{g-1}-\alpha_X,$$
we conclude
$$\int_{pr_g} s_{g+1,g}^*\gamma^*\omega_0^g=(-1)^g\cdot\int_{pr_g}c_1(\widetilde{\mathcal{L}}')^g- \tfrac{g}{8}\cdot (g-1)!\cdot g!\cdot e_1^A.$$

Next, we compute the second term of the right hand side of (\ref{equintegralrest}). For this purpose, we define the map
\begin{align*}
\widetilde{\sigma}\colon \widetilde{\mathcal{X}_g^{(g-1)}}\times_{\mathcal{M}'_g}\mathcal{X}_g^2&\to \widetilde{\mathcal{X}_g^{(g-1)}}\times_{\mathcal{M}'_g}\mathcal{X}_g^2,\\
((P_1,\dots,P_{g-1}),P_g,P_{g+1})&\mapsto (\sigma(P_1,\dots,P_{g-1}),P_g,P_{g+1}).
\end{align*}
If we shortly write $\gamma_{\sigma}=\gamma'\circ\widetilde{\sigma}\circ(\rho_{g-1}\times\id_{\mathcal{X}_g^2})$, we obtain
\begin{align}\label{secondintegral}
\int_{pr_{g+1}}\delta_{\lbrace P_g\in\sigma(P_1+\dots+P_{g-1})\rbrace}\gamma^*\omega_0^{g}
=\sum_{j=1}^{g-1}\int_{pr_{g+1}}\delta_{\lbrace P_j=P_g\rbrace} \gamma_{\sigma}^*\omega_0^g.
\end{align}
Since $\gamma_{\sigma}$ is the map
\begin{align*}
\gamma_{\sigma}\colon\mathcal{X}_g^{g+1}&\to \mathcal{U}_g,\\
[X;P_1,\dots,P_{g+1}]&\mapsto [\Jac (X);-P_1-\dots-P_{g-1}+P_g-P_{g+1}+\alpha_X],
\end{align*}
we again apply \cite[Proposition 6.3]{dJo16} to compute $\gamma_{\sigma}^*\omega_0=-\tfrac{1}{8}c_1(\mathcal{N})$,
where $\mathcal{N}$ denotes the line bundle $\mathcal{N}=\widetilde{\mathcal{N}}^{\otimes 8}\otimes T^{\langle 2\rangle}$
with
$$\widetilde{\mathcal{N}}=\bigotimes_{\gf{1\le j\le g+1}{j\neq g}}pr_j^*T\otimes \bigotimes_{\gf{1\le j<k\le g+1}{j\neq g,k\neq g}}pr_{j,k}^*\O(\Delta)\otimes\bigotimes_{\gf{1\le j\le g+1}{j\neq g}}pr_{j,g}\O(\Delta)^{\vee}.$$

Further, we denote the following line bundle on $\mathcal{X}_g^g$
$$\widetilde{\mathcal{N}'}=\bigotimes_{\gf{1\le j\le g}{j\neq g-1}} pr_j^* T\otimes \bigotimes_{\gf{1\le j<k\le g}{j\neq g-1, k\neq g-1}} pr_{j,k}^*\O(\Delta)$$
and set $\mathcal{N}'=\widetilde{\mathcal{N}'}^{\otimes 8}\otimes T^{\langle 2\rangle}$.
Let $s$ be the section of $pr^{g}\colon \mathcal{X}_g^{g+1}\to\mathcal{X}_g^g$ defined by
$$s\colon \mathcal{X}_g^{g}\to\mathcal{X}_g^{g+1},\quad [X;P_1,\dots,P_g]\mapsto[X;P_1,\dots,P_{g-2},P_{g-1},P_{g-1},P_g].$$
As for $\mathcal{L}'$, we obtain 
$(\gamma_{\sigma}\circ s)^*\omega_0=-\tfrac{1}{8}c_1(\mathcal{N}').$
Since $c_1(\mathcal{N}')$ does not depend on the $(g-1)$-th factor of $\mathcal{X}_g^g$, we conclude
$$\int_{pr_{g+1}}\delta_{\lbrace P_{g-1}=P_g\rbrace}\gamma_{\sigma}^*\omega_0^g=\int_{pr_g}s^*\gamma_{\sigma}^*\omega^g=-\tfrac{1}{8}\int_{pr_g}c_1(\mathcal{N}')^g=0.$$
By symmetry the entire sum in (\ref{secondintegral}) vanishes.

By (\ref{delignechern}) we obtain $\int_{pr_{g+1}}c_1(\widetilde{\mathcal{L}})^{g+1}=c_1\left(\widetilde{\mathcal{L}}^{\langle g+1\rangle}\right)$, where the Deligne pairing is with respect to the family $pr_{g+1}\colon \mathcal{X}_g^{g+1}\to\mathcal{X}_g$. In the same way, we have $\int_{pr_{g}}c_1(\widetilde{\mathcal{L}}')^{g}=c_1\left(\widetilde{\mathcal{L}}'^{\langle g\rangle}\right)$, where the Deligne pairing is with respect to the family $pr_g\colon \mathcal{X}_g^g\to \mathcal{X}_g$.
If we apply all results from this section to the equation (\ref{laplaceH}), we obtain the following relation
\begin{align}\label{laplaceHmain}
\tfrac{1}{\pi i} \partial\overline{\partial}H(X)=\tfrac{1}{2}\omega_{\mathrm{Hdg}}-\tfrac{1}{8}e_1^A+\tfrac{(-1)^{g+1}}{(g!)^2}\left(c_1(\widetilde{\mathcal{L}}^{\langle g+1\rangle})+g\cdot
c_1(\widetilde{\mathcal{L}}'^{\langle g\rangle})\right)
\end{align}
of forms on $\mathcal{M}'_g$.
Thus, we have to calculate the forms $c_1(\widetilde{\mathcal{L}}^{\langle g+1\rangle})$ and 
$c_1(\widetilde{\mathcal{L}}'^{\langle g\rangle})$.

\subsection{Graphs and Terms}
\label{secgraphs}
We compute $c_1(\widetilde{\mathcal{L}}^{\langle g+1\rangle})$ and $c_1(\widetilde{\mathcal{L}}'^{\langle g\rangle})$ by associating a graph to each term in the expansions of the powers $\widetilde{\mathcal{L}}^{\langle g+1\rangle}$ and $\widetilde{\mathcal{L}}'^{\langle g\rangle}$.
First, we define for $n\le g$ the sets
$$\mathfrak{L}_n=\lbrace pr_{j}^*T,pr_{j,g+1}^*\O(\Delta), pr_{k,l}^*\O(\Delta)| 1\le j\le n, 1\le k<l\le n\rbrace,$$
$$\mathfrak{L}'_n=\lbrace pr_{j}^*T,pr_{k,l}^*\O(\Delta)| 1\le j\le n, 1\le k<l\le n\rbrace.$$

For any $(n+1)$-tuple $(\mathcal{L}_0,\dots,\mathcal{L}_n)\in \mathfrak{L}_n^{n+1}$ the graph $\Gamma(\mathcal{L}_0,\dots,\mathcal{L}_n)$ is defined as follows: The set of vertices is $\lbrace v_1,\dots,v_n,v_{g+1}\rbrace$ and there are $(n+1)$ edges, for every $0\le j\le n$ either the loop $e_j=(v_k,v_k)$ if it holds $\mathcal{L}_j=pr_k^*T$ or the edge $e_j=(v_k,v_l)$ if $\mathcal{L}_j=pr_{k,l}^*\O(\Delta)$.
Further, we define the graph $\Gamma'(\mathcal{L}_0,\dots,\mathcal{L}_{n})$ for any $(n+1)$-tuple $(\mathcal{L}_0,\dots,\mathcal{L}_{n})\in \mathfrak{L}_n'^{n+1}$ as the graph 
$\Gamma(\mathcal{L}_0,\dots,\mathcal{L}_{n})$ without the vertex $v_{g+1}$.

\begin{Lem}\label{lemmaconstants}
There are constants $a_1, a_2, a_3,a'_1,a'_2,a'_3\in \Z$ such that we have the following equalities of forms on $\mathcal{X}_g$:
\begin{enumerate}[(a)]
\item $c_1\left(\widetilde{\mathcal{L}}^{\langle g+1\rangle}\right)=a_1\cdot \int_{\pi_2} h^3+a_2\cdot e_1^A+a_3\cdot e^A$,
\item $c_1\left(\widetilde{\mathcal{L}}'^{\langle g\rangle}\right)=a'_1\cdot \int_{\pi_2} h^3+a'_2\cdot e_1^A+a'_3\cdot e^A$.
\end{enumerate}
\end{Lem}
\begin{proof}
We only prove (a). The proof of (b) can be done in a very similar way.
By linearity, it is enough to show
$$c_1(\langle \mathcal{L}_0,\dots, \mathcal{L}_g\rangle) \in \Z\cdot\int_{\pi_2} h^3+\Z\cdot e^A + \Z\cdot e_1^A$$
for all $\mathcal{L}_0,\dots,\mathcal{L}_g\in \mathfrak{L}_g$. We denote by $\Gamma_1,\dots, \Gamma_r$ the connected components of the graph $\Gamma(\mathcal{L}_0,\dots\mathcal{L}_g)$, where $\Gamma_1$ is the connected component containing the vertex $v_{g+1}$. We write $g_j$ for the first Betti number of $\Gamma_j$ for $1\le j\le r$.
We have $\sum_{j=1}^r g_j=r$. If we had $g_j=0$ for some $j\ge 2$, we would obtain $c_1(\langle \mathcal{L}_0,\dots, \mathcal{L}_g\rangle)=0$ by (\ref{delignetrivial}). Hence, we distinguish the following two cases:
\begin{itemize}
\item
In the first case we have $g_j=1$ for all $1\le j\le r$. By symmetry we can assume that the edges contained in $\Gamma_1$ are the edges associated to $\mathcal{L}_0,\dots,\mathcal{L}_q$ and that $\mathcal{L}_0,\dots,\mathcal{L}_q\in\mathfrak{L}_q$. If $q<g$, we factorize the family $pr_{g+1}\colon \mathcal{X}^{g+1}_g\to\mathcal{X}_g$ over $pr_{1,\dots,q+1,g+1}\colon \mathcal{X}^{g+1}_g\to\mathcal{X}_g^{q+2}$.
Then we obtain by (\ref{doublefamily})
\begin{align*}
\langle \mathcal{L}_0,\dots,\mathcal{L}_g\rangle
=\langle\mathcal{L}_0,\dots,\mathcal{L}_q,\langle \mathcal{L}_{q+1},\dots,\mathcal{L}_{g}\rangle(\mathcal{X}_g^{g+1}/\mathcal{X}_g^{q+2})\rangle (\mathcal{X}_g^{q+2}/\mathcal{X}_g).
\end{align*}
If we again factorize the family $pr_{q+2}\colon \mathcal{X}^{q+2}_g\to\mathcal{X}_g$ by $pr^{q+1}\colon \mathcal{X}^{q+2}_g\to\mathcal{X}_g^{q+1}$,
we can apply (\ref{delignedegree}) to the right hand side of the equality, such that we get
$$c_1(\langle \mathcal{L}_0,\dots,\mathcal{L}_g\rangle)=\deg(\langle \mathcal{L}_{q+1},\dots,\mathcal{L}_{g}\rangle)c_1(
\langle\mathcal{L}_0,\dots,\mathcal{L}_q\rangle).$$

Hence, we only have to consider $\langle \mathcal{L}_0,\dots,\mathcal{L}_q\rangle$. The graph $\Gamma_1=\Gamma(\mathcal{L}_0,\dots,\mathcal{L}_q)$ is connected and has first Betti number $g_1=1$. If $v_j$ with $1\le j\le q$ is a vertex of $\Gamma_1$ with $\deg(v_j)=1$, then we may assume, that $e_q$ is the unique edge connected to $v_j$ and we obtain by (\ref{delignedegree})
$$c_1(\langle \mathcal{L}_0,\dots,\mathcal{L}_q\rangle(\mathcal{X}_g^{q+1}/\mathcal{X}_g))=\deg(\mathcal{L}_q)\cdot c_1(\langle \mathcal{L}_0,\dots,\mathcal{L}_{q-1}\rangle(\mathcal{X}_g^q/\mathcal{X}_g)),$$
where we factorize the family $pr_{q+1}\colon\mathcal{X}_g^{q+1}\to\mathcal{X}_g$ by $pr^j\colon\mathcal{X}_g^{q+1}\to\mathcal{X}_g^q$. The associated graph $\Gamma(\mathcal{L}_0,\dots,\mathcal{L}_{q-1})$ is obtained from $\Gamma_1$ by removing the vertex $v_j$ and the edge $e_q$.

If $v_j$ with $1\le j\le q$ is a vertex of $\Gamma_1$ with $\deg(v_j)=2$, we may assume, that $e_{q-1}$ and $e_q$ are the edges connected to $v_j$. Now we get by (\ref{doublefamily})
\begin{align*}
&\langle\mathcal{L}_0,\dots,\mathcal{L}_q\rangle(\mathcal{X}_g^{q+1}/\mathcal{X}_g)\\
=&\langle\mathcal{L}_0,\dots,\mathcal{L}_{q-2},\langle \mathcal{L}_{q-1},\mathcal{L}_q\rangle(\mathcal{X}_g^{q+1}/\mathcal{X}^q_g)\rangle (\mathcal{X}_g^q/\mathcal{X}_g),
\end{align*}
where we again factorize the family $pr_{q+1}\colon\mathcal{X}_g^{q+1}\to\mathcal{X}_g$ by the projection $pr^j\colon\mathcal{X}_g^{q+1}\to\mathcal{X}_g^q$. The line bundles $\mathcal{L}_{q-1}$ and $\mathcal{L}_q$ have to be equal to $pr_{k_1,j}^*\O(\Delta)$, respectively $pr_{k_2,j}^*\O(\Delta)$, for some numbers $k_1, k_2\in\lbrace 1,\dots,q,g+1\rbrace$.
Hence, we have by a similar computation as for (\ref{delignetangent}) and (\ref{delignedelta})
$$\langle\mathcal{L}_{q-1},\mathcal{L}_q\rangle=\langle pr_{k_1,j}^*\O(\Delta), pr_{k_2,j}^*\O(\Delta)\rangle=pr_{k_1,k_2}^*\O(\Delta)$$
if $k_1\neq k_2$, and
$$\langle\mathcal{L}_{q-1},\mathcal{L}_q\rangle=\langle pr_{k_1,j}^*\O(\Delta), pr_{k_1,j}^*\O(\Delta)\rangle=pr_{k_1}^*T$$
if $k_1=k_2$. Thus, the associated graph $\Gamma(\mathcal{L}_0,\dots,\mathcal{L}_{q-2},\langle\mathcal{L}_{q-1},\mathcal{L}_q\rangle)$ is well defined and it arises from $\Gamma_1$ by removing the vertex $v_j$ and replacing the edges $e_{q-1}$ and $e_q$ by an edge connecting the two not necessarily different neighbours of $v_j$. Therefore, we can assume that the vertices $v_1,\dots, v_q$ of $\Gamma_1$ have degree at least $3$. This is only possible if $\Gamma_1$ only consists of the vertex $v_{g+1}$ and the loop $e_0=(v_{g+1},v_{g+1})$: 
$$\xygraph{!{(0,0)}*+!(-.3,.11){\bullet v_{g+1}}="a" "a"-@`{p+(-1,1.5),p+(-1,-1.5)} "a"_{e_0}}.$$
This means, that we always have in this case
$$c_1(\langle\mathcal{L}_0,\dots,\mathcal{L}_g\rangle)=n\cdot c_1(T)=n\cdot e^A$$
for some $n\in \Z$.
\item
The second case can be handled very similarly to the first one.
Here, we have $g_1=0$, $g_k=2$ for some $2\le k\le r$ and $g_j=1$ for $j\notin \lbrace 1,k\rbrace$. 
Again, we may assume by symmetry that the edges contained in $\Gamma_k$ are the edges associated to $\mathcal{L}_0,\dots,\mathcal{L}_q$ and that $\mathcal{L}_0,\dots,\mathcal{L}_q\in \mathfrak{L}'_q$. As in the first case we get by (\ref{doublefamily}) and (\ref{delignedegree})
$$c_1(\langle \mathcal{L}_0,\dots,\mathcal{L}_g\rangle)=\deg(\langle \mathcal{L}_{q+1},\dots,\mathcal{L}_g\rangle)\cdot c_1(\langle \mathcal{L}_0,\dots,\mathcal{L}_q\rangle).$$
By the same arguments as in the first case, we can reduce to the case, where the vertices $v_1,\dots,v_q$ of $\Gamma_k$ have degree at least $3$. But these are all vertices of $\Gamma_k$. Moreover, $\Gamma_k$ is connected and its first Betti number is $2$. Hence, there are up to permutations only the following possibilities for $\Gamma_k$:
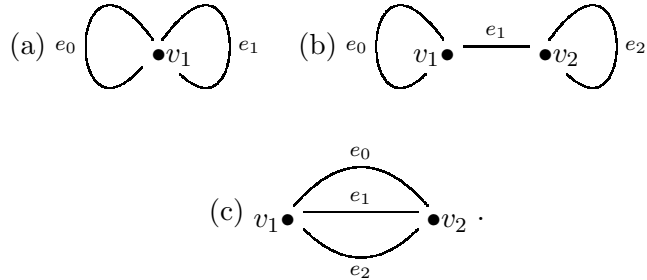
$$\text{(a)}\xygraph{!{(0,0)}*+!(-.15,.11){\bullet v_{1}}="a" "a"-@`{p+(-1,1.5),p+(-1,-1.5)} "a"_{e_0} "a"-@`{p+(1,1.5),p+(1,-1.5)} "a"^{e_1}} \quad\text{(b)}\xygraph{!{(0,0)}*+!(.15,.11){v_1 \bullet}="a" !{(1.0,0)}*+!(-.15,.11){\bullet v_{2}}="b" "a"-@`{p+(-1,1.5),p+(-1,-1.5)} "a"_{e_0} "a"-"b"^{e_1} "b"-@`{p+(1,1.5),p+(1,-1.5)} "b"^{e_2} }$$  $$\quad\text{(c)}\xygraph{!{(0,0)}*+!(.15,.11){v_1 \bullet}="a" !{(1.5,0)}*+!(-.15,.11){\bullet v_{2}}="b" "a"-@/^0.6cm/"b"^{e_0} "a"-"b"^{e_1} "a"-@/_0.6cm/"b"_{e_2} }.$$

The graph (a) corresponds to $q=1$ and $\mathcal{L}_0=\mathcal{L}_1=T$. In this case we have $c_1(\langle\mathcal{L}_0,\mathcal{L}_1\rangle)=e_1^A$ by (\ref{delignechern}). The graph (b) corresponds to $q=2$, $\mathcal{L}_0=pr_1^*T$, $\mathcal{L}_1=pr_{1,2}^*\O(\Delta)$ and $\mathcal{L}_2=pr_2^*T$.
We can apply (\ref{doublefamily}) to obtain
\begin{align*}
\langle pr_1^*T,pr_{1,2}^*\O(\Delta), pr_2^*T\rangle=\langle\langle  pr_1^*T,pr_{1,2}^*\O(\Delta)\rangle, 
T\rangle,
\end{align*}
where we factorize the family $\pi_2\colon \mathcal{X}_g^2\to\mathcal{M}'_g$ by $pr_{2}\colon\mathcal{X}_g^2\to\mathcal{X}_g$. Since
$$\langle pr_1^*T,pr_{1,2}^*\O(\Delta)\rangle= T,$$
we again conclude $c_1(\langle\mathcal{L}_0,\mathcal{L}_1,\mathcal{L}_2\rangle)=e_1^A$.
Finally, the graph (c) corresponds to $q=2$ and $\mathcal{L}_0=\mathcal{L}_1=\mathcal{L}_2=pr_{1,2}^*\O(\Delta)$. Hence, we can again apply (\ref{delignechern}) to obtain
$c_1(\langle\mathcal{L}_0,\mathcal{L}_1,\mathcal{L}_2\rangle)=\int_{\pi_2} h^3$.
\end{itemize}
 \end{proof}

By the lemma and formula (\ref{laplaceHmain}), we have the following identity of forms on $\mathcal{M}_g'$
\begin{align}\label{laplaceH2}
\tfrac{1}{\pi i}\partial\overline{\partial}H(X)=\tfrac{1}{2}\omega_{\mathrm{Hdg}}+b_1\cdot \int_{\pi_2} h^3+b_2\cdot e_1^A +b_3 \cdot e^A,
\end{align}
where $X$ is non-hyperelliptic and $b_1,b_2,b_3\in\Q$ are constants depending only on $g$. 
Since $H(X)$ does not depend on the choice of the point $P_{g+1}$, we must have $b_3=0$.
Comparing the lemma with equation (\ref{laplaceHmain}), we obtain $b_1=\tfrac{(-1)^{g+1}}{(g!)^2}(a_1+g\cdot a'_1)$ and $b_2=\tfrac{(-1)^{g+1}}{(g!)^2}(a_2+g\cdot a'_2)-\tfrac{1}{8}$. Next, we compute the numbers $a_1,a'_1,a_2$ and $a'_2$.
\begin{Lem}\label{lema1}
We have $a_1=0$ and $a'_1=\tfrac{g!(g-1)!}{12}(-1)^{g-1}$ for the constants $a_1$ and $a'_1$ in Lemma \ref{lemmaconstants}. In particular, it holds $b_1=\tfrac{1}{12}$.
\end{Lem}
\begin{proof}
We will weight and count the graphs associated to the terms in the expansions of the powers $c_1(\widetilde{\mathcal{L}}^{\langle g+1\rangle})$ and $c_1(\widetilde{\mathcal{L}}'^{\langle g\rangle})$. We first consider the case $c_1(\widetilde{\mathcal{L}}^{\langle g+1\rangle})$.
By the arguments of the proof of Lemma \ref{lemmaconstants} we have $c_1(\langle\mathcal{L}_0,\dots,\mathcal{L}_g\rangle)=a\cdot\int_{\pi_2} h^3$ for some $\mathcal{L}_0,\dots,\mathcal{L}_g\in\mathfrak{L}_g'$ and $a\neq 0$ only if the associated graph $\Gamma(\mathcal{L}_0,\dots,\mathcal{L}_g)$ has a subgraph $\Gamma_0$ containing two different vertices $v_j,v_k$ which are connected by three disjoint paths not involving the vertex $v_{g+1}$. 

Hence, we can compute $a_1$ by
\begin{align}\label{a1formula}
a_1=\sum_{k=3}^{g+1}A_{k}\cdot B_{g,g+1-k}\cdot\tbinom{g}{k-1}\cdot\tbinom{g+1}{k},
\end{align}
where $A_{k}$ is the number of $k$-tuples $(\mathcal{L}_0,\dots,\mathcal{L}_{k-1})\in\mathfrak{L}_{k-1}'^k$ such that the associated graph $\Gamma'(\mathcal{L}_0,\dots,\mathcal{L}_{k-1})$ is connected and has two vertices of degree $3$, which are connected by $3$ paths and all other vertices have degree $2$, that means it is of the form $\Gamma_0$ described above. To define $B_{g,k}$ we introduce another graph $\Gamma_{g}(\mathcal{L}_0,\dots,\mathcal{L}_{q})$ for any $(q+1)$-tuple $(\mathcal{L}_0,\dots,\mathcal{L}_{q})\in \mathfrak{L}_{g}^{q+1}$, which is defined as follows: The set of vertices is $\lbrace v_1,\dots,v_{g+1}\rbrace$ and there are $q+1$ edges associated to $\mathcal{L}_0,\dots,\mathcal{L}_q$ in the same way as for $\Gamma$.
Now $B_{g,k}$ is the sum of the weights $w(\mathcal{L}_0,\dots,\mathcal{L}_{k-1})$ associated to all $k$-tuples $(\mathcal{L}_0,\dots,\mathcal{L}_{k-1})\in\mathfrak{L}_{g}^k$, where the weight is defined as follows:
$w(\mathcal{L}_0,\dots,\mathcal{L}_{k-1})$ is $0$ if for some $l\le k$ and some subset $\lbrace j_1,\dots, j_l\rbrace\subseteq \lbrace 0,\dots,k-1\rbrace$ of cardinality $l$ at most $l-1$ of the vertices $v_{1},\dots,v_{k}$ of the graph
$\Gamma_g(\mathcal{L}_{j_1},\dots,\mathcal{L}_{j_l})$ has non-zero degree and otherwise it has the value $(2-2g)^{b_1}\cdot (-1)^{\deg(v_{g+1})}$, where $b_1$ denotes the first betti number of $\Gamma_g(\mathcal{L}_0,\dots,\mathcal{L}_{k-1})$. Note, that if we define another weight $w'$ in exactly the same way except that we replace the vertices $v_1,\dots,v_k$ by $v_{g-k+1},\dots,v_g$, we will obtain the same number $B_{g,k}$ by symmetry.
In particular, if the graph $\Gamma(\mathcal{L}_0,\dots,\mathcal{L}_{k-1})$ is of the form $\Gamma_0$ and if we have $w'(\mathcal{L}_k,\dots,\mathcal{L}_g)=0$, then it follows $c_1(\langle \mathcal{L}_0,\dots,\mathcal{L}_g\rangle)=0$ by (\ref{delignetrivial}). But for simpler notations we will work with the weight $w$.

We obtain the binomial coefficient $\binom{g}{k-1}$ in formula (\ref{a1formula}) by choosing $k-1$ of the $g$ vertices $\lbrace v_1,\dots, v_g\rbrace$ of the associated graph $\Gamma(\mathcal{L}_0,\dots,\mathcal{L}_g)$ to be the vertices of $\Gamma_0$ and the binomial coefficient $\binom{g+1}{k}$ by choosing the position of the $k$-tuple associated to the graph $\Gamma_0$ in the whole $(g+1)$-tuple $(\mathcal{L}_0,\dots,\mathcal{L}_g)$.

One can check the correctness of formula (\ref{a1formula}) by the methods of proof of Lemma \ref{lemmaconstants}: Every circle in the associated graph of a tuple $(\mathcal{L}_0,\dots,\mathcal{L}_g)$ outside of $\Gamma_0$ can be reduced to a loop, which is associated to a line bundle $pr_j^*T$ having degree $\deg(pr_j^*T)=2-2g$. Moreover, every line bundle of the form $pr_{j,g+1}^*\O(\Delta)$ occurs as its dual in $\widetilde{\mathcal{L}}$. Thus, we have to multiply with $\deg(pr_{j,g+1}^*\O(\Delta)^{\vee})=-1$ for every line bundle of this form in the tuple $(\mathcal{L}_0,\dots,\mathcal{L}_g)$. This justifies the formula of the weight $w$ and hence, formula (\ref{a1formula}) follows by elementary combinatorics and using (\ref{delignedegree}) inductively.

\begin{clm}
It holds $A_{k}=\binom{k-1}{2}\cdot\frac{k!(k-1)!}{12}$ for $3\le k\le g+1$.
\end{clm}
\begin{proof}[Proof of the claim]
Let $k=3$. Since $(pr_{1,2}^*\O(\Delta),pr_{1,2}^*\O(\Delta),pr_{1,2}^*\O(\Delta))$ is the only tuple with the desired property, we have $A_{3}=1$. Hence, we can assume $k\ge 4$. Write $(\mathcal{L}_0,\dots,\mathcal{L}_{k-1})$ for a tuple of the desired form and $v_{j_1},v_{j_2}$ for the vertices of the associated graph, which have degree $3$. There are $\binom{k-1}{2}$ possible choices for $v_{j_1}$ and $v_{j_2}$.
Further, there are $(k-3)!\binom{k-1}{2}$ choices to order the remaining vertices and to divide them in $3$ groups representing the $3$ paths from $v_{j_1}$ to $v_{j_2}$. But here, the association of groups of vertices to paths is ordered. Hence, we have to divide by the possibilities to order them.
Also, we have to multiply by the number of possibilities to order the line bundles in the tuple $(\mathcal{L}_0,\dots,\mathcal{L}_{k-1})$. We distinguish the following two cases:
\begin{itemize}
\item The lengths of two paths from $v_{j_1}$ to $v_{j_2}$ are $1$. 
Then two of the groups of vertices are empty and we have $3$ choices for the non-empty group. But on the other side, we have $\frac{k!}{2}$ possibilities to order the line bundles in the tuple $(\mathcal{L}_0,\dots,\mathcal{L}_{k-1})$, since two of them are equal. Hence, in this case we have to multiply by $\frac{k!}{6}$.
\item Otherwise, there are two paths from $v_{j_1}$ to $v_{j_2}$ with length at least $2$. Then one can distinguish all $3$ groups of vertices. Therefore, there are $3!=6$ possibilities to order them. Further, we have $k!$ possibilities to order the line bundles in the tuple $(\mathcal{L}_0,\dots,\mathcal{L}_{k-1})$, since all of them are different. Hence, in this case we also have to multiply by $\frac{k!}{6}$.
\end{itemize}
Thus, we conclude $A_{k}=\binom{k-1}{2}\cdot(k-3)!\cdot\binom{k-1}{2}\cdot\frac{k!}{6}=\binom{k-1}{2}\cdot\frac{k!(k-1)!}{12}$.
 \end{proof}

\begin{clm}\label{claim2}
The number $B_{g,k}$ is given by $B_{g,k}=(-1)^k \cdot k!\cdot \frac{g!}{(g-k)!}$.
\end{clm}
\begin{proof}[Proof of the claim]
We prove this by induction over $k$. If $k=0$, we only have the empty tuple, which is weighted by $1$. Hence, we assume $k>0$ and that the claim is true for $k-1$.
For any $q\le g$ denote by $\Z[\mathfrak{L}_g^q]$ the free abelian group over the set of $q$-tuples of elements in $\mathfrak{L}_g$. Write $w_q\in \Z[\mathfrak{L}_g^q]$ for the distinguished element
$$w_q=\sum_{(\mathcal{L}_0,\dots,\mathcal{L}_{q-1})\in \Z[\mathfrak{L}_g^q]} w(\mathcal{L}_0,\dots,\mathcal{L}_{q-1})\cdot (\mathcal{L}_0,\dots,\mathcal{L}_{q-1}).$$
For any element $c\in\Z[\mathfrak{L}_g^q]$ we define the degree $\deg(c)\in \Z$ to be the sum of its coefficients. Then we have $B_{g,q}=\deg(w_q)$, and the induction hypothesis states
$$\deg(w_{k-1})=(-1)^{k-1}\cdot (k-1)!\cdot \tfrac{g!}{(g-k+1)!}.$$
We have to prove $\deg(w_k)=-k(g-k+1)\cdot\deg(w_{k-1})$.
We distinguish the following cases to extend a non-zero weighted $(k-1)$-tuple to a non-zero weighted $k$-tuple:
\begin{enumerate}[$(1)_{k}$]
\item
$(\mathcal{L}_0,\dots,\mathcal{L}_{k-2})\to(\mathcal{L}_0,\dots,\mathcal{L}_{k-2}, pr_k^*T)$,
\item
$(\mathcal{L}_0,\dots,\mathcal{L}_{k-2})\to(\mathcal{L}_0,\dots,\mathcal{L}_{k-2}, pr_{k,l}^*\O(\Delta))$ for any $1\le l\le g$ with $l\neq k$ and $pr_{k,l}^*\O(\Delta)\notin \lbrace \mathcal{L}_0,\dots,\mathcal{L}_{k-2}\rbrace$,
\item
$(\mathcal{L}_0,\dots,\mathcal{L}_{k-2})\to(\mathcal{L}_0,\dots,\mathcal{L}_{k-2}, pr_{k,l}^*\O(\Delta))$ for any $1\le l\le g$ with $l\neq k$ and $pr_{k,l}^*\O(\Delta)\in \lbrace \mathcal{L}_0,\dots,\mathcal{L}_{k-2}\rbrace$,
\item
$(\mathcal{L}_0,\dots,\mathcal{L}_{i-1},pr_{l}^*T,\mathcal{L}_{i+1},\dots,\mathcal{L}_{k-2})$ \\$\to(\mathcal{L}_0,\dots,\mathcal{L}_{i-1},pr_{l,k}^*\O(\Delta),\mathcal{L}_{i+1},\dots,\mathcal{L}_{k-2},pr_{k,l}^*\O(\Delta))$ for any $l\neq k$,
\item
$(\mathcal{L}_0,\dots,\mathcal{L}_{i-1},pr_{l,m}^*\O(\Delta),\mathcal{L}_{i+1},\dots,\mathcal{L}_{k-2})$ \\$\to(\mathcal{L}_0,\dots,\mathcal{L}_{i-1},pr_{l,k}^*\O(\Delta),\mathcal{L}_{i+1},\dots,\mathcal{L}_{k-2},pr_{k,m}^*\O(\Delta))$ for any $l$ and $m$ different from $k$ with $l\neq m$,
\item
$(\mathcal{L}_0,\dots,\mathcal{L}_{k-2})\to(\mathcal{L}_0,\dots,\mathcal{L}_{k-2}, pr_{k,g+1}^*\O(\Delta))$.
\end{enumerate}

We additionally consider the extensions $(1)_j$-$(6)_j$ for any $1\le j\le k$ which coincide with $(1)_k$-$(6)_k$ with the change that the new line bundle occurs in the $j$-th factor instead of the last factor. In this way, we obtain all $k$-tuples of non-zero weight as extensions of $(k-1)$-tuples of non-zero weight. However, the same $k$-tuple can be constructed by different extensions. 
Hence, we have to count them with suitable multiplicities.

For a $(k-1)$-tuple $(\mathcal{L}_0,\dots,\mathcal{L}_{k-2})\in\mathfrak{L}_g^{k-1}$ we denote by $m=\deg(v_{k})$ the degree of the vertex $v_{k}$ of the associated graph $\Gamma_g(\mathcal{L}_0,\dots,\mathcal{L}_{k-2})$.
Let $w'_k$ be the element in $\Z[\mathcal{L}_g^k]$, which we obtain by taking for all $(k-1)$-tuples $(\mathcal{L}_0,\dots,\mathcal{L}_{k-2})\in\mathfrak{L}_g^{k-1}$ and all $j\le k$
\begin{itemize}[-]
\item the extensions $(1)_j$ times $(2-2g)\cdot w(\mathcal{L}_0,\dots,\mathcal{L}_{k-2})$,
\item the extensions $(2)_j$ times $(1-m) \cdot w(\mathcal{L}_0,\dots,\mathcal{L}_{k-2})$, 
\item the extensions $(3)_j$ times $(g-m) \cdot w(\mathcal{L}_0,\dots,\mathcal{L}_{k-2})$, 
\item the extensions $(4)_j$ and $(5)_j$ times $w (\mathcal{L}_0,\dots,\mathcal{L}_{k-2})$ and 
\item the extensions $(6)_j$ times $(m-1) \cdot w(\mathcal{L}_0,\dots,\mathcal{L}_{k-2})$.
\end{itemize}

Next, we prove $w'_k=w_k$. Let $(\mathcal{L}_0,\dots,\mathcal{L}_{k-1})\in\mathfrak{L}_g^{k}$ be a $k$-tuple with non-zero weight. Denote by $m'=\deg(v_k)$ the degree of the vertex $v_k$ in the associated graph $\Gamma=\Gamma(\mathcal{L}_0,\dots,\mathcal{L}_{k-1})$. If $\Gamma$ has a loop at $v_{k}$, the tuple $(\mathcal{L}_0,\dots,\mathcal{L}_{k-1})$ can only be obtained by an extension of kind $(1)_j$ from a non-zero weighted $(k-1)$-tuple. Hence, we can assume, that $\Gamma$ has no loop at $v_k$. Since every extension $(1)_j$-$(6)_j$ only adds edges connected to $v_k$, it is enough to consider the connected component $\Gamma_1$ of the graph $\Gamma$, which contains the vertex $v_k$. Its first Betti number $b_1(\Gamma_1)$ is either $1$ and all its vertices form a subset of $\lbrace v_1,\dots,v_k\rbrace$ or $b_1(\Gamma_1)=0$ and $\Gamma_1$ additionally contains one vertex $v_i$ with $k<i\le g+1$. More precisely, we distinguish the following four cases, where we denote by $Z_{1}$ a connected subgraph of $\Gamma_1$ with first Betti number $b_1(Z_{1})=1$ and by $\Gamma_{1,1}, \dots, \Gamma_{1,m'}$ connected subgraphs of $\Gamma_1$, which are trees. The sets of vertices of $Z_1, \Gamma_{1,1}, \dots, \Gamma_{1,m'-1}$ are assumed to be non-empty subsets of $\lbrace v_1, \dots, v_{k-1}\rbrace$ and the set of vertices of $\Gamma_{1,m'}$ is assumed to be a subset of $\lbrace v_1, \dots, v_{k-1},v_i\rbrace$, which has to contain $v_i$.

\begin{itemize}
\item In the first case, we consider $\Gamma_1$ with $b_1(\Gamma_1)=1$ and $\Gamma_1$ has the structure
$$\xygraph{!{(-0.5,0)}*+!(0.15,0){v_k \bullet}="a" 
!{(-1,1)}*+{\text{\fbox{$Z_{1}$}}}="b"
!{(0,1)}*+{ \text{\fbox{$\Gamma_{1,1}$}}}="c"
!{(1,1)}*+{\dots}
!{(2,1)}*+{ \text{\fbox{$\Gamma_{1,m'-1}$}}}="d"
"a"-"b"^e
"a"-"c"
"a"-"d"
}.$$
These graphs can be obtained by an extension of the form $(2)_j$, where the edge $e$ is added, or by an extension of the form $(5)_j$, where an edge from $Z_1$ to $\Gamma_{1,l}$ for some $l\le m'-1$ is replaced by the edges from $Z_1$ to $v_k$ and from $\Gamma_{1,l}$ to $v_k$. In both cases the weight of the tuple is preserved by the extension and we have $m=m'-1$ for the extension of the form $(2)_j$. Here, the $j$ is unique by the choice of the $k$-tuple and the kind of extension. All in all, the coefficient of the $k$-tuple $(\mathcal{L}_0,\dots,\mathcal{L}_{k-1})$ in $w'_k$ equals
$$((1-(m'-1))+(m'-1)) \cdot w(\mathcal{L}_0,\dots,\mathcal{L}_{k-1}) = w(\mathcal{L}_0,\dots,\mathcal{L}_{k-1}).$$
\item Next, we consider graphs $\Gamma_1$ with $b_1(\Gamma_1)=1$ and having the structure
$$\xygraph{!{(-0.5,0)}*+!(0.15,0){v_k \bullet}="a" 
!{(-1.2,0.8)}*+{\bullet}="f"
!{(-0.8,0.8)}*+{\bullet}="e"
!{(-1,1)}*+{\text{\fbox{$\Gamma_{1,1}$}}}
!{(0.5,1)}*+{ \text{\fbox{$\Gamma_{1,2}$}}}="c"
!{(1.5,1)}*+{\dots}
!{(2.5,1)}*+{ \text{\fbox{$\Gamma_{1,m'-1}$}}}="d"
"a"-"e"_{e'}
"a"-"f"^{e}
"a"-"c"
"a"-"d"
}.$$
Here, there are two edges from $v_k$ to the subgraph $\Gamma_{1,1}$ landing in two different vertices.
These graphs can be obtained by extensions of the form $(2)_j$, where the edge $e$ or the edge $e'$ is added, or by an extension of the form $(5)_j$, where an edge from one of the neighbours of $v_k$ in the graph $\Gamma_1$ to another is replaced by two edges connecting each of these two neighbours with $v_k$, where at least one of the neighbours has to be contained in $\Gamma_{1,1}$.
Hence, there are $2$ possible extensions of the form $(2)_j$, where $m=m'-1$, and $(m'-1)+(m'-2)$ possible extensions of the form $(5)_j$. The weight of the tuple is preserved by these extensions. The $j$ is unique by the choice of the $k$-tuple and the kind of extension. Hence, the coefficient of the $k$-tuple $(\mathcal{L}_0,\dots,\mathcal{L}_{k-1})$ in $w'_k$ is given by
$$(2(1-(m'-1))+(m'-1+m'-2))w(\mathcal{L}_0,\dots,\mathcal{L}_{k-1})= w(\mathcal{L}_0,\dots,\mathcal{L}_{k-1}).$$
\item
We have a third case with $b_1(\Gamma_1)=1$, where $\Gamma_1$ is of the form
$$\xygraph{!{(-0.5,0)}*+ !(0.15,0) {v_k \bullet}="a" 
!{(-1,0.8)}*+{\bullet}="e"
!{(-1,1)}*+{\text{\fbox{$\Gamma_{1,1}$}}}
!{(0.5,1)}*+{ \text{\fbox{$\Gamma_{1,2}$}}}="c"
!{(1.5,1)}*+{\dots}
!{(2.5,1)}*+{ \text{\fbox{$\Gamma_{1,m'-1}$}}}="d"
"a"- @/^0.2cm/ "e"^e
"a"-@/_0.2cm/"e"_{e'}
"a"-"c"
"a"-"d"
}.$$
Here, there are two edges from $v_k$ to the subgraph $\Gamma_{1,1}$ landing in the same vertex. These graphs can be obtained by an extension of the form $(3)_j$, where the edge $e$ or the edge $e'$ is added, by an extension of the form $(4)_j$, where a loop is replaced by the edges $e$ and $e'$, or by an extension of the form $(5)_j$, where an edge from one of the neighbours of $v_k$ in the graph $\Gamma_1$ to another is replaced by two edges connecting each of the two neighbours with $v_k$, where one of the neighbours has to be the one in $\Gamma_{1,1}$. The extensions of the form $(3)_j$ and $(5)_j$ multiply the weight by $(2-2g)$, the extensions of the form $(4)_j$ preserve the weight and for the extensions of the form $(3)_j$ we have $m=m'-1$. Since each of these extensions adds at least one of the edges $e$ and $e'$, which represent two isomorphic line bundles in the tuple $(\mathcal{L}_0,\dots,\mathcal{L}_{k-1})$, there are two choices for the $j$. Therefore, the coefficient of the $k$-tuple $(\mathcal{L}_0,\dots,\mathcal{L}_{k-1})$ in $w'_k$ is
$$2\left(\tfrac{g-(m'-1)}{2-2g}+1+\tfrac{m'-2}{2-2g}\right)\cdot w(\mathcal{L}_0,\dots,\mathcal{L}_{k-1}) =w(\mathcal{L}_0,\dots,\mathcal{L}_{k-1}).$$
\item
Finally, it remains the case $b_1(\Gamma_1)=0$ and $\Gamma_1$ is of the form
$$\xygraph{!{(-0.5,0)}*+ !(0.15,0) {v_k \bullet}="a" 
!{(-1,1)}*+{\text{\fbox{$\Gamma_{1,1}$}}}="e"
!{(1,1)}*+{ \text{\fbox{$\Gamma_{1,m'-1}$}}}="c"
!{(0,1)}*+{\dots}
!{(2.5,1)}*+{ \text{\fbox{$\Gamma_{1,m'}$}}}="d"
"a"-"e"
"a"-"c"
"a"-"d"_e
}.$$
There is an $1\le r\le g+1$ with $r\neq k$ such that $e=(v_k,v_r)$. 
We additionally distinguish the following two cases.
\begin{enumerate}[(a)]
\item
Assume $r\le g$. Then we obtain graphs of this form by an extension of the form $(2)_j$, where the edge $e$ is added, or by an extension of the form $(5)_j$, where an edge from $v_r$ to $\Gamma_{1,l}$ for some $l< m'$ is replaced by the edge from $\Gamma_{1,l}$ to $v_k$ and the edge $(v_k,v_r)$. These extensions preserve the weights of the corresponding tuples. Further, we have $m=m'-1$ for the extension $(2)_j$. The $j$ is unique by the choice of the $k$-tuple and the kind of extension. Hence, the coefficient of the $k$-tuple $(\mathcal{L}_0,\dots,\mathcal{L}_{k-1})$ in $w'_k$ equals
$$((1-(m'-1))+(m'-1)) \cdot w(\mathcal{L}_0,\dots,\mathcal{L}_{k-1}) = w(\mathcal{L}_0,\dots,\mathcal{L}_{k-1}).$$
\item
Otherwise, we have $r=i=g+1$. Then graphs of this form can be obtained by extensions of the form $(5)_j$, where an edge from $v_r$ to $\Gamma_{1,l}$ for some $l< m'$ is replaced by the edge from $\Gamma_{1,l}$ to $v_k$ and the edge $(v_k,v_r)$, or by an extension of the form $(6)_j$, where the edge $e$ is added. The extensions of the form $(5)_j$ preserve the weight, while the extension of the form $(6)_j$ changes the weight by the factor $-1$.
Further, we have $m=m'-1$ for the extension of the form $(6)_j$. The $j$ is unique by the choice of the $k$-tuple and the kind of extension. Thus, the coefficient of the $k$-tuple $(\mathcal{L}_0,\dots,\mathcal{L}_{k-1})$ in $w'_k$ is given by
$$((m'-1)+(-1)\cdot((m'-1)-1)) \cdot w(\mathcal{L}_0,\dots,\mathcal{L}_{k-1}) = w(\mathcal{L}_0,\dots,\mathcal{L}_{k-1}).$$
\end{enumerate}
\end{itemize}
Thus, we obtain $w'_k=w_k$.
We conclude that $\deg(w_k)=\deg(w_{k-1})\cdot c(g,k)$, where $c(g,k)$ equals
\begin{align*}
&k\cdot\left((2-2g)+(1-m)(g-1-m)+(g-m)m+(k-1-m)+(m-1)\right)\\
=&-k(g-k+1).
\end{align*}
This proves the claim.
 \end{proof}
To prove the equations in the lemma, we use the following identity. Let $f\in \Q[X]$ be a polynomial of degree $\deg f<g$. Then it holds
\begin{align}\label{binom}
\sum_{k=0}^g (-1)^k f(k)\tbinom{g}{k-1}=0.
\end{align}
One can check this as follows: Let $0\le j< g$ be an integer. If one differentiates $j$ times the identity $(x-1)^g=\sum_{k=0}^g (-1)^{g-k} x^k \tbinom{g}{k}$
and sets $x=1$, one obtains
$$0=\sum_{k=0}^g (-1)^{k} k(k-1)\ldots (k-j+1)\tbinom{g}{k}.$$
Since this holds for any $j<g$, we also obtain (\ref{binom}) by taking linear combinations.

Now we can prove the first equation of the lemma by putting the values for $A_{k}$ and $B_{g,k}$ into equation (\ref{a1formula})
\begin{align*}
a_1&=\sum_{k=3}^{g+1}\tbinom{k-1}{2}\tfrac{k!(k-1)!}{12}(-1)^{g+1-k} (g+1-k)! \tfrac{g!}{(k-1)!} \tbinom{g}{k-1}\tbinom{g+1}{k}\\
&=\tfrac{g!(g+1)!}{12}(-1)^g\sum_{k=2}^{g}(-1)^{k} \tbinom{k}{2}\tbinom{g}{k}=0,
\end{align*}
where we applied (\ref{binom}) in the last equality

For $a'_1$ we obtain by the same arguments
$$a'_1=\sum_{k=3}^g A_{k}\cdot B'_{g,g-k}\cdot \tbinom{g-1}{k-1}\cdot\tbinom{g}{k}.$$
Here, $B'_{g,k}$ is the sum of the weights $w(\mathcal{L}_0,\dots,\mathcal{L}_{k-1})$ for all $(\mathcal{L}_0,\dots,\mathcal{L}_{k-1})\in{\mathfrak{L}'}_{g-1}^k$. We obtain
$B'_{g,k}=(-1)^k \cdot k!\cdot \frac{g!}{(g-k)!}$ in the same way as for $B_{g,k}$. One only has to note, that there is no extension of the form $(6)_j$ and we have $l,m\le g-1$ for all extensions $(2)_j-(5)_j$.
To calculate $a'_1$, we deduce from (\ref{binom}), that
$\sum_{k=3}^g(-1)^{k-1}\tbinom{k-1}{2}\tbinom{g}{k}=1$.
Now we get for $a'_1$
\begin{align*}
a'_1&=\sum_{k=3}^g \tbinom{k-1}{2}\tfrac{k!(k-1)!}{12}(-1)^{g-k}(g-k)!\tfrac{g!}{k!}\tbinom{g-1}{k-1}\tbinom{g}{k}\\
&=\tfrac{g!(g-1)!}{12}(-1)^{g-1}\sum_{k=3}^{g}(-1)^{k-1}\tbinom{k-1}{2}\tbinom{g}{k}=\tfrac{g!(g-1)!}{12}(-1)^{g-1}.
\end{align*}
This completes the proof of the lemma.
 \end{proof}
In a very similar way, we obtain the constants $a_2$ and $a'_2$.
\begin{Lem}\label{lema2}
We have $a_2=\tfrac{g!(g+1)!}{8}(-1)^{g+1}$ and $a'_2=\tfrac{(g-1)!(g+1)!}{8}(-1)^{g}$ for the constants $a_2$ and $a'_2$ in Lemma \ref{lemmaconstants}. In particular, it holds $b_2=-\tfrac{1}{8}$.
\end{Lem}
\begin{proof}
We proceed in the same way as in the proof of Lemma \ref{lema1}. In particular, we compute $a_2$ and $a'_2$ by
\begin{align}\label{equa2}
&a_2=\sum_{k=2}^{g+1}(A'_k+A''_k)\cdot B_{g,g+1-k} \cdot \tbinom{g}{k-1}\tbinom{g+1}{k} \text{ and }\nonumber \\ &a'_2=\sum_{k=2}^{g}(A'_k+A''_k)\cdot B_{g,g-k} \cdot \tbinom{g-1}{k-1}\tbinom{g}{k},
\end{align}
where $A'_k$ (respectively $A''_k$) is the number of $k$-tuples $(\mathscr{L}_0,\dots,\mathscr{L}_{k-1})\in\mathfrak{L}'^k_{k-1}$, such that the associated graph $\Gamma'(\mathscr{L}_0,\dots,\mathscr{L}_{k-1})$ is connected and has two vertices of degree $3$, which are connected by $1$ path and all other vertices have degree $2$ (respectively has one vertex of degree $4$ and all other vertices have degree $2$).
\begin{clm}\label{claim3}
It holds $A'_k=\tfrac{k!(k-1)!}{8}\tfrac{k^2+k-4}{2}$ and $A''_k=\tfrac{k!(k-1)!}{8}(k+1)$ for $k\ge 3$ and $A'_2=0$ and $A''_2=1$.
\end{clm}
\begin{proof}[Proof of the claim]
For $k=2$ one can directly compute $A'_k$ and $A''_k$. For $k\ge 3$ we only compute $A'_k$. The computation for $A''_k$ can be done similarly. Let $(\mathscr{L}_0,\dots,\mathscr{L}_{k-1})\in\mathfrak{L}'^k_{k-1}$ be a $k$-tuple of the form counted by $A'_k$ and $\Gamma$ its associated graph $\Gamma'(\mathscr{L}_0,\dots,\mathscr{L}_{k-1})$. In general there are $k!$ possibilities to sort the line bundles in the tuple and $\tfrac{(k-1)!}{2}$ possibilities to sort the list of the vertices in $\Gamma$, since if we reverse the list, we obtain the same line bundles. We have to choose the two vertices of degree $3$. We distinguish the following three cases:
\begin{itemize}
\item The first and the last vertex in the list of vertices are of degree $3$. That means the two circles in $\Gamma$ are paths of length $1$, hence they are loops.
\item Either the first or the last vertex is of degree $3$. Then exactly one circle in $\Gamma$ is of length at least $2$. If the length is $2$, there are two line bundles in the $k$-tuple, which are equal. Hence, there are $\tfrac{k!}{2}$ instead of $k!$ possibilities to sort the line bundles in the $k$-tuple. Now let the length of the circle be at least $3$. If we reverse the vertices in this circle without the vertex of degree $3$, we obtain the same line bundles. Hence, we only have to count $\tfrac{(k-1)!}{2}$ instead of $(k-1)!$ possibilities to sort the list of vertices in $\Gamma$.
\item Otherwise, both circles in $\Gamma$ have length at least $2$. By the same arguments of the case above, we have to divide the number of possibilities by $4$.
\end{itemize}
We conclude
$$A'_k=\tfrac{k!(k-1)!}{2}\left(1+\tfrac{2(k-3)}{2}+\tfrac{\tbinom{k-3}{2}}{4}\right)=\tfrac{k!(k-1)!}{8}\tfrac{k^2+k-4}{2}.$$
This proves the claim.
 \end{proof}
Now we can apply Claim \ref{claim2} and \ref{claim3} to (\ref{equa2}). That yields
$$a_2=\tfrac{(-1)^g g! (g+1)!}{8}\left(\sum_{k=3}^{g+1}\tfrac{k^2+3k-2}{2}(-1)^{k-1}\tbinom{g}{k-1}-4g\right)\text{ and}$$
$$a'_2=\tfrac{(-1)^g g! (g-1)!}{8}\left(\sum_{k=3}^g\tfrac{k^2+3k-2}{2}(-1)^k\tbinom{g}{k}+4\tbinom{g}{2}\right).$$
Applying (\ref{binom}) to the sums, we obtain the formulas in the lemma.
 \end{proof}
Now we apply Lemmas \ref{lema1} and \ref{lema2} to equation (\ref{laplaceH2}). This yields
\begin{align}\label{laplaceH3}
\tfrac{1}{\pi i} \partial \overline{\partial} H(X)=\tfrac{1}{2}\omega_{\mathrm{Hdg}}+\tfrac{1}{12}\int_{\pi_2} h^3-\tfrac{1}{8} e_1^A
\end{align}
as forms on $\mathcal{M}'_g$. By continuity, this relation also holds for $\mathcal{M}_g[2]$. Since all these forms are already defined on $\mathcal{M}_g$, this formula also holds for the corresponding forms on $\mathcal{M}_g$.

\subsection{Main result}
\label{secgeneral}
In this subsection we deduce our main result, which generalizes the second formula in Theorem \ref{hypdelta} to compact and connected Riemann surfaces. Precisely, we prove the following theorem.
\begin{Thm}\label{generalthm}
Any compact and connected Riemann surface $X$ of genus $g\ge 1$ satisfies $\delta(X)=-24H(X)+2\varphi(X)-8g\log 2\pi$.
\end{Thm}
\begin{proof}
For $g=1$ and $g=2$ this follows from \cite[Section 7]{Fal84}, respectively Theorem \ref{hypdelta}. Thus, we assume $g\ge 3$.
Consider the function 
$$f(X)=\delta(X)+24H(X)-2\varphi(X)$$
as a real-valued function on $\mathcal{M}_g$.
By (\ref{laplacekz}), (\ref{laplacedelta}) and (\ref{laplaceH3}) this function satisfies $\partial\overline{\partial} f(X)=0$, that means $f$ is pluriharmonic on $\mathcal{M}_g$.  But any pluriharmonic function is locally the real part of a holomorphic function, which is unique up to an additive constant by this property. Let $f$ be the real part of a holomorphic function $h_j$ on $U_j$, where $\mathcal{M}_g=\bigcup_{j}U_j$ is an open covering.

We can consider $\mathcal{M}_g$ as the orbifold $\mathcal{M}_g\cong\mathcal{T}_g/\Gamma_g$, where $\mathcal{T}_g$ is the Teichmüller space and $\Gamma_g$ is the mapping class group. Since $\mathcal{T}_g$ is contractible, the pullback of $f$ to $\mathcal{T}_g$ has to be globally the real part of a holomorphic function on $\mathcal{T}_g$. Hence, we can choose the $h_j$'s, such that we can glue their pullbacks in $\mathcal{T}_g$ to a globally holomorphic function. On the other hand, the mapping class group $\Gamma_g$ is perfect, see \cite[Theorem 1]{Pow78}. This means $H^1(\Gamma_g,\Z)=0$. Thus, we can glue the $h_j$'s also in $\mathcal{M}_g$. Therefore, $f$ has to be globally the real part of a holomorphic function on $\mathcal{M}_g$. But every holomorphic function on $\mathcal{M}_g$ is constant, see for example \cite[Proposition 7.4]{Sch07}.
 Thus, $f(X)$ is constant on $\mathcal{M}_g$, and we obtain $f(X)=-8g \log 2\pi$ by Theorem \ref{hypdelta}.
 \end{proof}

As an application of the theorem we obtain a lower bound for the invariant $\delta(X)$ by applying the lower bounds in Proposition \ref{Hbound} and (\ref{kzbound}).
\begin{Cor}\label{deltabound}
For any compact and connected Riemann surface $X$ of genus $g\ge 1$ we have $\delta(X)> -2g\log 2\pi^{4}$.
\end{Cor}
One can generalize the invariant $\|\Delta_g\|$ of hyperelliptic Riemann surfaces to arbitrary compact and connected Riemann surfaces of positive genus, even to principally polarised complex abelian varieties. Let $(A,\Theta)$ be any principally polarised complex abelian variety of dimension $g\ge 1$ as in Section \ref{secinvariantsabelian}. We define the set $\mathcal{D}=\lbrace \eta\in \tfrac{1}{2}\Z^{2g}/\Z^{2g}~|~\theta[\eta](0)\neq 0\rbrace$ and we set
$$\|\Delta_g\|(A,\Theta)=2^{-4(g+1)\binom{2g}{g-1}}(\det Y)^{2r}\left|\sum_{\gf{\mathcal{J}\subseteq\mathcal{D}}{|\mathcal{J}|=r}}\prod_{\eta\in\mathcal{J}}\theta[\eta](0)^8\right|,$$
where $r=\tbinom{2g+1}{g+1}$. In particular, we have $\|\Delta_g\|(\Jac(X))=\|\Delta_g\|(X)$ if $X$ is a hyperelliptic Riemann surface of genus $g\ge 2$. Hence, we define $\|\Delta_g\|(X)=\|\Delta_g\|(\Jac(X))$ if $X$ is an arbitrary connected and compact Riemann surface of genus $g\ge 1$. However,
the first formula of Theorem \ref{hypdelta} and Corollaries \ref{discdeltavarphi} and \ref{discbound} are not true for arbitrary connected and compact Riemann surfaces. Indeed, we have
$$\tfrac{1}{\pi i}\partial\overline{\partial}\log\|\Delta_g\|(X)=4r\cdot\omega_{\mathrm{Hdg}}-\delta_Z$$ as forms on $\mathcal{M}_g$, where $Z\subseteq \mathcal{M}_g$ is the vanishing locus of $\|\Delta_g\|$. Comparing this with the forms (\ref{laplacekz}), (\ref{laplacedelta}) and (\ref{laplaceH3}), we notice that each of the mentioned formulas for hyperelliptic Riemann surfaces implies $3e_1^A=(2-2g)\int_{\pi_2}h^3$, which is not true in general on $\mathcal{M}_g$, see also \cite[Section 10]{dJo16}.

\subsection{Bounds for theta functions}
\label{secbound}
In this subsection we give an upper bound for the function $\|\theta\|$. This bound will be used in the next subsection to obtain an upper bound for the Arakelov--Green function.
Let $(A,\Theta)$ be any principally polarised complex abelian variety of dimension $g\ge1$ as in Section \ref{secinvariantsabelian}. We use the same notation as in Section \ref{secinvariantsabelian}, where we may assume, that the matrix $\Omega$ is Siegel reduced, see \cite[Chapter V.\S 4]{Igu72}.

Due to Autissier \cite[Proposition 1,1]{Aut15}, any $z\in A$ satisfies
\begin{align}\label{boundautissier}
\|\theta\|(z)\le c_g \det (Y)^{1/4},
\end{align}
where $c_g=\frac{g+2}{2}$ for $g\le 3$ and $c_g=\frac{g+2}{2}\left(\frac{g+2}{\pi\sqrt{3}}\right)^{g/2}$ for $g\ge 4$. Hence, it remains to find an upper bound for $\det (Y)$. This is done by the following lemma.
\begin{Lem}\label{Ybound}
	For any real number $s\ge 0$ we obtain
	$$s\log \det (Y)+H(A,\Theta)\le g\left(s+\tfrac{1}{4}\right)\log\tfrac{4s+1}{2}.$$
\end{Lem}
\begin{proof}
Write $F=[0,1]^g$ and denote by $\lambda$ the Lebesgue measure on $\R^g$.
We bound $H(A,\Theta)$ in the following way:
\begin{align*}
&H(A,\Theta)=\tfrac{1}{2}\int_F\int_F \log\|\theta\|(x+\Omega y)^2\lambda(x)\lambda(y) \\
\le& \tfrac{1}{2}\int_F\left(\log\int_F\|\theta\|(x+\Omega y)^2\lambda(x)\right)\lambda(y) \\
=&\tfrac{1}{2}\int_F\log\left(\sqrt{\det(Y)}\sum_{n\in\Z^g}\exp\left(-2\pi\ltrans{(m+y)}Y(m+y)\right)\right)\lambda(y)\\
\le&\tfrac{1}{4}\log\det(Y)\\ &+\tfrac{4s+1}{2}\log\int_F\left(\sum_{n\in \Z^g}\exp(-2\pi\ltrans{(m+y)}Y(m+y))\right)^{1/(4s+1)}\lambda(y)\\
\le&\tfrac{1}{4}\log\det(Y)+\tfrac{4s+1}{2}\log\int_F\sum_{n\in \Z^g}\exp(-2\pi\ltrans{(m+y)}(\tfrac{1}{4s+1}\cdot Y)(m+y))\lambda(y)\\
=&\tfrac{1}{4}\log\det(Y)\\ &-\tfrac{4s+1}{2}\log \sqrt{\det\left(\tfrac{2}{4s+1}\cdot Y\right)}=-s\log \det(Y)+g\left(s+\tfrac{1}{4}\right)\log\tfrac{4s+1}{2}.
\end{align*}
In the first line, we applied the definition of $H(X)$ and replaced the integral over $A$ by an integral over the fundamental domain $F+\Omega \cdot F$ in $\C^g$ for $A$. The second line follows by Jensen's inequality. The third line follows by Parseval's formula. In the fourth line we again applied Jensen's inequality. Since $s\ge 0$, we have $4s+1\ge 1$, such that $a^{4s+1}+b^{4s+1}\le(a+b)^{4s+1}$ for positive real numbers $a,b$. Applying this to the converging series of positive real numbers in the fourth line, we obtain the fifth line. For the computation of the integral in the second last line, see for example the proof of \cite[Proposition 8.5.6]{BL04}.
 \end{proof}
Combining the lemma with Autissier's bound (\ref{boundautissier}), we obtain the following bound for the theta function:
\begin{Cor}\label{thetabound}
For any real number $r> 0$ and any $z\in A$ it holds
$$\log\|\theta\|(z)+rH(A,\Theta)\le \log c_g+\tfrac{g(1+r)}{4}\log\tfrac{1+r}{2r}.$$
\end{Cor}

\subsection{The Arakelov--Green function}
\label{green}
We give an explicit expression for the Arakelov--Green function by calculating Bost's invariant $A(X)$ in (\ref{Bostinvariant}). Furthermore, we will bound the supremum of the Arakelov--Green function in terms of $\delta(X)$ and we give another expression for $\delta$. Let $X$ be any compact and connected Riemann surface of genus $g\ge 1$.
\begin{Thm}\label{propgreen}
	It holds
	$$g(P,Q)=\tfrac{1}{g!}\int_{\Theta+P-Q}\log\|\theta\|\nu^{g-1}+\tfrac{1}{2g} \varphi(X)-H(X).$$
\end{Thm}
\begin{proof}
For $g=1$ this is just a reformulation of Faltings formula in \cite[Section 7]{Fal84}. Hence, we may assume $g\ge 2$.
Integrating (\ref{Bostinvariant}) with $\mu(P)$ gives
$$-A(X)=\tfrac{1}{g!}\int_{X}\left(\int_{\Theta+P-Q}\log\|\theta\|\nu^{g-1}\right)\mu(P).$$
We define the map
$$\Phi_{\Theta}\colon X^{g-1}\to \Theta,\quad(P_1,\dots,P_{g-1})\mapsto P_1+\dots+P_{g-1},$$
which is smooth, surjective and generically of degree $(g-1)!$.
Since $\nu$ is translation-invariant, we conclude that
$$-A(X)=\tfrac{1}{(g-1)!g!}\int_{X^{g}}\log\|\theta\|(P_1+\dots+P_g-Q)\Phi_{\Theta}^*\nu^{g-1}(P_1,\dots,P_{g-1})\mu(P_g).$$
Using the function $\|\Lambda\|$ in (\ref{Lambda}), we obtain
\begin{align*}
-A(X)=&\tfrac{1}{(g-1)!g!}\int_{X^{g-1}}\log\|\Lambda\|(P_1+\dots+P_{g-1})\Phi_{\Theta}^*\nu^{g-1}(P_1,\dots,P_{g-1})\\
=&\tfrac{1}{(g!)^2}\int_{X^{g}}\log\|\Lambda\|(P_1+\dots+P_{g-1})\Phi^*\nu^{g}(P_1,\dots,P_g),
\end{align*}
since the Arakelov--Green functions in (\ref{Lambda}) integrates to $0$. The latter equality follows by Lemma \ref{lemformen}.
If we again substitute $\log \|\Lambda\|$ by (\ref{Lambda}) in the last expression, only the integrals of $\log\|\theta\|(P_1+\dots+P_g-Q)$ and of $-g(\sigma(P_1+\dots+P_{g-1}),P_g)$ are non-zero. The first one gives $H(X)$ and the second one equals $-\tfrac{1}{2g}\varphi(X)$ by Lemma \ref{lemgsigmavarphi}.
Thus, we obtain the identity $A(X)=\tfrac{1}{2g}\varphi(X)-H(X)$.
 \end{proof}

As a corollary we bound the Arakelov--Green function in terms of $\delta(X)$.
\begin{Cor}\label{greendelta}
For any real number $r$ satisfying $r\ge \tfrac{6}{g}-1$ and $r>0$ we have
$$\sup_{P,Q\in X}g(P,Q)\le\tfrac{1+r}{24}\delta(X)+\tfrac{g(1+r)}{3}\log 2\pi+\log c_g+\tfrac{g(1+r)}{4}\log\tfrac{1+r}{2r}.$$
In particular, the Arakelov--Green function is bounded by
$$\sup_{P,Q\in X}g(P,Q)<\tfrac{1}{24g}\max(6,g+1)\delta(X)+\tfrac{3}{4}g\cdot\log g+4.$$ 
\end{Cor}
\begin{proof}
	Since $\int_{\Theta+P-Q}\nu^{g-1}=g!$, applying the bound in Corollary \ref{thetabound} to Theorem \ref{propgreen} yields
	$$\sup_{P,Q\in X}g(P,Q)\le \tfrac{1}{2g}\varphi(X)-(1+r)H(X)+\log c_g+\tfrac{g(1+r)}{4}\log\tfrac{1+r}{2r}.$$
By Theorem \ref{generalthm} and (\ref{kzbound}) we get
	$$\tfrac{1}{2g}\varphi(X)-(1+r)H(X)\le \tfrac{1+r}{24}\delta(X)+\tfrac{g(1+r)}{3}\log 2\pi.$$
The first bound in the corollary follows by combining these two inequalities. One obtains the second assertion by putting $r=\tfrac{6}{g}-1$ for $g\le 5$ and $r=1/g$ for $g\ge 6$ and bounding the constants explicitly.	
 \end{proof}

As an application of the proof of Theorem \ref{propgreen}, we obtain a formula for $\delta(X)$ only in terms of integrals of the function $\log\|\theta\|$.
\begin{Cor}
We have
$$\delta(X)=-\tfrac{4g}{g!}\int_X\left(\int_{\Theta+P-Q}\log\|\theta\|\nu^{g-1}\right)\mu(P)+(4g-24)H(X)-8g\log 2\pi.$$
\end{Cor}
\begin{proof}
By the proof of Theorem \ref{propgreen} we have
$$-\tfrac{4g}{g!}\int_X\left(\int_{\Theta+P-Q}\log\|\theta\|\nu^{g-1}\right)\mu(P)=2\varphi(X)-4g H(X).$$
If we apply this to Theorem \ref{generalthm}, we obtain the corollary.
 \end{proof}

\section{The case of abelian varieties}\label{Sectionabelian}
We state formulas for $\delta(X)$ and $\varphi(X)$ only in terms of $H(X)$ and $\Lambda(X)$, such that we obtain canonical extensions of the functions $\delta$ and $\varphi$ to the moduli space of indecomposable principally polarised complex abelian varieties. Further, we discuss some of the asymptotics of these extensions.
\subsection{The delta invariant of abelian varieties}
\label{secabelian}
We deduce the following expressions for $\delta$ and $\varphi$ from the expressions in Theorem \ref{generalthm} and formula (\ref{deltaeta}).
\begin{Thm}\label{thmextension}
For any compact and connected Riemann surface $X$ of genus $g\ge 1$, the invariant $\delta(X)$ satisfies
$$\delta(X)=2(g-7)H(X)-2\Lambda(X)-4g\log 2\pi.$$
Further, the invariant $\varphi(X)$ satisfies
$$\varphi(X)=(g+5)H(X)-\Lambda(X)+2g\log 2\pi.$$
\end{Thm}
\begin{proof}
If $g=1$, the equations can again be obtained from \cite[Section 7]{Fal84}. We leave it to the reader, to work this out.
Hence, we may assume $g\ge 2$.

If we integrate the logarithm of formula (\ref{deltaeta}) with respect to $\Phi^*\nu^g$, we obtain by equation (\ref{intpb2}) and by Lemma \ref{lemgsigmaint}
$$\tfrac{1}{(g!)^2}\int_{X^g}\log\|\eta\|(P_1+\dots+P_{g-1})\Phi^*\nu^g=(g-1)H(X)-\tfrac{1}{4}\delta(X)-\tfrac{1}{2}\varphi(X).$$
Denote by $\Phi_{\Theta}$ the map defined in Section \ref{green}. We have
\begin{align*}
\Lambda(X)&=\tfrac{1}{(g-1)!g!}\int_{X^{g-1}}\log\|\eta\|(P_1+\dots+P_{g-1})\Phi_{\Theta}^*\nu^{g-1}\\
&=\tfrac{1}{(g!)^2}\int_{X^g}\log\|\eta\|(P_1+\dots+P_{g-1})\Phi^*\nu^g,
\end{align*}
where the latter equality follows from Lemma \ref{lemformen}. Putting both equations together, we obtain
$$\Lambda(X)=(g-1)H(X)-\tfrac{1}{4}\delta(X)-\tfrac{1}{2}\varphi(X).$$
Now both formulas in the theorem follow by Theorem \ref{generalthm}.
 \end{proof}

Let $(A,\Theta)$ be an indecomposable principally polarised complex abelian variety of dimension $g\ge2$ as in Section \ref{secinvariantsabelian}. We define
\begin{align*}
\delta(A,\Theta)&=2(g-7)H(A,\Theta)-2\Lambda(A,\Theta)-4g\log 2\pi,\\
\varphi(A,\Theta)&=(g+5)H(A,\Theta)-\Lambda(A,\Theta)+2g\log 2\pi.
\end{align*}
Then we have $\delta(\Jac(X))=\delta(X)$ and $\varphi(\Jac(X))=\varphi(X)$ for any compact and connected Riemann surface $X$ by Theorem \ref{thmextension}. Hence, we obtain canonical extensions of $\delta$ and $\varphi$ to the moduli space of indecomposable principally polarised complex abelian varieties.
For Riemann surfaces we have the bounds $\varphi(X)>0$ and $\delta(X)>-2g\log2\pi^4$. It is a natural question whether these bounds are still true for the extended versions of $\delta$ and $\varphi$.
\begin{Que}
Do all indecomposable principally polarised complex a\-belian varieties $(A,\Theta)$ of dimension $g\ge 2$ satisfy $\varphi(A,\Theta)> 0$?
\end{Que}
If the answer of this question is yes, we will also obtain the lower bound $\delta(A,\Theta)>-2g\log 2\pi^4$. If the answer is no, $\varphi$ could be seen as an indicator for an abelian variety to be a Jacobian.

Finally in this section, we consider the Hain--Reed invariant $\beta_g(X)$ of any compact and connected Riemann surface $X$ of genus $g\ge 2$, which was introduced by Hain and Reed in \cite{HR04} as a quotient of two canonical metrics on $\left(\bigwedge^g H^0(X,\Omega_X^1)\right)^{\otimes 8g+4}$. This invariant is only defined modulo constants on $\mathcal{M}_g$. De Jong obtained a canonical normalization by proving that $\tfrac{1}{3}((2g-2)\varphi(X)+(2g+1)\delta(X))$ is a representative of $\beta_g(X)$, see \cite[Theorem 1.4]{dJo13}.
Hence, we can also define $\beta_g$ for indecomposable principally polarised complex abelian varieties by
$$\beta_g(A,\Theta)=2(g-4)(g+1)H(A,\Theta)-2g\Lambda(A,\Theta)-\tfrac{4g(g+2)}{3}\log 2\pi.$$
By Theorem \ref{thmextension} we have $\beta_g(\Jac(X))=\beta_g(X)$ for any compact and connected Riemann surface $X$ of genus $g\ge2$.

\subsection{Asymptotics}\label{secasmypt}
Next, we discuss some of the asymptotics of the extended versions of the invariants $\delta$ and $\varphi$ for degenerating families of indecomposable principally polarised complex abelian varieties. We denote by $D\subseteq \C$ the open unit disc, and we write $f(t)=O(g(t))$ for two functions $f,g\colon D\to \R$ if there exists a bound $M\in\R$ not depending on $t$, such that $|f(t)|\le M\cdot |g(t)|$ for all $t\in D$.
If $\mathscr{X}\to D$ is a family of complex curves, such that $\mathscr{X}_t$ is a Riemann surface if and only if $t\neq 0$ and $\mathscr{X}_0$ has exactly one node, then Jorgenson \cite{Jor90}, Wentworth \cite{Wen91} and de Jong \cite{dJo14a} showed that $\delta(\mathscr{X}_t)$ and $\varphi(\mathscr{X}_t)$ go to infinity for $t\to 0$.

It is a natural question, whether the same is true for the extended versions of $\delta$ and $\varphi$ on the moduli space of indecomposable principally polarised complex abelian varieties. As a first step, we obtain the following asymptotic behaviour of $\delta$ and $\varphi$ for families of indecomposable principally polarised complex abelian varieties degenerating to a decomposable principally polarised complex abelian variety.
\begin{Pro}
Let $\tau\colon D\to \mathds{H}_g$ be a holomorphic embedding and write $(A_t,\Theta_t)$ for the principally polarised complex abelian variety associated to $\tau(t)$. If $(A_t,\Theta_t)$ is indecomposable for $t\neq 0$ and $(A_0,\Theta_0)$ is the product of two indecomposable 
principally polarised complex abelian varieties $(A_1,\Theta_1)$ and $(A_2,\Theta_2)$ of positive dimensions $g_1$, respectively $g_2$, then it holds
$$\lim_{t\to 0} H(A_t,\Theta_t)=H(A_1,\Theta_1)+H(A_2,\Theta_2),$$
$$\Lambda(A_t,\Theta_t)-\tfrac{2g_1 g_2}{g}\log |t|=O(1),$$
 $$\delta(A_t,\Theta_t)+\tfrac{4g_1 g_2}{g}\log |t|=O(1) \text{ and}$$
$$\varphi(A_t,\Theta_t)+\tfrac{2g_1 g_2}{g}\log |t|=O(1).$$
\end{Pro} 
\begin{proof}
For $t\in D$ and $j\in\lbrace1, 2\rbrace$ we denote by $\nu_t=\nu_{(A_t,\Theta_t)}$ and $\nu_j=\nu_{(A_j,\Theta_j)}$ the canonical $(1,1)$ form of $(A_t,\Theta_t)$ respectively $(A_j,\Theta_j)$.
We may assume, that $\tau(0)$ is of the form
$$\tau(0)=\begin{pmatrix} \Omega_1 & 0\\
0& \Omega_2 \end{pmatrix},$$
where $\Omega_j \in \mathds{H}_{g_j}$ is a matrix associated to $(A_j,\Theta_j)$.
We have $\nu_0=\nu_1+\nu_2$ and hence,
$$\tfrac{1}{g!}\nu_0^g=\tfrac{1}{g_1! g_2!}\nu_1^{g_1} \nu_2^{g_2}\quad\text{and}\quad \tfrac{1}{g!}\nu_0^{g-1}=\tfrac{g_1}{g \cdot g_1! g_2!}\nu_1^{g_1-1}\nu_2^{g_2}+\tfrac{g_2 }{g \cdot g_1! g_2!}\nu_1^{g_1}\nu_2^{g_2-1}.$$
Likewise, we obtain $\det(\Im\tau(0))=\det(\Im\Omega_1)\cdot \det(\Im\Omega_2)$.
Every $z\in A_t$ can be represented by $a+\tau(t)\cdot b$ for some real vectors $a,b\in [0,1]^g$. Fix arbitrary vectors $a,b\in [0,1]^g$ and write $z_t=a+\tau(t)\cdot b$.
We obtain for the function $\theta$
\begin{align*}
&\exp\left(-\pi\ltrans{(\Im z_t)}(\Im\tau(t))^{-1}(\Im z_t)\right)\cdot|\theta|(\tau(t);z_t)\\
=&\left|\sum_{n\in\Z^g}\exp\left(\pi i\ltrans{(n+b)}\tau(t)(n+b)+2\pi i \ltrans{n} a\right)\right|.
\end{align*}
In particular, we have $\|\theta\|(\tau(0);\begin{pmatrix} z_1\\z_2\end{pmatrix})=\|\theta\|(\Omega_1;z_1)\cdot \|\theta\|(\Omega_2;z_2)$, where $z_j\in \C^{g_j}$, and hence, $H(A_0,\Theta_0)=H(A_1,\Theta_1)+H(A_2,\Theta_2)$.

We also deduce, that $\Theta_0=\left(\Theta_1\times A_2\right)\cup\left(A_1\times \Theta_2\right)$.
If we set $n_{g+1}=\tfrac{1}{2\pi i}$, the function $\|\eta\|$ can be written by
\begin{align}\label{etadeterminante}
&\|\eta\|(\tau(t);z_t)\cdot\det(\Im\tau(t))^{-(g+5)/4}\\
=&\left|\det\left(4\pi^2\sum_{n\in \Z^g} n_j n_k\exp(\pi i \ltrans{(n+b)}\tau(t)(n+b)+2\pi i \ltrans{n} a)\right)_{j,k\le g+1}\right|,\nonumber
\end{align}
where $z_t=a+\tau(t)\cdot b\in\Theta_t$.
Write
$a=\begin{pmatrix}a_1 \\ a_2\end{pmatrix}$ and $b=\begin{pmatrix}b_1 \\ b_2\end{pmatrix},$
where $a_j$ and  $b_j$ are $g_j$-dimensional vectors.
Let $a+\tau(0)\cdot b$ represent an element in $\Theta_1\times A_{2}$.
Consider the expression
\begin{align*}
\widetilde{\theta}_{jk}(\tau(t);a,b)=\sum_{n\in \Z^g} n_j n_k\exp(\pi i \ltrans{(n+b)}\tau(t)(n+b)+2\pi i \ltrans{n} a).
\end{align*}
If $j\le g_1$ or $k\le g_1$, then $\widetilde{\theta}_{jk}(\tau(0);a,b)$ is non-zero for a dense subset of pairs $(a,b)$ in 
$$M=\left\lbrace (a,b)\in\left[0,1\right]^g~|~a+\tau(0)\cdot b\in \Theta_1\times A_2\right\rbrace.$$
Otherwise, it is zero, since we can write it as a product containing the factor
\begin{align}\label{thetal}
\sum_{n\in \Z^{g_1}} \exp(\pi i \ltrans{(n+b_1)}\Omega_1(n+b_1)+2\pi i \ltrans{n} a_1),\end{align}
which vanishes by $(a_1+\Omega_1\cdot b_1)\in \Theta_1$. But the expression
$$\lim_{t\to 0}\frac{\widetilde{\theta}_{jk}(\tau(t);a,b)}{t},$$
is non-zero for a dense subset of pairs $(a,b)$ in $M$. To check this, one uses the chain rule to obtain a linear combination of partial derivatives of (\ref{thetal}) with coefficients $\tfrac{\partial\tau_{pq}(t)}{\partial t}|_{t=0}$ with $p\le g_1$ and $q>g_2$, which do not vanish all by the definition of $\tau$.

We have to compute the order of vanishing at $t=0$ for the summands in the expansion of the determinant in (\ref{etadeterminante}). Let $\sigma\in\mathrm{Sym}(g+1)$ be any permutation with $\sigma(g+1)\neq g+1$. Denote by $m(\sigma)$ the cardinality of $\lbrace j\le g_1~|~\sigma(j)>g_1\rbrace$. The observations above shows, that
$$\prod_{j=1}^{g+1}\widetilde{\theta}_{j,\sigma(j)}(\tau(t);a,b)$$
vanishes of order $g_2+1-m(\sigma)$ at $t=0$ for a dense subset of pairs $(a,b)$ in $M$.
But for different $j_1,j_2\le g_1$ and different $k_1,k_2> g_1$ the function $\widetilde{\theta}_{j_l k_m}(\tau(t);a,b)$ splits into a product of two factors, such that the expression
$$\widetilde{\theta}_{j_1k_1}(\tau(t);a,b)\cdot\widetilde{\theta}_{j_2k_2}(\tau(t);a,b)- \widetilde{\theta}_{j_1k_2}(\tau(t);a,b)\cdot\widetilde{\theta}_{j_2k_1}(\tau(t);a,b)$$
vanishes at $t=0$. If $\sigma$ satisfies $m(\sigma)\ge 2$, $\sigma(j_1)=k_1$ and $\sigma(j_2)=k_2$, then we construct $\sigma'\in\mathrm{Sym}(g+1)$ by setting $\sigma'(j_1)=\sigma(j_2)$, $\sigma'(j_2)=\sigma(j_1)$ and $\sigma'(j)=\sigma(j)$ for $j\notin\lbrace j_1,j_2\rbrace$. We obtain that
$$\prod_{j=1}^{g+1}\widetilde{\theta}_{j,\sigma(j)}(\tau(t);a,b)-\prod_{j=1}^{g+1}\widetilde{\theta}_{j,\sigma'(j)}(\tau(t);a,b)$$
vanishes of order at least $g_2+2-m(\sigma)$.
Inductively, we deduce that the determinant in (\ref{etadeterminante}) vanishes of order at least $g_2$. Since there is no such cancellation for permutations with $m(\sigma)=1$, we conclude that 
$$\log\|\eta\|(\tau(t);a+\tau(t)\cdot b)=g_2  \log |t|+O(1)$$
for a dense subset of pairs $(a,b)$ in $M$.
We can argue analogously for $a$, $b$ satisfying $(a+\tau(0)\cdot b)\in A_1\times \Theta_2$.
Then we obtain for the invariant $\Lambda(A_t,\Theta_t)$:
\begin{align*}
\Lambda(A_t,\Theta_t)=&\int_{\Theta_1\times A_2}\left(g_2  \log |t|+O(1)\right)\tfrac{g_1}{g \cdot g_1! g_2!}\nu_1^{g_1-1}\nu_2^{g_2}\\
&+\int_{A_1\times \Theta_2}\left(g_1\log |t|+O(1)\right)\tfrac{g_2 }{g \cdot g_1! g_2!}\nu_1^{g_1}\nu_2^{g_2-1}\\
=&\tfrac{2g_1 g_2}{g}\log |t|+O(1).
\end{align*}
Now the formulas for $\delta$ and $\varphi$ in the proposition follow by Theorem \ref{thmextension}.
 \end{proof}

\vspace{2cm}
Robert Wilms\\
Institut für Mathematik\\
Johannes Gutenberg-Universität Mainz\\
Staudingerweg 9\\
55128 Mainz\\
Germany\\
E-mail: rowilms@uni-mainz.de

\begin{thebibliography}{AA}
\bibitem{Ara74} Arakelov, S. Y.: \emph{Intersection theory of divisors on an arithmetic surface}. Izv. Akad. USSR {\bf 8}, no. 6, 1167--1180 (1974).
\bibitem{Aut15} Autissier, P.: \emph{An upper bound for the theta function}. Preprint, \url{https://www.math.u-bordeaux.fr/\~pautissi/Theta.pdf} (2015). Accessed 27 October 2015
\bibitem{BL04} Birkenhake, C.; Lange, H.: \emph{Complex abelian varieties}. Grundlehren der Mathematischen Wissenschaften {\bf 302}, Springer, Berlin (2004).
\bibitem{BMM90} Bost, J.-B.; Mestre, J.-F.; Moret-Bailly, L.: \emph{Sur le calcul explicite des ``classes de Chern'' des surfaces arithmétiques de genre 2}. In: S\'eminaire sur les Pinceaux de Courbes Elliptiques, Ast\'erisque No. {\bf 183}, pp. 69--105 (1990).
\bibitem{Bos87} Bost, J.-B.: \emph{Fonctions de Green-Arakelov, fonctions thêta et courbes de genre 2}. C. R. Acad. Sci. Paris S\'er. I Math. {\bf 305}, no. 14, 643--646 (1987).
\bibitem{Del85} Deligne, P.: \emph{Le d\'eterminant de la cohomologie}. In: Current trends in arithmetical algebraic geometry (Arcata, Calif., 1985), Contemp. Math. {\bf 67}, pp. 387--424 (1987).
\bibitem{DFN85} Dubrovin, B. A.; Fomenko, A. T.; Novikov, S. P.: \emph{Modern Geometry - Methods and Applications - Part II. The Geometry and Topology of Manifolds}. Graduate Texts in Mathematics, {\bf 104}, Springer, New York (1985).
\bibitem{EC11} Edixhoven, B.; Couveignes, J.-M.: \emph{Computational aspects of modular forms and Galois representations}. Annals of Mathematics Studies, {\bf 176}, Princeton University Press, Princeton (2011).
\bibitem{Elk89} Elkik, R.: \emph{Fibrés d’intersections
et intégrales de classes de Chern}. Ann. Sci. Ecole Norm. Sup. {\bf 22}, no. 2, 195--226 (1989).
\bibitem{Fal84} Faltings, G.: \emph{Calculus on arithmetic surfaces}. Ann. of Math. {\bf 119}, 387--424 (1984).
\bibitem{FK80} Farkas, H. M.; Kra, I.: \emph{Riemann Surfaces}. Graduate Texts in Mathematics, {\bf 71}, Springer, New York-Berlin (1980).
\bibitem{Gua99} Guàrdia, J.: \emph{Analytic invariants in Arakelov theory for curves}. C.R. Acad. Sci. Paris Ser. I {\bf 329}, 41--46 (1999).
\bibitem{Gua02} Guàrdia, J.: \emph{Jacobian nullwerte and algebraic equations}. Jnl. of Algebra {\bf 253}, 1, 112--132 (2002).
\bibitem{HL97} Hain, R.; Looijenga, E.: \emph{Mapping class groups and moduli spaces of curves}. Algebraic Geometry--Santa Cruz 1995, Proc. Sympos. Pure Math. {\bf 62},
pp. 97–142 (1997).
\bibitem{HR01} Hain, R.; Reed, D.: \emph{Geometric proofs of some results of Morita}. J. Algebraic Geom. {\bf 10}, no. 2, 199--217 (2001).
\bibitem{HR04} Hain, R.; Reed, D.: \emph{On the Arakelov geometry of moduli spaces of curves}. J. Differential Geom. {\bf 67}, no. 2, 195--228 (2004).
\bibitem{Igu72} Igusa, J.-I.: \emph{Theta Functions}. Grundlehren der mathematischen Wissenschaften {\bf 194}, Springer, New York (1972).
\bibitem{Jav14} Javanpeykar, A.: \emph{Polynomial bounds for Arakelov invariants of Belyi curves. With an appendix by Peter Bruin}. Algebra Number Theory {\bf 8}, no. 1, 89--140 (2014). 
\bibitem{dJo05a} de Jong, R.: \emph{Arakelov invariants of Riemann surfaces}. Doc. Math. {\bf 10}, 311--329 (2005).
\bibitem{dJo05b} de Jong, R.: \emph{Faltings' delta-invariant of a hyperelliptic Riemann surface}. In: G. van der Geer, B. Moonen, R. Schoof (eds.), Number Fields and Function Fields -- Two Parallel Worlds, Progress in Mathematics vol. {\bf 239}, pp. 223--236. Birkhäuser Verlag (2005).
\bibitem{dJo07} de Jong, R.: \emph{Explicit Mumford isomorphism for hyperelliptic curves}. Jnl. Pure Appl. Algebra {\bf 208}, 1--14 (2007).
\bibitem{dJo08} de Jong, R.: \emph{Gauss map on the theta divisor and Green's functions}. In: B. Edixhoven, G. van der Geer and B. Moonen (eds.), Modular Forms on Schiermonnikoog, pp. 67--78. Cambridge University Press (2008).
\bibitem{dJo10} de Jong, R.: \emph{Theta functions on the theta divisor}. Rocky Mountain Jnl. Math. {\bf 40}, 155--176 (2010).
\bibitem{dJo13} de Jong, R.: \emph{Second variation of Zhang's $\lambda$-invariant on the moduli space of curves}. American Jnl. Math. {\bf 135}, 275--290 (2013).
\bibitem{dJo14a} de Jong, R.: \emph{Asymptotic behavior of the Kawazumi--Zhang invariant for degenerating Riemann surfaces}. Asian Jnl. Math. {\bf 18}, 507--524 (2014).
\bibitem{dJo16} de Jong, R.: \emph{Torus bundles and 2-forms on the universal family of Riemann surfaces}. In: A. Papadopoulos (ed.), Handbook of Teichmüller theory. Volume VI., p. 195--227. EMS Publishing House (2016).
\bibitem{JK06} Jorgenson, J., Kramer, J.: \emph{Bounds on canonical {G}reen's functions}. Compos. Math. {\bf 142}, no. 3, 679--700 (2006).
\bibitem{JK09} Jorgenson, J., Kramer, J.: \emph{Bounds on Faltings's delta function through covers.} Ann. of Math. (2) {\bf 170}, no. 1, 1--43 (2009).
\bibitem{Jor90} Jorgenson, J.: \emph{Asymptotic behavior of Faltings's delta function}. Duke Math. J. {\bf 61}, 221--254 (1990).
\bibitem{vKa14} von Känel, R.: \emph{On Szpiro's Discriminant Conjecture}. Int. Math. Res. Notices no. {\bf 16}, 4457--4491 (2014).
\bibitem{vKa14b} von Känel, R.: \emph{Integral points on moduli schemes of elliptic curves}. Trans. London Math. Soc. {\bf 1}, no. 1, 85--115 (2014).
\bibitem{Kaw08} Kawazumi, N.: \emph{Johnson's homomorphisms and the Arakelov--Green function}. Preprint, arXiv:0801.4218 (2008). Accessed 09 December 2014
\bibitem{Loc94} Lockhart, P.: \emph{On the discriminant of a hyperelliptic curve}. Trans. Amer. Soc. {\bf 342}, no. 2, 729--752 (1994).
\bibitem{Mum83} Mumford, D.: \emph{Tata Lectures on Theta I}. Progr. in Math. vol. {\bf 28}, Birkhäuser Verlag (1983).
\bibitem{Mum84} Mumford, D.: \emph{Tata Lectures on Theta II}. Progr. in Math. vol. {\bf 43}, Birkhäuser Verlag (1984).
\bibitem{Pio16} Pioline, B.: \emph{A Theta lift representation for the Kawazumi--Zhang and Faltings invariants of genus-two Riemann surfaces}. J. Number Theory {\bf 163}, 520--541 (2016).
\bibitem{Pow78} Powell, J.: \emph{Two theorems on the mapping class group of a surface}. Proc. Amer. Math. Soc. {\bf 68}, no. 3, 347--350 (1978).
\bibitem{Ros51} Rosenhain, G: \emph{M\'emoire sur les fonctions de deux variables et \`a quatre p\'eriodes qui sont les inverses des int\'egrales ultra-elliptiques de la premi\`ere classe}. M\'emoires des savants \'etrangers {\bf 11}, 362--468 (1851).
\bibitem{Sch07} Schlichenmaier, M.: \emph{An introduction to {R}iemann surfaces, algebraic curves and moduli spaces}. Theoretical and Mathematical Physics, Springer, Berlin (2007).
\bibitem{Szp85b} Szpiro, L.: \emph{Degr\'es, intersections, hauteurs}. Ast\'erisque No. {\bf 127}, 11--28 (1985).
\bibitem{Wen91} Wentworth, R.: \emph{The asymptotics of the Arakelov--Green's function and Faltings' delta invariant}. Comm. Math. Phys. {\bf 137}, 427--459 (1991).
\bibitem{Zha96} Zhang, S.: \emph{Heights and reductions of semi-stable varieties}. Compositio Math. {\bf 104}, no. 1, 77--105 (1996).
\bibitem{Zha10} Zhang, S.: \emph{Gross-Schoen cycles and dualising sheaves}. Invent. Math. {\bf 179}, 1--73 (2010).
\end{thebibliography}
\end{document}